\documentclass[11pt]{article}

\usepackage[T1]{fontenc}
\usepackage[utf8]{inputenc}

\usepackage[a4paper,margin=1in]{geometry}

\usepackage{amsmath,amssymb,amsthm,mathtools}
\usepackage{bm}

\usepackage{graphicx}
\usepackage[colorlinks=true,linkcolor=blue,citecolor=blue,urlcolor=blue]{hyperref}
\usepackage{enumitem}
\usepackage{cite}
\usepackage{subfig}
\usepackage{float}
\usepackage[numbers]{natbib}
\usepackage[labelfont=bf]{caption}

\newtheorem{theorem}{Theorem}[section]
\newtheorem{lemma}[theorem]{Lemma}

\theoremstyle{definition}

\newtheorem{problem}[theorem]{Problem}
\newtheorem{assumption}[theorem]{Assumption}

\title{Numerical Analysis of a Variable-Order Time-Fractional Incompressible Magnetohydrodynamics System}
\author{Abdumauvlen Berdyshev${}^{1,2}$, Dossan Baigereyev${}^{1,3}$, Aibek Bakishev${}^{1,3}$,\\Nurlana Alimbekova${}^{1,3}$, Talgat Farkhadov${}^{1}$\\
\small ~\\
\small ${}^{1}$Institute of Information and Computational Technologies, Almaty, Kazakhstan\\
\small ${}^{2}$Abai Kazakh National Pedagogical University, Almaty, Kazakhstan\\
\small ${}^{3}$Sarsen Amanzholov East Kazakhstan University, Ust-Kamenogorsk, Kazakhstan\\}
\date{\today}

\begin{document}

\maketitle

\begin{abstract}
We consider an incompressible magnetohydrodynamics (MHD) model in which the classical first-order time derivatives in the momentum and magnetic induction equations are replaced by variable-order Caputo time-fractional derivatives. This formulation allows the memory effect to vary during the evolution and represents a time-fractional generalization of the incompressible MHD system with nonstationary memory. To approximate the problem, we use a fully discrete scheme combining a finite element discretization in space with an L1-type approximation of the variable-order Caputo operators in time. For this discretization, we establish a discrete stability estimate and also derive an auxiliary corrected discrete energy estimate for the fully discrete solution. Convergence is proved by showing that the kernels generated by the variable-order L1 approximation satisfy the assumptions of an abstract discrete fractional Gr\"{o}nwall theorem, which is then applied to the coupled MHD system. The numerical study consists of four parts. First, representative order profiles are used to examine temporal convergence. Second, consistency with the classical incompressible MHD equations is studied as the fractional orders approach one, using norms of solution differences and deviations in kinetic and magnetic energies. Third, the influence of the variable-order fractional terms on nonlinear evolution is investigated through the periodic divergence-free vortex benchmark, with comparisons based on energy and enstrophy histories, divergence errors, Reynolds-number dependence, and time-integrated diagnostics. Fourth, parameter-space maps show how the parameters defining the variable orders affect global indicators. The results show that the variable orders can noticeably affect the evolution of the energy, enstrophy, and current enstrophy even when the Reynolds number is fixed.
\end{abstract}

\section{Introduction}

Magnetohydrodynamics (MHD) describes the interaction between magnetic fields 
and electrically conducting fluids, including plasmas, liquid metals, electrolytes, 
and conducting fluids arising in geophysical and astrophysical contexts. 
MHD models are widely used in plasma physics, geophysics, metallurgy, 
nuclear engineering, liquid-metal heat transfer, and other areas of engineering 
science \cite{Moreau1990,Davidson2001,Roberts2013}. These models are used to study 
the coupling between fluid motion and electromagnetic fields in applications 
such as liquid-metal flows, plasma control, and natural magnetic phenomena.

On the other hand, many physical processes, including diffusion, heat
transfer, and viscoelastic flow, exhibit memory and hereditary effects
that cannot always be adequately described by classical differential
models. Fractional calculus, based on derivatives and integrals of
non-integer order, provides a widely used tool for modeling such effects
\cite{Mainardi2010a,Metzler2000,Tarasov2010,Das2011}.
Fractional derivatives have been successfully applied to describe
anomalous diffusion, subdiffusion, and superdiffusion, as well as
nonlocal effects in charge transport, heat transfer, and deformation of
viscoelastic materials
\cite{Magin2004,Tarasov2010,Das2011,Madiyarov2025,Abdiramanov2024,Shiyapov2026,Baishemirov2024}.
Related numerical analyses for Caputo-type fractional problems can be found, 
for example, in \cite{Alikhanov2015,Liao_2019,Lin2007,Altybay2026}.
Recently, fractional differential models with variable order, where the
order may depend on time, space, or even the unknown solution, have
attracted increasing attention
\cite{Sun2019,Alimbekova2024b}. Such models allow the strength of memory
or anomalous transport to vary during the evolution of the process and
therefore provide a useful way to describe systems with changing
dynamical regimes.

The combination of MHD flow modeling and fractional calculus has motivated
a number of studies involving fractional operators in reduced MHD-related
fluid and heat-transfer models, including viscous-fluid models with
Newtonian heating, generalized Oldroyd--B flows, and Maxwell nanofluid
flows \cite{Ali2024,Arif2023,Asjad2023a}. These works demonstrate the use
of fractional operators in MHD-related flow models, although they differ
substantially from the variable-order time-fractional incompressible MHD
system considered here.

Numerical methods for classical MHD equations have been extensively developed,
including finite difference, finite volume, finite element, spectral element,
and structure-preserving approaches
\cite{Bader2025,Fambri2025,Shakeri2011}. Fractional MHD-related flow and
heat-transfer models have also been studied numerically, for example using
finite difference, L1-type, and spectral-collocation techniques
\cite{Liu2024a,Liu2024,Shen2018,Zhao2016}. These studies mostly concern
constant-order, space-fractional, distributed-order, or reduced flow
configurations. By comparison, variable-order time-fractional incompressible MHD systems appear
to be much less explored.

The aim of this paper is to develop, analyze, and test a fully discrete numerical scheme 
for incompressible magnetohydrodynamics with variable-order Caputo time-fractional derivatives. 
Particular attention is given to the influence of time-dependent fractional orders on transient 
MHD behavior and to the recovery of the classical incompressible MHD model in the limit 
as the orders approach unity. The numerical study is therefore used both to assess the discretization and
to examine how different order profiles affect the evolution of standard
diagnostic quantities.

The present work differs from many existing fractional MHD studies.
Most time-fractional MHD models use constant-order derivatives and are considered
in reduced flow configurations \cite{Rehman2023,Arif2023,Abbas2024}.
Works involving variable- or distributed-order fractional operators are less common
(cf. \cite{Khan2023,Li2024,Qiao2024}). To our knowledge, these works do not provide
a fully discrete numerical analysis for the variable-order time-fractional
incompressible MHD system considered here. 
In particular, it remains insufficiently understood
how the temporal profile of the fractional order influences standard MHD diagnostics
and whether a numerical discretization consistently recovers the classical MHD
behavior as the fractional orders approach unity. From the analytical point of view,
the variable-order formulation introduces an additional difficulty because the
discrete memory kernels depend on the current time level and must be treated within
a coupled incompressible MHD system. Consequently, the convergence proof requires
more than a direct extension of the constant-order analysis: one must verify that
the kernels generated by the variable-order L1 discretization satisfy the assumptions
of the abstract discrete fractional Gr\"{o}nwall theorem \cite{Liao_2019} used in
the error analysis.

In view of these considerations, the main contributions of this paper are as follows.

(1) We formulate a variable-order time-fractional generalization of the incompressible MHD 
equations by replacing the classical first-order time derivatives in the momentum and 
magnetic induction equations with Caputo derivatives of time-dependent orders 
$\alpha\left(t\right),\beta\left(t\right)\in\left(0,1\right)$, respectively. 
This formulation allows the temporal memory effect to vary during the evolution 
and provides a model for MHD dynamics with nonstationary relaxation behavior.

(2) We construct a fully discrete numerical method for the coupled incompressible MHD 
system by combining a finite element approximation in space with an L1-type discretization 
in time for the variable-order Caputo operators.

(3) For the fully discrete scheme, we establish stability and derive a convergence estimate. 
The stability analysis uses a discrete energy inequality for the variable-order 
fractional derivative together with complementary kernel bounds, and is complemented 
by a corrected discrete energy estimate for the fully discrete solution. In the convergence analysis, 
we verify that the discrete kernels generated by the variable-order L1 approximation satisfy 
the assumptions required to apply the abstract discrete fractional Gr\"{o}nwall theorem 
of \cite{Liao_2019} to the finite element--L1 scheme.

(4) We perform representative numerical tests to assess the method and investigate how 
time-dependent fractional orders affect the solution behavior. These include convergence 
studies for selected order profiles, a classical-limit test showing consistency with the 
standard incompressible MHD equations as $\alpha\left(t\right),\beta\left(t\right)\to 1$, 
and simulations of the periodic divergence-free vortex under several variable-order scenarios.

(5) We further examine the sensitivity of the model to the temporal memory profile through 
parameter-space maps for the order-function parameters, using global diagnostics to quantify 
how the evolution depends on the choice of the variable order.

The rest of the paper is organized as follows. Section \ref{sec:Materials} introduces 
the variable-order time-fractional incompressible magnetohydrodynamics model, 
summarizes the numerical method used for its discretization and implementation,
and presents the stability and convergence analysis.
Section \ref{sec:Results} presents the computational results and analyzes how the 
variable-order Caputo time derivatives influence the flow dynamics across the considered 
numerical experiments, including verification tests and the variable-order fractional 
periodic divergence-free vortex benchmark. Section~\ref{sec:Discussion} discusses and
summarizes the main findings, and Section~\ref{sec:Conclusions} states the concluding
remarks and outlines directions for future work.

\section{Problem Formulation and Numerical Method}
\label{sec:Materials}

\subsection{Problem Formulation}

In this paper, we consider a variable-order time-fractional generalization
of the incompressible magnetohydrodynamic equations in the domain
$\Omega\times\left(0,T\right]$, where $\Omega$ is a bounded
subdomain of $\mathbb{R}^{2}$, and $T>0$ is a finite time:
\begin{align}
 & \rho_{0}\theta^{\alpha\left(t\right)-1}\,\,{}^{C}D_{t}^{\alpha\left(t\right)}\mathbf{u}+\rho_{0}\left(\mathbf{u}\cdot\nabla\right)\mathbf{u}=-\nabla P+\mu\nabla^{2}\mathbf{u}+\frac{1}{\mu_{0}}\left(\nabla\times\mathbf{B}\right)\times\mathbf{B}+\mathbf{f} &  & \text{in}\;\Omega\times\left(0,T\right],\label{eq:given_u1}\\
 & \theta^{\beta\left(t\right)-1}\,\,{}^{C}D_{t}^{\beta\left(t\right)}\mathbf{B}=\nabla\times\left(\mathbf{u}\times\mathbf{B}\right)+\eta\nabla^{2}\mathbf{B}+\mathbf{g} &  & \text{in}\;\Omega\times\left(0,T\right],\label{eq:given_B1}\\
 & \nabla\cdot\mathbf{u}=0,\qquad\nabla\cdot\mathbf{B}=0 &  & \text{in}\;\Omega\times\left(0,T\right],\label{eq:eq:eq:given_B2}
\end{align}
obtained from the classical incompressible MHD system (cf. \cite{He2018}) by replacing
the first-order time derivatives with
Caputo derivatives of orders $\alpha\left(t\right)$, $\beta\left(t\right)\in\left(0,1\right)$
that depend on the observation time $t$,
\begin{equation}
(\,^{C}D_{t}^{\nu\left(t\right)}f)\left(t\right)=\frac{1}{\Gamma\left(1-\nu\left(t\right)\right)}\int_{0}^{t}\frac{\partial_{s}f\left(s\right)}{\left(t-s\right)^{\nu\left(t\right)}}ds,\qquad0<\nu\left(t\right)<1,\quad t\in\left[0,T\right],\label{eq:vo_caputo}
\end{equation}
where $\nu\in\left\{ \alpha,\beta\right\} $. 
Since $^{C}D_{t}^{\nu\left(t\right)}$
carries the physical dimension {$\left[\mathrm{s}\right]^{-\nu\left(t\right)}$}, we multiply it by $\theta^{\nu\left(t\right)-1}$
so that the fractional term has the same dimension as a first-order time derivative (cf. \cite{Vaz2025}). 
When $\nu\left(t\right)\to 1$, this prefactor tends to 1, and the fractional derivative formally approaches the classical first-order derivative.
Thus, in the limiting case $\alpha\left(t\right),\beta\left(t\right)\to 1$, the system (\ref{eq:given_u1})--(\ref{eq:eq:eq:given_B2}) reduces to the standard
incompressible MHD equations \cite{He2018}.

In Eqs. (\ref{eq:given_u1})--(\ref{eq:eq:eq:given_B2}), $\mathbf{u}$
denotes the velocity, $\mathbf{B}$ is the magnetic field, $P$ is
the pressure, $\rho_{0}$ is the density, $\mu$ is the dynamic viscosity,
$\eta$ is the magnetic diffusion, $\mu_{0}$ is magnetic force scaling parameter,
$\mathbf{f}\left(\mathbf{x},t\right)$ and $\mathbf{g}\left(\mathbf{x},t\right)$
are source terms, and $\mathbf{x}\in\mathbb{R}^{2}$. Eqs. (\ref{eq:given_u1})--(\ref{eq:eq:eq:given_B2})
are supplemented with initial conditions 
\[
\mathbf{u}=\mathbf{u}_{0},\qquad \mathbf{B}=\mathbf{B}_{0}\quad\text{in}\quad\overline{\Omega}\times\left\{ 0\right\}, 
\]
satisfying the conditions $\nabla\cdot\mathbf{u}_{0}=0$, $\nabla\cdot\mathbf{B}_{0}=0$,
and with one of the following boundary conditions:

(i) homogeneous Dirichlet boundary conditions,
\[
\mathbf{u}=\mathbf{0},\qquad \mathbf{B}=\mathbf{0}
\qquad \text{on}\quad\partial\Omega\times\left(0,T\right];
\]

(ii) periodic boundary conditions.

The curl and cross product operations are understood through the standard embedding
of planar vector fields into $\mathbb{R}^3$, namely
$\mathbf{u}=\left(u_{1},u_{2},0\right)$ and $\mathbf{B}=\left(B_{1},B_{2},0\right)$, 
so that, for example,
$\nabla\times\mathbf{B}=\left(0,0,\partial_{x_{1}}B_{2}-\partial_{x_{2}}B_{1}\right)$.

We introduce characteristic scales $S_{0}$ (length), $U_{0}$ (velocity),
$T_{0}=\theta=S_{0}/U_{0}$ (time), $B_{0}=U_{0}\sqrt{\rho_{0}\mu_{0}}$
(magnetic field), and $p_{0}=\rho_{0} U_{0}^{2}$ (pressure), and define
dimensionless variables $\mathbf{x}^{*}=\frac{\mathbf{x}}{S_{0}}$, $t^{*}=\frac{t}{T_{0}}$,
$\mathbf{u}^{*}=\frac{\mathbf{u}}{U_{0}}$, $P^{*}=\frac{P}{p_{0}}$, and $\mathbf{B}^{*}=\frac{\mathbf{B}}{B_{0}}$.

Using the scaling property of the Caputo derivative,
\begin{align*}
 & (\,^{C}D_{t}^{\alpha\left(t\right)}\mathbf{u})\left(t\right)=U_{0}T_{0}^{-\alpha\left(T_{0}t^{*}\right)}(\,^{C}D_{t^{*}}^{\alpha\left(T_{0}t^{*}\right)}\mathbf{u}^{*})\left(t^{*}\right),\\
 & (\,^{C}D_{t}^{\beta\left(t\right)}\mathbf{B})\left(t\right)=B_{0}T_{0}^{-\beta\left(T_{0}t^{*}\right)}(\,^{C}D_{t^{*}}^{\beta\left(T_{0}t^{*}\right)}\mathbf{B}^{*})\left(t^{*}\right),
\end{align*}
we obtain the dimensionless system
\begin{align}
 & ^{C}D_{t^{*}}^{\alpha^{*}\left(t^{*}\right)}\mathbf{u}^{*}+\left(\mathbf{u}^{*}\cdot\nabla^{*}\right)\mathbf{u}^{*}=-\nabla^{*}P^{*}+\frac{1}{\mathrm{Re}}\nabla^{*2}\mathbf{u}^{*}+\left(\nabla^{*}\times\mathbf{B}^{*}\right)\times\mathbf{B}^{*}+\mathbf{f}^{*},\label{eq:nondim_u1}\\
 & ^{C}D_{t^{*}}^{\beta^{*}\left(t^{*}\right)}\mathbf{B}^{*}=\nabla^{*}\times\left(\mathbf{u}^{*}\times\mathbf{B}^{*}\right)+\frac{1}{\mathrm{Rm}}\nabla^{*2}\mathbf{B}^{*}+\mathbf{g}^{*},\label{eq:nondim_B1}\\
 & \nabla^{*}\cdot\mathbf{u}^{*}=0,\qquad\nabla^{*}\cdot\mathbf{B}^{*}=0,\label{eq:nondim_B2}
\end{align}
where $\alpha^{*}\left(t^{*}\right)=\alpha\left(T_{0}t^{*}\right)$, $\beta^{*}\left(t^{*}\right)=\beta\left(T_{0}t^{*}\right)$,
$\mathrm{Re}=\frac{\rho_{0} U_{0}S_{0}}{\mu}$, 
$\mathrm{Rm}=\frac{U_{0}S_{0}}{\eta}$,
$\mathbf{f}^{*}=\frac{S_{0}}{\rho_{0} U_{0}^{2}}\mathbf{f}$, and 
$\mathbf{g}^{*}=\frac{T_{0}}{U_{0}\sqrt{\rho_{0} \mu_{0}}}\mathbf{g}$.
For convenience, we henceforth omit the asterisks from all dimensionless variables and operators,
and write $\alpha\left(t\right)$, $\beta\left(t\right)$ for the dimensionless order functions.

We now derive the variational (weak) formulation associated with the
non-dimensional system \eqref{eq:nondim_u1}--\eqref{eq:nondim_B2}.
Throughout the paper, we use standard notation for Sobolev spaces. 
The usual spaces $H^{s}(\Omega)$ and $W^{k,p}(\Omega)$ are equipped with 
their standard norms. In particular, $\lVert\cdot\rVert$ denotes the norm in $L^{2}(\Omega)$, 
and $\lVert\cdot\rVert_{L^{p}}$ denotes the norm in $L^{p}(\Omega)$.
We introduce the pressure space ${\mathcal{Q}=L^{2}_{0}\left(\Omega\right)=\left\{ q\in L^{2}\left(\Omega\right):\,\,\int_{\Omega}q\,d\mathbf{x}=0\right\} }$.
For the velocity and magnetic field, we introduce the corresponding function spaces according to the boundary conditions under consideration.
In the case of homogeneous Dirichlet boundary conditions, we take $\mathcal{V}=H^{1}_{0}\left(\Omega\right)^{2}$, $\mathcal{C}=H^{1}_{0}\left(\Omega\right)^{2}$.
In the periodic case, we take $\mathcal{V}=H^{1}_{\mathrm{per}}\left(\Omega\right)^{2}$, $\mathcal{C}=H^{1}_{\mathrm{per}}\left(\Omega\right)^{2}$, 
where $H^{1}_{\mathrm{per}}\left(\Omega\right)^{2}$ denotes the corresponding periodic Sobolev space. 

To obtain a convenient form of the pressure and coupling terms, we
employ the standard vector identities
\[
\left(\nabla\times\mathbf{B}\right)\times\mathbf{B}=\left(\mathbf{B}\cdot\nabla\right)\mathbf{B}-\nabla\left(\frac{1}{2}\left|\mathbf{B}\right|^{2}\right),\qquad
\nabla\times\left(\mathbf{u}\times\mathbf{B}\right)=\left(\mathbf{B}\cdot\nabla\right)\mathbf{u}-\left(\mathbf{u}\cdot\nabla\right)\mathbf{B}
\]
under the constraints (\ref{eq:nondim_B2}), and introduce the total pressure $p=P+\frac{1}{2}\left|\mathbf{B}\right|^{2}$.
For $\mathbf{a},\mathbf{b},\mathbf{w}\in H^{1}\left(\Omega\right)^{2}$
we further define the trilinear form
\[
\ell\left(\mathbf{a};\mathbf{b},\mathbf{w}\right)=\left(\left(\mathbf{a}\cdot\nabla\right)\mathbf{b},\mathbf{w}\right)+\frac{1}{2}\left(\left(\nabla\cdot\mathbf{a}\right)\mathbf{b},\mathbf{w}\right),\label{eq:advective_derivative}
\]
where $(\cdot,\cdot)$ denotes the $L^{2}(\Omega)$ inner product.
Under homogeneous Dirichlet or periodic boundary conditions, integration by parts yields the equivalent representation
\[
\ell\left(\mathbf{a};\mathbf{b},\mathbf{w}\right)=\frac{1}{2}\left(\left(\mathbf{a}\cdot\nabla\right)\mathbf{b},\mathbf{w}\right)-\frac{1}{2}\left(\left(\mathbf{a}\cdot\nabla\right)\mathbf{w},\mathbf{b}\right),\label{eq:trilinear}
\]
and therefore
\begin{equation}
\ell\left(\mathbf{a};\mathbf{b},\mathbf{b}\right)=0,\qquad \ell\left(\mathbf{a};\mathbf{b},\mathbf{w}\right)=-\ell\left(\mathbf{a};\mathbf{w},\mathbf{b}\right).\label{eq:trilinear_prop_c}
\end{equation}

Then the weak formulation reads:
\begin{problem}
\label{prob:weak}
Find $(\mathbf{u},p,\mathbf{B})\in \mathcal{V}\times\mathcal{Q}\times\mathcal{C}$ such that, for all
$(\mathbf{v},q,\mathbf{C})\in \mathcal{V}\times\mathcal{Q}\times\mathcal{C}$,
\begin{align*}
 & \left(\,^{C}D^{\alpha\left(t\right)}_{t}\mathbf{u},\mathbf{v}\right)+\frac{1}{\mathrm{Re}}\left(\nabla\mathbf{u},\nabla\mathbf{v}\right)+\ell\left(\mathbf{u};\mathbf{u},\mathbf{v}\right)-\ell\left(\mathbf{B};\mathbf{B},\mathbf{v}\right)-\left(p,\nabla\cdot\mathbf{v}\right)=\left<\mathbf{f},\mathbf{v}\right>,\\
 & \left(\,^{C}D^{\beta\left(t\right)}_{t}\mathbf{B},\mathbf{C}\right)+\frac{1}{\mathrm{Rm}}\left(\nabla\mathbf{B},\nabla\mathbf{C}\right)+\ell\left(\mathbf{u};\mathbf{B},\mathbf{C}\right)-\ell\left(\mathbf{B};\mathbf{u},\mathbf{C}\right)=\left<\mathbf{g},\mathbf{C}\right>,\\
 & \left(\nabla\cdot\mathbf{u},q\right)=0.
\end{align*}
\end{problem}
The magnetic solenoidal condition $\nabla\cdot\mathbf{B}=0$ is understood
as part of the continuous MHD model, but it is not introduced here as a
separate weak equation. Similar $H^{1}$-conforming variational formulations for
the magnetic field have been used in finite element analyses of incompressible
MHD (cf. \cite{He2018}).

\subsection{Numerical Method}

To discretize the problem in time, we introduce the partition $\mathcal{I}_{N}$ of the
time interval $\left[0,T\right]$ by the points $t_{n}=n\tau$, $\tau>0$,
$n=0,1,...,N$, so that $N\tau=T$. Denote $\phi^{n}=\phi\left(\cdot,t_{n}\right)$.
Further, we approximate the variable-order Caputo fractional derivative
(\ref{eq:vo_caputo}) using the formula \cite{Liu2022}
\begin{equation}
(\,^{C}D_{t}^{\nu\left(t_{n}\right)}\phi)\left(t_{n}\right)
=\Delta_{\tau}^{\nu\left(t_{n}\right)}\phi^{n}+O\left(\tau\right),\qquad
0<\nu\left(t_{n}\right)<1
\label{eq:frac_replacement}
\end{equation}
with
\begin{equation}
\Delta_{\tau}^{\nu\left(t_{n}\right)}\phi^{n}=\frac{1}{\tau}\left(b^{\left(\nu\right)}_{n,n}\phi^{n}-\sum^{n-1}_{k=1}\left(b^{\left(\nu\right)}_{n,k+1}-b^{\left(\nu\right)}_{n,k}\right)\phi^{k}-b^{\left(\nu\right)}_{n,1}\phi^{0}\right)=\frac{1}{\tau}\sum^{n}_{k=1}b^{\left(\nu\right)}_{n,k}\left(\phi^{k}-\phi^{k-1}\right),\label{eq:discrete_caputo}
\end{equation}
where 
\begin{equation}
b^{\left(\nu\right)}_{n,k}=\frac{1}{\Gamma\left(2-\nu\left(t_{n}\right)\right)}\left(\left(t_{n}-t_{k-1}\right)^{1-\nu\left(t_{n}\right)}-\left(t_{n}-t_{k}\right)^{1-\nu\left(t_{n}\right)}\right).
\label{eq:coeffs_b}
\end{equation}

Using (\ref{eq:frac_replacement}), we rewrite Problem \ref{prob:weak} as follows.
\begin{problem}
Let the solutions $\mathbf{u}^{i}\in \mathcal{V}$, $\mathbf{B}^{i}\in \mathcal{C}$ be known at time levels $i=0,...,n-1$.
Find $\left(\mathbf{u}^{n},p^{n},\mathbf{B}^{n}\right)\in \mathcal{V}\times\mathcal{Q}\times\mathcal{C}$,
satisfying the following identities for all $\left(\mathbf{v},q,\mathbf{C}\right)\in \mathcal{V}\times\mathcal{Q}\times\mathcal{C}$:
\begin{align}
 & \left(\Delta^{\alpha\left(t_{n}\right)}_{\tau}\mathbf{u}^{n},\mathbf{v}\right)+\frac{1}{\mathrm{Re}}\left(\nabla\mathbf{u}^{n},\nabla\mathbf{v}\right)+\ell\left(\mathbf{u}^{n};\mathbf{u}^{n},\mathbf{v}\right)-\ell\left(\mathbf{B}^{n};\mathbf{B}^{n},\mathbf{v}\right)\nonumber\\
 & \qquad\qquad\qquad\qquad\qquad -\left(p^{n},\nabla\cdot\mathbf{v}\right)=\left<\mathbf{f}^{n},\mathbf{v}\right>-\left(\mathbf{r}^{n}_{1},\mathbf{v}\right),\displaybreak[0]\label{eq:semi_u1}\\
 & \left(\Delta^{\beta\left(t_{n}\right)}_{\tau}\mathbf{B}^{n},\mathbf{C}\right)+\frac{1}{\mathrm{Rm}}\left(\nabla\mathbf{B}^{n},\nabla\mathbf{C}\right)+\ell\left(\mathbf{u}^{n};\mathbf{B}^{n},\mathbf{C}\right)\nonumber\\
 & \qquad\qquad\qquad\qquad\qquad -\ell\left(\mathbf{B}^{n};\mathbf{u}^{n},\mathbf{C}\right)=\left<\mathbf{g}^{n},\mathbf{C}\right>-\left(\mathbf{r}^{n}_{2},\mathbf{C}\right),\displaybreak[0]\label{eq:semi_B1}\\
 & \left(\nabla\cdot\mathbf{u}^{n},q\right)=0,\label{eq:semi_u2}
\end{align}
where
\begin{equation}
\mathbf{r}^{n}_{1}=(\,^{C}D_{t}^{\alpha\left(t_{n}\right)}\mathbf{u})\left(t_{n}\right)-\Delta^{\alpha\left(t_{n}\right)}_{\tau}\mathbf{u}^{n},\qquad
\mathbf{r}^{n}_{2}=(\,^{C}D_{t}^{\beta\left(t_{n}\right)}\mathbf{B})\left(t_{n}\right)-\Delta^{\beta\left(t_{n}\right)}_{\tau}\mathbf{B}^{n}.
\label{eq:semi_r}
\end{equation}
\end{problem}

Let us further introduce a triangulation $\mathfrak{T}_{h}$ of $\Omega$ 
with the discretization parameter $h>0$,
and let $\mathcal{V}_{h}\subset \mathcal{V}$, $\mathcal{Q}_{h}\subset \mathcal{Q}$, and $\mathcal{C}_{h}\subset \mathcal{C}$ be
conforming finite element spaces. We define the numerical scheme as follows.
\begin{problem}
\label{prob:fully}Suppose that the discrete solutions $\mathbf{u}_{h}^{i}\in \mathcal{V}_{h}$,
$\mathbf{B}_{h}^{i}\in \mathcal{C}_{h}$ are known at time levels $i=0,...,n-1$.
Find $\left(\mathbf{u}_{h}^{n},p_{h}^{n},\mathbf{B}_{h}^{n}\right)\in \mathcal{V}_{h}\times\mathcal{Q}_{h}\times\mathcal{C}_{h}$
satisfying the following identities for all $\left(\mathbf{v}_{h},q_{h},\mathbf{C}_{h}\right)\in \mathcal{V}_{h}\times\mathcal{Q}_{h}\times\mathcal{C}_{h}$:
\begin{align}
 & \left(\Delta^{\alpha\left(t_{n}\right)}_{\tau}\mathbf{u}^{n}_{h},\mathbf{v}_{h}\right)+\frac{1}{\mathrm{Re}}\left(\nabla\mathbf{u}^{n}_{h},\nabla\mathbf{v}_{h}\right)+\ell\left(\mathbf{u}^{n}_{h};\mathbf{u}^{n}_{h},\mathbf{v}_{h}\right)-\ell\left(\mathbf{B}^{n}_{h};\mathbf{B}^{n}_{h},\mathbf{v}_{h}\right)\nonumber \\
 & \qquad\qquad\qquad\qquad\qquad-\left(p^{n}_{h},\nabla\cdot\mathbf{v}_{h}\right)+\zeta\left(\nabla\cdot\mathbf{u}^{n}_{h},\nabla\cdot\mathbf{v}_{h}\right)=\left<\mathbf{f}^{n},\mathbf{v}_{h}\right>,\label{eq:fully_u1}\\
 & \left(\Delta^{\beta\left(t_{n}\right)}_{\tau}\mathbf{B}^{n}_{h},\mathbf{C}_{h}\right)+\frac{1}{\mathrm{Rm}}\left(\nabla\mathbf{B}^{n}_{h},\nabla\mathbf{C}_{h}\right)+\ell\left(\mathbf{u}^{n}_{h};\mathbf{B}^{n}_{h},\mathbf{C}_{h}\right)\nonumber \\
 & \qquad\qquad\qquad\qquad\qquad-\ell\left(\mathbf{B}^{n}_{h};\mathbf{u}^{n}_{h},\mathbf{C}_{h}\right)+\chi\left(\nabla\cdot\mathbf{B}^{n}_{h},\nabla\cdot\mathbf{C}_{h}\right)=\left<\mathbf{g}^{n},\mathbf{C}_{h}\right>,\label{eq:fully_B1}\\
 & \left(\nabla\cdot\mathbf{u}^{n}_{h},q_{h}\right)=0.\label{eq:fully_u2}
\end{align}
\end{problem}

Since the velocity and magnetic field are approximated in $H^{1}$-conforming
finite element spaces, the constraints (\ref{eq:nondim_B2}) are not enforced pointwise
at the discrete level. To improve the control of these constraints,
we augment the momentum and magnetic equations by the consistent stabilization
terms $\zeta\left(\nabla\cdot\mathbf{u}^{n}_{h},\nabla\cdot\mathbf{v}_{h}\right)$
and $\chi\left(\nabla\cdot\mathbf{B}^{n}_{h},\nabla\cdot\mathbf{C}_{h}\right)$,
where $\zeta\geq 0$ and $\chi\geq 0$ are stabilization parameters. These terms
vanish for divergence-free exact solutions and therefore do not affect
the consistency of the method. Their role is to penalize discrete
divergence errors and improve mass conservation and magnetic solenoidality
in the numerical approximation.

\subsection{Stability and Convergence Analysis}

In this section, we establish stability and convergence of the fully discrete finite element--L1 scheme. 
The main analytical difficulty arises from the variable-order L1 discretization of the Caputo derivatives. 
In contrast to the constant-order case, the discrete coefficients depend on the current time level $t_{n}$, 
so the resulting memory kernels vary with $n$ and cannot be handled directly by the standard fixed-order analysis. 
In the present problem, this issue must be treated within the nonlinear coupled structure of the incompressible 
MHD system. 
We first derive a corrected discrete energy inequality, which accounts for the
contribution of the nonstationary memory kernels to the discrete energy balance.
The stability estimate is then obtained using a discrete energy inequality for
the variable-order fractional derivative together with complementary kernel
bounds.
For the convergence proof, we additionally verify that the kernels 
generated by the variable-order L1 approximation satisfy the assumptions of the abstract discrete 
fractional Gr\"{o}nwall theorem \cite{Liao_2019} used to bound the fully discrete error sequence.

\begin{assumption}
\label{assu:regularity}Let $l$ and $m$ be the exponents determined by the approximation properties 
of the finite element spaces for the velocity and magnetic field, respectively.
Assume that the exact solution satisfies
\[
\mathbf{u}\in L^{\infty}(0,T;H^{l+1}(\Omega)^2)\cap L^{\infty}(0,T;W^{1,\infty}(\Omega)^2),
\qquad
\partial_t\mathbf{u}\in L^{\infty}(0,T;H^{l+1}(\Omega)^2),
\]
\[
\mathbf{B}\in L^{\infty}(0,T;H^{m+1}(\Omega)^2)\cap L^{\infty}(0,T;W^{1,\infty}(\Omega)^2),
\qquad
\partial_t\mathbf{B}\in L^{\infty}(0,T;H^{m+1}(\Omega)^2),
\]
and that $\mathbf{u}$ and $\mathbf{B}$ satisfy the prescribed boundary conditions.
In the periodic case, we additionally assume that the velocity and magnetic field have zero spatial mean.
\end{assumption}

\begin{assumption}
\label{assu:rhs}
Assume that $\mathbf{f},\mathbf{g}\in L^{2}\left(0,T;H^{-1}\left(\Omega\right)^{2}\right)$.
\end{assumption}

\begin{assumption}
\label{assu:orders_bounds}Assume that $\alpha,\beta\in C\left[0,T\right]$, and
\[
0<\alpha_{*}\leq\alpha\left(t\right)\leq\alpha^{*}<1,\qquad0<\beta_{*}\leq\beta\left(t\right)\leq\beta^{*}<1.
\]
\end{assumption}

We will repeatedly use several standard inequalities in the analysis. 
In the case of homogeneous Dirichlet boundary conditions, the Poincar\'{e} inequality
holds on the discrete velocity and magnetic field spaces since they are conforming
subspaces of $H_{0}^{1}\left(\Omega\right)^{2}$. In the periodic case, the corresponding estimate is
used on the subspace of periodic functions with zero spatial mean. Therefore, in both
cases there exists a constant $C>0$, independent of $h$, such that
\[
\lVert\mathbf v_h\rVert \le C\lVert\nabla \mathbf v_h\rVert,\qquad
\lVert\mathbf C_h\rVert \le C\lVert\nabla \mathbf C_h\rVert
\]
for all $\mathbf v_{h}\in \mathcal{V}_{h}$ and $\mathbf C_{h}\in \mathcal{C}_{h}$.

Under the same assumptions, we use 
the two-dimensional Ladyzhenskaya inequality
\[
\lVert\mathbf{v}\rVert_{L^{4}} \le C\lVert\mathbf{v}\rVert^{1/2}\lVert\nabla \mathbf{v}\rVert^{1/2}
\]
for $\mathbf{v}\in H_{0}^{1}(\Omega)^{2}$ in the homogeneous Dirichlet case
and for $\mathbf{v}\in H_{\mathrm{per}}^{1}(\Omega)^{2}$ with zero spatial mean in the periodic case.

We also use the inverse inequality for discrete functions
\[
\lVert\mathbf{v}_h\rVert_{L^{4}} \le C h^{-1/2}\lVert\mathbf{v}_h\rVert,
\qquad
\lVert\nabla \mathbf{v}_h\rVert \le C h^{-1}\lVert\mathbf{v}_h\rVert.
\]

We first establish some basic properties of the coefficients $b_{n,k}^{\left(\nu\right)}$ in (\ref{eq:coeffs_b}).

\begin{lemma}
\label{lem:prop_b}
Let $0<\nu_{*}\leq \nu\left(t\right)\leq\nu^{*}<1$, $t\in\left[0,T\right]$.
Then the coefficients $b^{\left(\nu\right)}_{n,k}$
have the following properties:
\begin{align}
 & b^{\left(\nu\right)}_{n,n}>b^{\left(\nu\right)}_{n,n-1}>...>b^{\left(\nu\right)}_{n,1}>0,\label{eq:prop_b1}\\
 & \sum^{n}_{k=1}b^{\left(\nu\right)}_{n,k}=\frac{t^{1-\nu\left(t_{n}\right)}_{n}}{\Gamma\left(2-\nu\left(t_{n}\right)\right)},\qquad t_{n}\in\mathcal{I}_{N}, \label{eq:prop_b3}\\
 & {\displaystyle \sum^{n}_{k=1}b^{\left(\nu\right)}_{k,1}\leq\gamma,\qquad\gamma=\left\{ \begin{array}{ll}
\frac{T^{1-\nu^{*}}}{1-\nu^{*}}, & T\leq1,\\
\frac{1}{1-\nu^{*}}+\frac{T^{1-\nu_{*}}-1}{1-\nu_{*}}, & T>1.
\end{array}\right.}\label{eq:prop_b2}
\end{align}
\end{lemma}
\begin{proof}
Eq. (\ref{eq:prop_b3}) is verified directly. Let $\nu_{n}=\nu\left(t_{n}\right)$, $t_{n}\in\mathcal{I}_{N}$.
Since the function $f\left(x\right)=x^{1-\nu_{n}}$ is increasing
and concave on $\left(0,\infty\right)$, the increments $f\left(m+1\right)-f\left(m\right)=\left(m+1\right)^{1-\nu_{n}}-m^{1-\nu_{n}}$, $m\geq 0$,
are positive and decrease as $m$ increases. 
Taking $m=n-k$, and noting that $b_{n,k}^{(\nu)}$ is obtained from these increments by multiplication by the positive factor
${\displaystyle \frac{\tau^{1-\nu_n}}{\Gamma\left(2-\nu_n\right)}}$,
we obtain (\ref{eq:prop_b1}).

Finally, by the definition of $b_{k,1}^{\left(\nu\right)}$, for each $k\geq 1$,
\[
b^{\left(\nu\right)}_{k,1}=\frac{t^{1-\nu_{k}}_{k}-t^{1-\nu_{k}}_{k-1}}{\Gamma\left(2-\nu_{k}\right)}=\frac{1}{\Gamma\left(1-\nu_{k}\right)}\int^{t_{k}}_{t_{k-1}}s^{-\nu_{k}}ds
\leq \int^{t_{k}}_{t_{k-1}}s^{-\nu_{k}}ds,
\]
since $\Gamma\left(1-\nu_{k}\right)>1$ as $1-\nu_{k}\in\left(0,1\right)$.

If $T\leq 1$, then $0<s\leq t_{n}\leq T\leq 1$. 
Since $s^{-\alpha}$ is increasing with respect to $\alpha$ for $0<s\leq 1$, we have $s^{-\nu_{k}}\leq s^{-\nu^{*}}$. 
Therefore,
\[
\sum^{n}_{k=1}b^{\left(\nu\right)}_{k,1}\leq\int^{t_{n}}_{0}s^{-\nu^{*}}ds=\frac{t^{1-\nu^{*}}_{n}}{1-\nu^{*}}\leq\frac{T^{1-\nu^{*}}}{1-\nu^{*}}.
\]

Now suppose $T>1$. If $t_{n}\leq 1$, then the same argument gives
\[
\sum^{n}_{k=1}b^{\left(\nu\right)}_{k,1}\leq\frac{1}{1-\nu^{*}}
\leq\frac{1}{1-\nu^{*}}+\frac{T^{1-\nu_{*}}-1}{1-\nu_{*}}.
\]

If $t_{n}>1$, then on $\left(0,1\right]$, we use $s^{-\nu_{k}}\leq s^{-\nu^{*}}$,
while on $\left[1,t_{n}\right]$, we use $s^{-\nu_{k}}\leq s^{-\nu_{*}}$
since $s^{-\alpha}$ is decreasing with respect to $\alpha$ for $s\geq 1$.
Therefore,
\[
\sum^{n}_{k=1}b^{\left(\nu\right)}_{k,1}\leq\int^{1}_{0}s^{-\nu^{*}}ds+\int^{t_{n}}_{1}s^{-\nu_{*}}ds=\frac{1}{1-\nu^{*}}+\frac{t^{1-\nu_{*}}_{n}-1}{1-\nu_{*}}\leq\frac{1}{1-\nu^{*}}+\frac{T^{1-\nu_{*}}-1}{1-\nu_{*}},
\]
since $t_{n}\leq T$.
\end{proof}

Now we derive a discrete energy estimate for the fully discrete solution.
For this purpose, we introduce a corrected discrete memory functional $\mathcal{E}^{(\nu)}_{\phi,n}=\Theta^{(\nu)}_{\phi,n}-R^{(\nu)}_{\phi,n}$.
\begin{lemma}
\label{lem:frac}Given the sequence $\left\{ \phi^{n}\right\}_{n\geq 0}$, $\phi^{n}\in L^{2}\left(\Omega\right)$,
define
\begin{align*}
 & \Theta^{\left(\nu\right)}_{\phi,0}=0,\qquad\Theta^{\left(\nu\right)}_{\phi,n}=\frac{1}{2}\sum^{n}_{k=1}b^{\left(\nu\right)}_{n,k}\lVert\phi^{k}\rVert^{2},\quad n\geq1,\\
 & \delta^{\left(\nu\right)}_{n,k}=b^{\left(\nu\right)}_{n,k+1}-b^{\left(\nu\right)}_{n-1,k},\qquad\left(\delta^{\left(\nu\right)}_{n,k}\right)_{+}=\max\left\{ \delta^{\left(\nu\right)}_{n,k},0\right\},\\
 & R^{\left(\nu\right)}_{\phi,0}=R^{\left(\nu\right)}_{\phi,1}=0,\qquad R^{\left(\nu\right)}_{\phi,n}=\frac{1}{2}\sum^{n}_{j=2}\sum^{j-1}_{k=1}\left(\delta^{\left(\nu\right)}_{j,k}\right)_{+}\lVert\phi^{k}\rVert^{2},\quad n\geq2,
\end{align*}
and 
\begin{equation}
\mathcal{E}^{\left(\nu\right)}_{\phi,n}=\Theta^{\left(\nu\right)}_{\phi,n}-R^{\left(\nu\right)}_{\phi,n}.\label{eq:lemfrac_E}
\end{equation}

Then for all $n\geq1$,
\[
\left(\Delta^{\nu\left(t_{n}\right)}_{\tau}\phi^{n},\phi^{n}\right)\geq\frac{1}{\tau}\left(\mathcal{E}^{\left(\nu\right)}_{\phi,n}-\mathcal{E}^{\left(\nu\right)}_{\phi,n-1}\right)-\frac{1}{2\tau}b^{\left(\nu\right)}_{n,1}\lVert \phi^{0}\rVert ^{2}.
\]
\end{lemma}
\begin{proof}
For fixed $n$, the order $\nu\left(t_{n}\right)$ is fixed. Hence
by (\ref{eq:discrete_caputo}), (\ref{eq:prop_b1}), summation by parts,
and Young's inequality, we obtain
\begin{align}
\left(\Delta^{\nu\left(t_{n}\right)}_{\tau}\phi^{n},\phi^{n}\right) & =\frac{1}{\tau}\left[b^{\left(\nu\right)}_{n,n}\lVert \phi^{n}\rVert ^{2}-\sum^{n-1}_{k=1}\left(b^{\left(\nu\right)}_{n,k+1}-b^{\left(\nu\right)}_{n,k}\right)\left(\phi^{k},\phi^{n}\right)-b^{\left(\nu\right)}_{n,1}\left(\phi^{0},\phi^{n}\right)\right]\nonumber\displaybreak[0]\\
 & \geq\frac{1}{2\tau}\left[b^{\left(\nu\right)}_{n,n}\lVert \phi^{n}\rVert ^{2}-\sum^{n-1}_{k=1}\left(b^{\left(\nu\right)}_{n,k+1}-b^{\left(\nu\right)}_{n,k}\right)\lVert \phi^{k}\rVert ^{2}-b^{\left(\nu\right)}_{n,1}\lVert \phi^{0}\rVert ^{2}\right],\label{eq:lemfrac_step4}
\end{align}
where we used ${\displaystyle \sum^{n-1}_{k=1}\left(b^{\left(\nu\right)}_{n,k+1}-b^{\left(\nu\right)}_{n,k}\right)+b^{\left(\nu\right)}_{n,1}=b^{\left(\nu\right)}_{n,n}}$.

By construction,
\[
2\left(\Theta^{\left(\nu\right)}_{\phi,n}-\Theta^{\left(\nu\right)}_{\phi,n-1}\right)=b^{\left(\nu\right)}_{n,n}\lVert \phi^{n}\rVert ^{2}+\sum^{n-1}_{k=1}\left(b^{\left(\nu\right)}_{n,k}-b^{\left(\nu\right)}_{n-1,k}\right)\lVert \phi^{k}\rVert ^{2}
\]
and $b^{\left(\nu\right)}_{n,k}-b^{\left(\nu\right)}_{n-1,k}=-\left(b^{\left(\nu\right)}_{n,k+1}-b^{\left(\nu\right)}_{n,k}\right)+\delta^{\left(\nu\right)}_{n,k}$,
therefore
\begin{equation}
2\left(\Theta^{\left(\nu\right)}_{\phi,n}-\Theta^{\left(\nu\right)}_{\phi,n-1}\right)=b^{\left(\nu\right)}_{n,n}\lVert \phi^{n}\rVert ^{2}-\sum^{n-1}_{k=1}\left(b^{\left(\nu\right)}_{n,k+1}-b^{\left(\nu\right)}_{n,k}\right)\lVert \phi^{k}\rVert ^{2}+\sum^{n-1}_{k=1}\delta^{\left(\nu\right)}_{n,k}\lVert \phi^{k}\rVert ^{2}.\label{eq:lemfrac_step6}
\end{equation}

Using (\ref{eq:lemfrac_step6}), we conclude from (\ref{eq:lemfrac_step4}) that
\[
\left(\Delta^{\nu\left(t_{n}\right)}_{\tau}\phi^{n},\phi^{n}\right)\geq\frac{1}{\tau}\left(\Theta^{\left(\nu\right)}_{\phi,n}-\Theta^{\left(\nu\right)}_{\phi,n-1}\right)-\frac{1}{2\tau}\sum^{n-1}_{k=1}\delta^{\left(\nu\right)}_{n,k}\lVert \phi^{k}\rVert ^{2}-\frac{1}{2\tau}b^{\left(\nu\right)}_{n,1}\lVert \phi^{0}\rVert ^{2}.
\]

Since $-\delta^{\left(\nu\right)}_{n,k}\geq-\left(\delta^{\left(\nu\right)}_{n,k}\right)_{+}$, we have
\begin{equation}
\left(\Delta^{\nu\left(t_{n}\right)}_{\tau}\phi^{n},\phi^{n}\right)\geq\frac{1}{\tau}\left(\Theta^{\left(\nu\right)}_{\phi,n}-\Theta^{\left(\nu\right)}_{\phi,n-1}\right)-\frac{1}{2\tau}\sum^{n-1}_{k=1}\left(\delta^{\left(\nu\right)}_{n,k}\right)_{+}\lVert \phi^{k}\rVert ^{2}-\frac{1}{2\tau}b^{\left(\nu\right)}_{n,1}\lVert \phi^{0}\rVert ^{2}.\label{eq:lem_2}
\end{equation}

By definition of $R^{\left(\nu\right)}_{\phi,n}$, we have
\begin{equation}
R^{\left(\nu\right)}_{\phi,n}-R^{\left(\nu\right)}_{\phi,n-1}=\frac{1}{2}\sum^{n-1}_{k=1}\left(\delta^{\left(\nu\right)}_{n,k}\right)_{+}\lVert \phi^{k}\rVert ^{2}.\label{eq:lemfrac_step7}
\end{equation}

Therefore, combining (\ref{eq:lem_2}), (\ref{eq:lemfrac_step7})
and (\ref{eq:lemfrac_E}), we immediately arrive at the assertion
of the lemma.
\end{proof}

\begin{lemma}[A corrected discrete energy estimate]\label{lem:corrected_energy}
Let $(\mathbf{u}^{n}_{h},p^{n}_{h},\mathbf{B}^{n}_{h})\in \mathcal{V}_{h}\times\mathcal{Q}_{h}\times\mathcal{C}_{h}$
be the solution of Problem \ref{prob:fully}. Suppose Assumptions \ref{assu:rhs} and \ref{assu:orders_bounds}
hold, and let $\mathcal{E}^{\left(\alpha\right)}_{\mathbf{u},n}$
and $\mathcal{E}^{\left(\beta\right)}_{\mathbf{B},n}$ be defined
as in Lemma \ref{lem:frac}. Then the fully discrete solution satisfies the corrected energy estimate
\begin{alignat*}{1}
 & \mathcal{E}^{\left(\alpha\right)}_{\mathbf{u},n}+\mathcal{E}^{\left(\beta\right)}_{\mathbf{B},n}
+\frac{\tau}{2}\sum^{n}_{k=1}\left(\frac{1}{\mathrm{Re}}\lVert \nabla\mathbf{u}^{k}_{h}\rVert ^{2}+\frac{1}{\mathrm{Rm}}\lVert \nabla\mathbf{B}^{k}_{h}\rVert ^{2}+\zeta\lVert \nabla\cdot\mathbf{u}^{k}_{h}\rVert ^{2}+\chi\lVert \nabla\cdot\mathbf{B}^{k}_{h}\rVert ^{2}\right)\\
 & \leq C\left(\lVert \mathbf{u}^{0}_{h}\rVert ^{2}+\lVert \mathbf{B}^{0}_{h}\rVert ^{2}\right)+C\tau\sum^{n}_{k=1}\left(\lVert \mathbf{f}^{k}\rVert ^{2}_{H^{-1}}+\lVert \mathbf{g}^{k}\rVert ^{2}_{H^{-1}}\right)
\end{alignat*}
for all $1\leq n\leq N$, where $C>0$ is independent of $h$ and
$\tau$.
\end{lemma}
\begin{proof}
Choose $\left(\mathbf{v}_{h},q_{h},\mathbf{C}_{h}\right)=\left(\mathbf{u}^{n}_{h},p^{n}_{h},\mathbf{B}^{n}_{h}\right)$
in (\ref{eq:fully_u1})–(\ref{eq:fully_u2}), then sum the resulting identities, and use (\ref{eq:trilinear_prop_c}) to obtain
\begin{align}
 & \left(\Delta^{\alpha\left(t_{n}\right)}_{\tau}\mathbf{u}^{n}_{h},\mathbf{u}^{n}_{h}\right)
  +\left(\Delta^{\beta\left(t_{n}\right)}_{\tau}\mathbf{B}^{n}_{h},\mathbf{B}^{n}_{h}\right)
  +\frac{1}{\mathrm{Re}}\lVert \nabla\mathbf{u}^{n}_{h}\rVert ^{2}
  +\frac{1}{\mathrm{Rm}}\lVert \nabla\mathbf{B}^{n}_{h}\rVert ^{2}\nonumber\\
 &+\zeta\lVert \nabla\cdot\mathbf{u}^{n}_{h}\rVert ^{2}
  +\chi\lVert \nabla\cdot\mathbf{B}^{n}_{h}\rVert ^{2}
  =\left<\mathbf{f}^{n},\mathbf{u}^{n}_{h}\right>
  +\left<\mathbf{g}^{n},\mathbf{B}^{n}_{h}\right>.\label{eq:stab_step2}
\end{align}

Applying Lemma \ref{lem:frac}, we obtain
\begin{alignat*}{1}
 & \left(\Delta^{\alpha\left(t_{n}\right)}_{\tau}\mathbf{u}^{n}_{h},\mathbf{u}^{n}_{h}\right)\geq\frac{1}{\tau}\left(\mathcal{E}^{\left(\alpha\right)}_{\mathbf{u},n}-\mathcal{E}^{\left(\alpha\right)}_{\mathbf{u},n-1}\right)-\frac{1}{2\tau}b^{\left(\alpha\right)}_{n,1}\lVert \mathbf{u}^{0}_{h}\rVert ^{2},\\
 & \left(\Delta^{\beta\left(t_{n}\right)}_{\tau}\mathbf{B}^{n}_{h},\mathbf{B}^{n}_{h}\right)\geq\frac{1}{\tau}\left(\mathcal{E}^{\left(\beta\right)}_{\mathbf{B},n}-\mathcal{E}^{\left(\beta\right)}_{\mathbf{B},n-1}\right)-\frac{1}{2\tau}b^{\left(\beta\right)}_{n,1}\lVert \mathbf{B}^{0}_{h}\rVert ^{2}.
\end{alignat*}

The terms in the right-hand side are estimated as follows:
\[
\left|\left<\mathbf{f}^{n},\mathbf{u}^{n}_{h}\right>\right|+
\left|\left<\mathbf{g}^{n},\mathbf{B}^{n}_{h}\right>\right|\leq
\frac{1}{2\mathrm{Re}}\lVert \nabla\mathbf{u}^{n}_{h}\rVert ^{2}+
\frac{1}{2\mathrm{Rm}}\lVert \nabla\mathbf{B}^{n}_{h}\rVert ^{2}+
C\left(\Vert \mathbf{f}^{n}\Vert ^{2}_{H^{-1}}+\Vert \mathbf{g}^{n}\Vert ^{2}_{H^{-1}}\right).
\]

Using the estimates derived above in (\ref{eq:stab_step2}),
then summing the resulting inequalities
over $n$, we obtain
\[
\mathcal{E}^{\left(\alpha\right)}_{\mathbf{u},n}+\mathcal{E}^{\left(\beta\right)}_{\mathbf{B},n}
+\frac{\tau}{2\mathrm{Re}}\sum^{n}_{k=1}\lVert \nabla\mathbf{u}^{k}_{h}\rVert ^{2}+\frac{\tau}{2\mathrm{Rm}}\sum^{n}_{k=1}\lVert \nabla\mathbf{B}^{k}_{h}\rVert ^{2}+\zeta\tau\sum^{n}_{k=1}\lVert \nabla\cdot\mathbf{u}^{k}_{h}\rVert ^{2}+\chi\tau\sum^{n}_{k=1}\lVert \nabla\cdot\mathbf{B}^{k}_{h}\rVert ^{2}
\]
\[
\leq\mathcal{E}^{\left(\alpha\right)}_{\mathbf{u},0}+\mathcal{E}^{\left(\beta\right)}_{\mathbf{B},0}+
\frac{1}{2}\left(\sum^{n}_{k=1}b^{\left(\alpha\right)}_{k,1}\right)\lVert \mathbf{u}^{0}_{h}\rVert ^{2}+\frac{1}{2}\left(\sum^{n}_{k=1}b^{\left(\beta\right)}_{k,1}\right)\lVert \mathbf{B}^{0}_{h}\rVert ^{2}+C\tau\sum^{n}_{k=1}\left(\lVert \mathbf{f}^{k}\rVert ^{2}_{H^{-1}}+\lVert \mathbf{g}^{k}\rVert ^{2}_{H^{-1}}\right).
\]

Finally, taking into account (\ref{eq:prop_b2}), and $\mathcal{E}^{\left(\alpha\right)}_{\mathbf{u},0}=0$,
$\mathcal{E}^{\left(\beta\right)}_{\mathbf{B},0}=0$, we arrive
at the assertion of the lemma.
\end{proof}

\begin{lemma}
\label{lem:frac1}Given the sequence $\left\{ \phi^{n}\right\} ,\phi^{n}\in L^{2}\left(\Omega\right)$, the following inequality holds:
\[
\left(\Delta^{\nu\left(t_{n}\right)}_{\tau}\phi^{n},\phi^{n}\right)\geq\frac{1}{2}\Delta^{\nu\left(t_{n}\right)}_{\tau}\lVert \phi^{n}\rVert ^{2}.
\]
\end{lemma}
\begin{proof}
This inequality immediately follows from the definition in \eqref{eq:discrete_caputo}, the Cauchy inequality, and Young’s inequality.
\end{proof}

For the stability and convergence analysis, we use the abstract discrete fractional Gr\"{o}nwall lemma of \cite{Liao_2019}. 
Its application to the present variable-order L1 discretization requires verification of Assumptions A1--A3 of \cite{Liao_2019} 
for the associated discrete kernels. For convenience, Appendix A recalls these assumptions, 
the complementary kernels, and the version of the discrete Gr\"{o}nwall theorem used here. 
It is shown below that the kernels generated by $\Delta_{\tau}^{\nu\left(t_{n}\right)}$ satisfy Assumptions A1--A3 
with comparison exponent $\nu_{*}$, provided
$0<\nu_{*}\le \nu\left(t\right)\le \nu^{*}<1$ on $\left[0,T\right]$.

\begin{lemma}\label{lem:check_frac_assump}
Assume that there exist constants $0<\nu_{*}<\nu^{*}<1$ such
that $\nu\left(t\right)\in\left[\nu_{*},\nu^{*}\right]$ for all
$t\in\left[0,T\right]$. For $1\leq k\leq n\leq N$, set $\nu_{n}=\nu\left(t_{n}\right)$, $t_{n}\in\mathcal{I}_{N}$,
and define
\[
A^{\left(n\right)}_{n-k}=\frac{1}{\tau}b^{\left(\nu\right)}_{n,k},
\]
or, equivalently,
\[
A^{\left(n\right)}_{n-k}=\frac{1}{\tau}\int^{t_{k}}_{t_{k-1}}\omega_{1-\nu_{n}}\left(t_{n}-s\right)ds,\qquad
\omega_{1-\nu}\left(t\right)=\frac{t^{-\nu}}{\Gamma\left(1-\nu\right)}.
\]

Then the kernels $\left\{ A^{\left(n\right)}_{n-k}\right\} $ satisfy
Assumptions A1–A3 of \cite{Liao_2019}, where the fixed exponent appearing in A2 is chosen as $\nu_{*}$,
and
\[
\rho=1,\qquad\pi_{A}=\frac{\Gamma\left(1-\nu^{*}\right)}{\Gamma\left(1-\nu_{*}\right)}\max\left\{ 1,T^{\nu^{*}-\nu_{*}}\right\} .
\]
\end{lemma}
\begin{proof}
For each fixed $n$, the function $f\left(x\right)=\omega_{1-\nu_{n}}\left(x\right)$
is positive and strictly decreasing on $\left(0,\infty\right)$. 
Therefore, the function $g\left(s\right)=f\left(t_{n}-s\right)$ is positive
and strictly increasing on $\left[0,t_{n}\right)$.
Since the mesh is uniform, for $k=1,...,n-1$,
\[
\int^{t_{k}}_{t_{k-1}}g\left(s\right)ds\leq\int^{t_{k+1}}_{t_{k}}g\left(s\right)ds.
\]

Therefore, 
\[
A^{\left(n\right)}_{n-k}\leq A^{\left(n\right)}_{n-k-1},\qquad k=1,...,n-1,
\]
or, equivalently,
\[
A^{\left(n\right)}_{0}\geq A^{\left(n\right)}_{1}\geq...\geq A^{\left(n\right)}_{n-1}>0.
\]
Thus A1 holds.

Next, let $x\in\left(0,T\right]$. Since $\nu_{n}\in\left[\nu_{*},\nu^{*}\right]$,
we have 
\[
\omega_{1-\nu_{n}}\left(x\right)=\frac{\Gamma\left(1-\nu_{*}\right)}{\Gamma\left(1-\nu_{n}\right)}x^{-\left(\nu_{n}-\nu_{*}\right)}\omega_{1-\nu_{*}}\left(x\right).
\]

Because $\Gamma$ is decreasing on $\left(0,1\right)$ and $1-\nu_{n}\geq1-\nu^{*}$, it follows that $\Gamma\left(1-\nu_{n}\right)\leq\Gamma\left(1-\nu^{*}\right)$, and hence
\[
\frac{\Gamma\left(1-\nu_{*}\right)}{\Gamma\left(1-\nu_{n}\right)}\geq\frac{\Gamma\left(1-\nu_{*}\right)}{\Gamma\left(1-\nu^{*}\right)}.
\]

Moreover, since $0\leq\nu_{n}-\nu_{*}\leq\nu^{*}-\nu_{*}$ and $0<x\leq T$,
\[
x^{-\left(\nu_{n}-\nu_{*}\right)}\geq\frac{1}{\max\left\{ 1,T^{\nu^{*}-\nu_{*}}\right\} }.
\]

Therefore
\[
\omega_{1-\nu_{n}}\left(x\right)\geq\frac{\Gamma\left(1-\nu_{*}\right)}{\Gamma\left(1-\nu^{*}\right)\max\left\{ 1,T^{\nu^{*}-\nu_{*}}\right\} }\omega_{1-\nu_{*}}\left(x\right)=\frac{1}{\pi_{A}}\omega_{1-\nu_{*}}\left(x\right).
\]

Applying this estimate with $x=t_{n}-s$, for a.e. $s\in\left(t_{k-1},t_{k}\right)$, and integrating over $\left(t_{k-1},t_{k}\right)$ gives
\[
A^{\left(n\right)}_{n-k}\geq\frac{1}{\pi_{A}\tau}\int^{t_{k}}_{t_{k-1}}\omega_{1-\nu_{*}}\left(t_{n}-s\right)ds,\qquad1\leq k\leq n\leq N.
\]
Thus A2 holds with the fixed exponent $\nu_{*}$.

Finally, since the mesh is uniform, ${\displaystyle \rho_{k}=\frac{\tau_{k}}{\tau_{k+1}}=1}$,
so A3 holds with $\rho=1$.
\end{proof}

In the stability and convergence theorems below, we assume $\alpha\left(t\right)=\beta\left(t\right)$.
This restriction is imposed only to simplify the presentation of the stability and convergence results.

We first establish a discrete stability estimate.

\begin{theorem}[Stability]\label{thm:stability}
Let $(\mathbf{u}^{n}_{h},p^{n}_{h},\mathbf{B}^{n}_{h})\in\mathcal{V}_{h}\times\mathcal{Q}_{h}\times\mathcal{C}_{h}$
be the solution of Problem \ref{prob:fully}. Suppose
$\alpha\left(t\right)=\beta\left(t\right)$ on $\left[0,T\right]$,
and Assumptions \ref{assu:rhs} and \ref{assu:orders_bounds}
hold. Then the fully discrete solution satisfies
\[
\lVert\mathbf{u}^{n}_{h}\rVert^{2}+\lVert\mathbf{B}^{n}_{h}\rVert^{2}
\leq \lVert\mathbf{u}^{0}_{h}\rVert^{2}+\lVert\mathbf{B}^{0}_{h}\rVert^{2}
+C\max_{1\leq j\leq n}\left(\lVert\mathbf{f}^{j}\rVert^{2}_{H^{-1}}+\lVert\mathbf{g}^{j}\rVert^{2}_{H^{-1}}\right)
\]
for all $1\leq n\leq N$, where $C>0$ is independent of $h$ and
$\tau$.
\end{theorem}
\begin{proof}
Denote $E_{n}=\lVert\mathbf{u}^{n}_{h}\rVert^{2}+\lVert\mathbf{B}^{n}_{h}\rVert^{2}$,
$D_{n}=\lVert\nabla\mathbf{u}^{n}_{h}\rVert^{2}+\lVert\nabla\mathbf{B}^{n}_{h}\rVert^{2}$.
As in the proof of Lemma \ref{lem:corrected_energy}, we choose $\left(\mathbf{v}_{h},q_{h},\mathbf{C}_{h}\right)=\left(\mathbf{u}^{n}_{h},p^{n}_{h},\mathbf{B}^{n}_{h}\right)$
in (\ref{eq:fully_u1})–(\ref{eq:fully_u2}),
then sum the resulting identities, and use (\ref{eq:trilinear_prop_c})
to obtain
\[
\left(\Delta^{\alpha\left(t_{n}\right)}_{\tau}\mathbf{u}^{n}_{h},\mathbf{u}^{n}_{h}\right)+\left(\Delta^{\alpha\left(t_{n}\right)}_{\tau}\mathbf{B}^{n}_{h},\mathbf{B}^{n}_{h}\right)+\frac{1}{\mathrm{Re}}\lVert\nabla\mathbf{u}^{n}_{h}\rVert^{2}+\frac{1}{\mathrm{Rm}}\lVert\nabla\mathbf{B}^{n}_{h}\rVert^{2}
\]
\begin{equation}
+\zeta\lVert\nabla\cdot\mathbf{u}^{n}_{h}\rVert^{2}+\chi\lVert\nabla\cdot\mathbf{B}^{n}_{h}\rVert^{2}=\left<\mathbf{f}^{n},\mathbf{u}^{n}_{h}\right>+\left<\mathbf{g}^{n},\mathbf{B}^{n}_{h}\right>.
\label{eq:stab_step1}
\end{equation}

Applying Lemma \ref{lem:frac1}, we obtain
\[
\left(\Delta^{\alpha\left(t_{n}\right)}_{\tau}\mathbf{u}^{n}_{h},\mathbf{u}^{n}_{h}\right)+\left(\Delta^{\alpha\left(t_{n}\right)}_{\tau}\mathbf{B}^{n}_{h},\mathbf{B}^{n}_{h}\right)
\geq\frac{1}{2}\Delta^{\alpha\left(t_{n}\right)}_{\tau}E_{n}.
\]

The right-hand side is estimated exactly as in the proof of Lemma
\ref{lem:corrected_energy}. Omitting the nonnegative divergence terms on the left-hand side, then using the estimates derived above in (\ref{eq:stab_step1}), we obtain
\begin{equation}
\Delta^{\alpha\left(t_{n}\right)}_{\tau}E_{n}+2c_{0}D_{n}\leq2F_{n},\label{eq:stab_step5}
\end{equation}
where $F_{n}=C\left(\lVert\mathbf{f}^{n}\rVert^{2}_{H^{-1}}+\lVert\mathbf{g}^{n}\rVert^{2}_{H^{-1}}\right)$,
$c_{0}=\frac{1}{2}\min\left\{ \frac{1}{\mathrm{Re}},\frac{1}{\mathrm{Rm}}\right\}$.

Set $A_{n-k}^{(n)}=\frac{1}{\tau} b_{n,k}^{(\alpha)}$, $1\le k\le n$.
Let $\{\mathcal P_{n-j}^{(n)}\}_{j=1}^{n}$, defined in Appendix A, 
be the complementary kernels
associated with $A_{n-k}^{(n)}$. Then
\[
\sum_{j=k}^{n}\mathcal P_{n-j}^{(n)}A_{j-k}^{(j)}=1,
\qquad 1\le k\le n.
\]

Multiplying (\ref{eq:stab_step5}) at the $j$th time level by $\mathcal{P}^{\left(n\right)}_{n-j}$,
and summing the resulting inequalities over $j=1,...,n$, we obtain
\[
\sum^{n}_{j=1}\mathcal{P}^{\left(n\right)}_{n-j}\Delta^{\alpha\left(t_{j}\right)}_{\tau}E_{j}+2c_{0}\sum^{n}_{j=1}\mathcal{P}^{\left(n\right)}_{n-j}D_{j}\leq2\sum^{n}_{j=1}\mathcal{P}^{\left(n\right)}_{n-j}F_{j}.
\]

Using ${\displaystyle \Delta^{\alpha\left(t_{j}\right)}_{\tau}E_{j}=\sum^{j}_{k=1}A^{\left(j\right)}_{j-k}\left(E_{k}-E_{k-1}\right)}$, we get
\[
\sum^{n}_{j=1}\mathcal{P}^{\left(n\right)}_{n-j}\Delta^{\alpha\left(t_{j}\right)}_{\tau}E_{j} 
 =\sum^{n}_{k=1}\left(E_{k}-E_{k-1}\right)\sum^{n}_{j=k}\mathcal{P}^{\left(n\right)}_{n-j}A^{\left(j\right)}_{j-k}
 =\sum^{n}_{k=1}\left(E_{k}-E_{k-1}\right)
 =E_{n}-E_{0}.
\]

Therefore,
\begin{equation}
E_{n}+2c_{0}\sum^{n}_{j=1}\mathcal{P}^{\left(n\right)}_{n-j}D_{j}\leq E_{0}+2\sum^{n}_{j=1}\mathcal{P}^{\left(n\right)}_{n-j}F_{j}.
\label{eq:stab_step6}
\end{equation}

By Lemma \ref{lem:check_frac_assump}, the kernels generated by the variable-order L1 approximation satisfy
Assumptions A1--A3, recalled in Appendix A, with comparison exponent $\alpha_{*}$, $\rho=1$,
and
\[
\pi_{A}=\frac{\Gamma\left(1-\alpha^{*}\right)}{\Gamma\left(1-\alpha_{*}\right)}\max\left\{ 1,T^{\alpha^{*}-\alpha_{*}}\right\}.
\]

Using the complementary kernel bound in the remark of Appendix A with $\gamma=\alpha_{*}$ and $g^{j}=F_{j}$, we obtain
\[
\sum_{j=1}^{n}\mathcal {P}^{(n)}_{n-j}F_{j}
\le
\pi_{A}\Gamma(1-\alpha_{*})
\max_{1\le j\le n}\left(t_{j}^{\alpha_{*}}F_{j}\right)
\le
\pi_{A}\Gamma(1-\alpha_{*})t_{n}^{\alpha_{*}}
\max_{1\le j\le n}F_{j}.
\]

Then Eq. (\ref{eq:stab_step6}) yields
\[
E_{n}\leq E_{0}+2\pi_{A}\Gamma(1-\alpha_{*})t_{n}^{\alpha_{*}}\max_{1\leq j\leq n}F_{j}.
\]
Considering $t_{n}\leq T$, we arrive at the statement of the theorem.
\end{proof}

We employ the Stokes projector $\left(\Pi_{h}\mathbf{u}^{n},\Psi_{h}p^{n}\right)\in \mathcal{V}_{h}\times \mathcal{Q}_{h}$
defined as
\begin{align*}
 & \left(\nabla\left(\Pi_{h}\mathbf{u}^{n}-\mathbf{u}^{n}\right),\nabla\mathbf{v}_{h}\right)-\left(\Psi_{h}p^{n}-p^{n},\nabla\cdot\mathbf{v}_{h}\right)=0,\\
 & \left(\nabla\cdot\left(\Pi_{h}\mathbf{u}^{n}-\mathbf{u}^{n}\right),q_{h}\right)=0
\end{align*}
for all $\mathbf{v}_{h}\in \mathcal{V}_{h}$, $q_{h}\in \mathcal{Q}_{h}$, and an elliptic
projector $\Xi_{h}$ defined as
\[
\left(\nabla\left(\Xi_{h}\mathbf{B}^{n}-\mathbf{B}^{n}\right),\nabla\mathbf{C}_{h}\right)=0
\]
for all $\mathbf{C}_{h}\in \mathcal{C}_{h}$. Further, we introduce the decomposition
\begin{align}
 & \mathbf{u}^{n}-\mathbf{u}^{n}_{h}=\left(\mathbf{u}^{n}-\Pi_{h}\mathbf{u}^{n}\right)+\left(\Pi_{h}\mathbf{u}^{n}-\mathbf{u}^{n}_{h}\right)=\psi^{n}_{\mathbf{u}}+\xi^{n}_{\mathbf{u}},\nonumber \\
 & p^{n}-p^{n}_{h}=\left(p^{n}-\Psi_{h}p^{n}\right)+\left(\Psi_{h}p^{n}-p^{n}_{h}\right)=\psi^{n}_{p}+\xi^{n}_{p},\label{eq:projection_difference}\\
 & \mathbf{B}^{n}-\mathbf{B}^{n}_{h}=\left(\mathbf{B}^{n}-\Xi_{h}\mathbf{B}^{n}\right)+\left(\Xi_{h}\mathbf{B}^{n}-\mathbf{B}^{n}_{h}\right)=\psi^{n}_{\mathbf{B}}+\xi^{n}_{\mathbf{B}}.\nonumber 
\end{align}

By the approximation properties of $\Pi_{h}$ and $\Xi_{h}$, we have
\begin{equation}
\lVert\psi^{n}_{\mathbf{u}}\rVert+h\lVert\nabla\psi^{n}_{\mathbf{u}}\rVert\le Ch^{l+1}\lVert\mathbf{u}^{n}\rVert_{H^{l+1}},\qquad
\lVert\psi^{n}_{\mathbf{B}}\rVert+h\lVert\nabla\psi^{n}_{\mathbf{B}}\rVert\le Ch^{m+1}\lVert\mathbf{B}^{n}\rVert_{H^{m+1}}
\label{eq:approximation_errors}
\end{equation}
for $0\le n\le N$.

We next derive the convergence estimate for the fully discrete method.
\begin{theorem}[Convergence]\label{thm:convergence}
Let $\left(\mathbf{u}^{n}_{h},p^{n}_{h},\mathbf{B}^{n}_{h}\right)\in \mathcal{V}_{h}\times\mathcal{Q}_{h}\times\mathcal{C}_{h}$
be the solution of Problem \ref{prob:fully}. 
Assume that $\alpha\left(t\right)=\beta\left(t\right)$ on $\left[0,T\right]$, and that
$\lVert \Pi_{h}\mathbf{u}^{0}-\mathbf{u}_{h}^{0}  \rVert \leq Ch^{l}$
and $\lVert \Xi_{h}\mathbf{B}^{0}-\mathbf{B}_{h}^{0}  \rVert \leq Ch^{m}$,
where $l$ and $m$ are defined in (\ref{eq:approximation_errors}).
Under Assumptions \ref{assu:regularity} and \ref{assu:orders_bounds},
there exists $\tau_{0}>0$, independent of $h$, such that for $0<h\leq 1$ 
and $0<\tau\leq\tau_{0}$,
\[
\lVert \mathbf{u}\left(t_{n}\right)-\mathbf{u}^{n}_{h}\rVert +\lVert \mathbf{B}\left(t_{n}\right)-\mathbf{B}^{n}_{h}\rVert \le C\left(h^{\min\left\{ l,m\right\} }+\tau\right),\qquad1\le n\le N,
\]
where $C$ is independent of $h$ and $\tau$.
\end{theorem}
\begin{proof}
Denote $E_{n}=\lVert \xi^{n}_{\mathbf{u}}\rVert ^{2}+\lVert \xi^{n}_{\mathbf{B}}\rVert ^{2}$,
$D_{n}=\lVert \nabla\xi^{n}_{\mathbf{u}}\rVert ^{2}+\lVert \nabla\xi^{n}_{\mathbf{B}}\rVert ^{2}$.
Subtract the identities (\ref{eq:fully_u1})–(\ref{eq:fully_u2})
from the identities (\ref{eq:semi_u1})–(\ref{eq:semi_u2}), use the
decomposition (\ref{eq:projection_difference}), then choose $\left(\mathbf{v}_{h},q_{h},\mathbf{C}_{h}\right)=\left(\xi^{n}_{\mathbf{u}},\xi^{n}_{p},\xi^{n}_{\mathbf{B}}\right)$
to obtain
\begin{align}
 & \left(\Delta^{\alpha\left(t_{n}\right)}_{\tau}\xi^{n}_{\mathbf{u}},\xi^{n}_{\mathbf{u}}\right)+\left(\Delta^{\alpha\left(t_{n}\right)}_{\tau}\xi^{n}_{\mathbf{B}},\xi^{n}_{\mathbf{B}}\right)+\left(\Delta^{\alpha\left(t_{n}\right)}_{\tau}\psi^{n}_{\mathbf{u}},\xi^{n}_{\mathbf{u}}\right)\nonumber\displaybreak[0]\\
 & \qquad+\left(\Delta^{\alpha\left(t_{n}\right)}_{\tau}\psi^{n}_{\mathbf{B}},\xi^{n}_{\mathbf{B}}\right)
+\frac{1}{\mathrm{Re}}\lVert \nabla\xi^{n}_{\mathbf{u}}\rVert ^{2}+\frac{1}{\mathrm{Rm}}\lVert \nabla\xi^{n}_{\mathbf{B}}\rVert ^{2}\displaybreak[0]\nonumber \\
 & \qquad+\ell\left(\xi^{n}_{\mathbf{u}};\mathbf{u}^{n},\xi^{n}_{\mathbf{u}}\right)+\ell\left(\psi^{n}_{\mathbf{u}};\mathbf{u}^{n},\xi^{n}_{\mathbf{u}}\right)+\ell\left(\mathbf{u}^{n}_{h};\psi^{n}_{\mathbf{u}},\xi^{n}_{\mathbf{u}}\right)\displaybreak[0]\nonumber \\
 & \qquad+\ell\left(\xi^{n}_{\mathbf{u}};\mathbf{B}^{n},\xi^{n}_{\mathbf{B}}\right)+\ell\left(\psi^{n}_{\mathbf{u}};\mathbf{B}^{n},\xi^{n}_{\mathbf{B}}\right)+\ell\left(\mathbf{u}^{n}_{h};\psi^{n}_{\mathbf{B}},\xi^{n}_{\mathbf{B}}\right)\displaybreak[0]\nonumber \\
 & \qquad-\ell\left(\xi^{n}_{\mathbf{B}};\mathbf{B}^{n},\xi^{n}_{\mathbf{u}}\right)-\ell\left(\psi^{n}_{\mathbf{B}};\mathbf{B}^{n},\xi^{n}_{\mathbf{u}}\right)-\ell\left(\mathbf{B}^{n}_{h};\psi^{n}_{\mathbf{B}},\xi^{n}_{\mathbf{u}}\right)\nonumber \\
 & \qquad-\ell\left(\xi^{n}_{\mathbf{B}};\mathbf{u}^{n},\xi^{n}_{\mathbf{B}}\right)-\ell\left(\psi^{n}_{\mathbf{B}};\mathbf{u}^{n},\xi^{n}_{\mathbf{B}}\right)-\ell\left(\mathbf{B}^{n}_{h};\psi^{n}_{\mathbf{u}},\xi^{n}_{\mathbf{B}}\right)\nonumber \\
 & \qquad+\chi\lVert \nabla\cdot\xi^{n}_{\mathbf{B}}\rVert ^{2}+\zeta\lVert \nabla\cdot\xi^{n}_{\mathbf{u}}\rVert ^{2}+\chi\left(\nabla\cdot\psi^{n}_{\mathbf{B}},\nabla\cdot\xi^{n}_{\mathbf{B}}\right)\nonumber\displaybreak[0]\\
 & \qquad+\zeta\left(\nabla\cdot\psi^{n}_{\mathbf{u}},\nabla\cdot\xi^{n}_{\mathbf{u}}\right)
 +\left(\mathbf{r}^{n}_{1},\xi^{n}_{\mathbf{u}}\right)+\left(\mathbf{r}^{n}_{2},\xi^{n}_{\mathbf{B}}\right)=0.
 \label{eq:conv_step2}
\end{align}

Using Lemma \ref{lem:frac1}, we have:
\[
\left(\Delta^{\alpha\left(t_{n}\right)}_{\tau}\xi^{n}_{\mathbf{u}},\xi^{n}_{\mathbf{u}}\right)+\left(\Delta^{\alpha\left(t_{n}\right)}_{\tau}\xi^{n}_{\mathbf{B}},\xi^{n}_{\mathbf{B}}\right)\geq\frac{1}{2}\Delta^{\alpha\left(t_{n}\right)}_{\tau}E_{n}.
\]

Next, 
\[
\left|\left(\Delta^{\alpha\left(t_{n}\right)}_{\tau}\psi^{n}_{\mathbf{u}},\xi^{n}_{\mathbf{u}}\right)\right|
\leq\frac{\epsilon}{\mathrm{Re}}\lVert \nabla\xi^{n}_{\mathbf{u}}\rVert ^{2}+C_{\epsilon}\mathrm{Re}\lVert \Delta^{\alpha\left(t_{n}\right)}_{\tau}\psi^{n}_{\mathbf{u}}\rVert ^{2}.
\]

Further, using the relation 
$\psi^{k}_{\mathbf{u}}-\psi^{k-1}_{\mathbf{u}}=\left(\mathbf{u}^{k}-\mathbf{u}^{k-1}\right)-\Pi_{h}\left(\mathbf{u}^{k}-\mathbf{u}^{k-1}\right)$, and
the approximation property of $\Pi_{h}$, we obtain
\[
\lVert \psi^{k}_{\mathbf{u}}-\psi^{k-1}_{\mathbf{u}}\rVert \leq C\tau h^{l+1}\sup_{t\in\left[t_{k-1},t_{k}\right]}\lVert \partial_{t}\mathbf{u}\left(t\right)\rVert _{H^{l+1}}.
\]

Hence, by using (\ref{eq:discrete_caputo}), (\ref{eq:prop_b3}) and (\ref{eq:approximation_errors}), we have
\begin{align*}
\lVert \Delta^{\alpha\left(t_{n}\right)}_{\tau}\psi^{n}_{\mathbf{u}}\rVert
 & \leq\frac{1}{\tau}\sum^{n}_{k=1}b^{\left(\alpha\right)}_{n,k}\lVert \psi^{k}_{\mathbf{u}}-\psi^{k-1}_{\mathbf{u}}\rVert \\
 & \leq Ch^{l+1}\left(\sum^{n}_{k=1}b^{\left(\alpha\right)}_{n,k}\right)\sup_{t\in\left[0,t_{n}\right]}\lVert \partial_{t}\mathbf{u}\left(t\right)\rVert _{H^{l+1}}\leq Ch^{l+1}\sup_{t\in\left[0,T\right]}\lVert \partial_{t}\mathbf{u}\left(t\right)\rVert _{H^{l+1}}.
\end{align*}

Therefore, for any $\epsilon>0$,
\[
\left|\left(\Delta^{\alpha\left(t_{n}\right)}_{\tau}\psi^{n}_{\mathbf{u}},\xi^{n}_{\mathbf{u}}\right)\right|\leq\frac{\epsilon}{\mathrm{Re}}\lVert \nabla\xi^{n}_{\mathbf{u}}\rVert ^{2}+C_{\epsilon}h^{2l+2}\sup_{t\in\left[0,T\right]}\lVert \partial_{t}\mathbf{u}\left(t\right)\rVert ^{2}_{H^{l+1}}.
\]

Similarly,
\[
\left|\left(\Delta^{\alpha\left(t_{n}\right)}_{\tau}\psi^{n}_{\mathbf{B}},\xi^{n}_{\mathbf{B}}\right)\right|\leq\frac{\epsilon}{\mathrm{Rm}}\lVert \nabla\xi^{n}_{\mathbf{B}}\rVert ^{2}+C_{\epsilon}h^{2m+2}\sup_{t\in\left[0,T\right]}\lVert \partial_{t}\mathbf{B}\left(t\right)\rVert ^{2}_{H^{m+1}}.
\]

We now estimate the nonlinear terms. By H\"{o}lder's inequality,
the two-dimensional Ladyzhenskaya inequality, the Poincar\'{e} inequality,
and Young's inequality, we obtain for any $\epsilon>0$:
\begin{align*}
 & \left|\ell\left(\xi^{n}_{\mathbf{u}};\mathbf{u}^{n},\xi^{n}_{\mathbf{u}}\right)\right|\leq\frac{\epsilon}{\mathrm{Re}}\left\Vert \nabla\xi^{n}_{\mathbf{u}}\right\Vert ^{2}+ C_{\epsilon}\left\Vert \xi^{n}_{\mathbf{u}}\right\Vert ^{2}\left\Vert \mathbf{u}^{n}\right\Vert ^{2}_{W^{1,\infty}},\displaybreak[0]\\
 & \left|\ell\left(\psi^{n}_{\mathbf{u}};\mathbf{u}^{n},\xi^{n}_{\mathbf{u}}\right)\right|\leq\frac{\epsilon}{\mathrm{Re}}\left\Vert \nabla\xi^{n}_{\mathbf{u}}\right\Vert ^{2}+C_{\epsilon}\left\Vert \nabla\psi^{n}_{\mathbf{u}}\right\Vert ^{2}\left\Vert \mathbf{u}^{n}\right\Vert ^{2}_{H^{1}}.
\end{align*}
The other nonlinear terms of the same form are estimated analogously.

Using the error decomposition (\ref{eq:projection_difference}) and
the trilinearity of $\ell$, we decompose the term involving the discrete solution into contributions of the exact solution, the projection error, and the discrete error:
\[
\left|\ell\left(\mathbf{u}^{n}_{h};\psi^{n}_{\mathbf{u}},\xi^{n}_{\mathbf{u}}\right)\right|\leq\left|\ell\left(\mathbf{u}^{n};\psi^{n}_{\mathbf{u}},\xi^{n}_{\mathbf{u}}\right)\right|+\left|\ell\left(\psi^{n}_{\mathbf{u}};\psi^{n}_{\mathbf{u}},\xi^{n}_{\mathbf{u}}\right)\right|+\left|\ell\left(\xi^{n}_{\mathbf{u}};\psi^{n}_{\mathbf{u}},\xi^{n}_{\mathbf{u}}\right)\right|.
\]

Applying H\"{o}lder's inequality, the two-dimensional Ladyzhenskaya inequality, the
Poincar\'{e} inequality, the inverse inequality for finite element
functions, and Young's inequality, we obtain
\[
\left|\ell\left(\mathbf{u}^{n}_{h};\psi^{n}_{\mathbf{u}},\xi^{n}_{\mathbf{u}}\right)\right|\leq\frac{\epsilon}{\mathrm{Re}}\lVert \nabla\xi^{n}_{\mathbf{u}}\rVert ^{2}+C_{\epsilon}\lVert \mathbf{u}^{n}\rVert ^{2}_{L^{\infty}}\lVert \psi^{n}_{\mathbf{u}}\rVert ^{2}+C_{\epsilon}\lVert \nabla\psi^{n}_{\mathbf{u}}\rVert ^{4}+Ch^{-1}\lVert \nabla\psi^{n}_{\mathbf{u}}\rVert ^{2}\lVert \xi^{n}_{\mathbf{u}}\rVert ^{2}.
\]
The corresponding terms involving $\mathbf{B}^{n}_{h}$ are handled
analogously and are omitted for brevity.

The rest of the terms are estimated as follows:
\begin{alignat*}{1}
 & \chi\left|\left(\nabla\cdot\psi^{n}_{\mathbf{B}},\nabla\cdot\xi^{n}_{\mathbf{B}}\right)\right|+\zeta\left|\left(\nabla\cdot\psi^{n}_{\mathbf{u}},\nabla\cdot\xi^{n}_{\mathbf{u}}\right)\right|\leq\frac{\chi}{2}\lVert \nabla\cdot\xi^{n}_{\mathbf{B}}\rVert ^{2}+\frac{\zeta}{2}\lVert \nabla\cdot\xi^{n}_{\mathbf{u}}\rVert ^{2}+C\left(\lVert \nabla\psi^{n}_{\mathbf{B}}\rVert ^{2}+\lVert \nabla\psi^{n}_{\mathbf{u}}\rVert ^{2}\right),
\end{alignat*}
\[
\left|\left(\mathbf{r}^{n}_{1},\xi^{n}_{\mathbf{u}}\right)\right|+\left|\left(\mathbf{r}^{n}_{2},\xi^{n}_{\mathbf{B}}\right)\right|\leq\frac{\epsilon}{\mathrm{Re}}\lVert \nabla\xi^{n}_{\mathbf{u}}\rVert ^{2}+\frac{\epsilon}{\mathrm{Rm}}\lVert \nabla\xi^{n}_{\mathbf{B}}\rVert ^{2}
+C\left(\lVert \mathbf{r}^{n}_{1}\rVert ^{2}+\lVert \mathbf{r}^{n}_{2}\rVert ^{2}\right).
\]

Using the estimates derived above in (\ref{eq:conv_step2}), and choosing
$\epsilon>0$ sufficiently small, we collect all terms involving $\lVert \nabla\xi^{n}_{\mathbf{u}}\rVert ^{2}$
and $\lVert \nabla\xi^{n}_{\mathbf{B}}\rVert ^{2}$ on the
left-hand side to obtain
\[
\frac{1}{2}\Delta^{\alpha\left(t_{n}\right)}_{\tau}E_{n}+c_{0}D_{n}+\frac{\chi}{2}\lVert \nabla\cdot\xi^{n}_{\mathbf{B}}\rVert ^{2}+\frac{\zeta}{2}\lVert \nabla\cdot\xi^{n}_{\mathbf{u}}\rVert ^{2}\leq I_{1}+I_{2},
\]
where $c_{0}=\frac{1}{2}\min\left\{ \frac{1}{\mathrm{Re}},\frac{1}{\mathrm{Rm}}\right\}$, and
\begin{alignat*}{1}
 & I_{1}=C\left(1+h^{-1}\lVert \nabla\psi^{n}_{\mathbf{u}}\rVert ^{2}+h^{-1}\lVert \nabla\psi^{n}_{\mathbf{B}}\rVert ^{2}\right)\left(  \lVert \xi^{n}_{\mathbf{u}}\rVert ^{2} + \lVert \xi^{n}_{\mathbf{B}}\rVert ^{2} \right) ,\\
 & I_{2}=C\left(\lVert \nabla\psi^{n}_{\mathbf{u}}\rVert ^{2}+\lVert \nabla\psi^{n}_{\mathbf{B}}\rVert ^{2}
 +\lVert \nabla\psi^{n}_{\mathbf{u}}\rVert ^{4}
 +\lVert \nabla\psi^{n}_{\mathbf{B}}\rVert ^{4}
 +\lVert \mathbf{r}^{n}_{1}\rVert ^{2}+\lVert \mathbf{r}^{n}_{2}\rVert ^{2}\right).
\end{alignat*}

Next, invoking the projection estimates (\ref{eq:approximation_errors}) as well as (\ref{eq:semi_r}) and (\ref{eq:frac_replacement}),
we infer that
\[
I_{1}+I_{2}\leq C_{1}E_{n}+C\left(h^{2l}+h^{2m}+\tau^{2}\right)
\]
for $0<h\leq 1$, where $C_{1}$ is independent of $h$, $\tau$ and $n$. Therefore,
\begin{equation}
\frac{1}{2}\Delta^{\alpha\left(t_{n}\right)}_{\tau}E_{n}+c_{0}D_{n}\leq C_{1}E_{n}+R,\label{eq:conv_step5}
\end{equation}
where $R=C\left(h^{2l}+h^{2m}+\tau^{2}\right)$. 

Using the definition of the coefficients $b^{\left(\alpha\right)}_{n,k}$,
one checks directly that
\[
\Delta^{\alpha\left(t_{n}\right)}_{\tau}E_{n}=\sum^{n}_{k=1}A^{\left(n\right)}_{n-k}\nabla_{\tau}E_{k},\qquad A^{\left(n\right)}_{n-k}=\frac{1}{\tau}b^{\left(\alpha\right)}_{n,k},\qquad\nabla_{\tau}E_{k}=E_{k}-E_{k-1}.
\]

Hence it follows from (\ref{eq:conv_step5}) that
\begin{equation}
\sum^{n}_{k=1}A^{\left(n\right)}_{n-k}\nabla_{\tau}E_{k}\leq2C_{1}E_{n}+2R.\label{eq:conv_step6}
\end{equation}

By setting $v^{n}=E_{n}$, $\theta=0$, $g^{n}=2R$, and choosing
$\lambda_{0}=2C_{1}$, $\lambda_{s}=0$, $1\leq s\leq N-1$, Eq. (\ref{eq:conv_step6})
can be written as
\[
\sum^{n}_{k=1}A^{\left(n\right)}_{n-k}\nabla_{\tau}v^{k}\leq\sum^{n}_{k=1}\lambda_{n-k}v^{k}+g^{n},\qquad1\leq n\leq N.
\]

Thus (\ref{eq:conv_step6}) has the form required by Theorem 3.2
of \cite{Liao_2019} (see Appendix A) with
\[
\Lambda\geq\sum^{N-1}_{s=0}\lambda_{s}=2C_{1}.
\]

Taking $\Lambda=2C_{1}$, and using Lemma \ref{lem:check_frac_assump},
Assumptions A1--A3 in \cite{Liao_2019} hold with exponent $\alpha_{*}$,
constant $\pi_{A}$, and $\rho=1$. Therefore the step-size condition
in Theorem 3.2 of \cite{Liao_2019} becomes
\[
\tau\leq\left(2\pi_{A}\Gamma\left(2-\alpha_{*}\right)\Lambda\right)^{-1/\alpha_{*}}=\left(4\pi_{A}\Gamma\left(2-\alpha_{*}\right)C_{1}\right)^{-1/\alpha_{*}},
\]
and Theorem 3.2 gives
\[
E_{n}\leq2\mathbb{E}_{\alpha_{*}}\left(4\pi_{A}C_{1}t^{\alpha_{*}}_{n}\right)\left(E_{0}+\max_{1\leq k\leq n}\sum^{k}_{j=1}2\mathcal{P}^{\left(k\right)}_{k-j}R\right),
\]
where $\mathbb{E}_{\alpha_{*}}$ is the Mittag--Leffler function and,
for each $1\leq z\leq N$, 
$\{\mathcal{P}^{\left(z\right)}_{z-j}\}_{j=1}^{z}$ are the complementary kernels
associated with the discrete kernels
$\{A^{\left(z\right)}_{z-k}\}_{k=1}^{z}$.
By Remark 1 following Theorem 3.1 in \cite{Liao_2019}, we have for $1\leq k\leq n$,
\[
\sum^{k}_{j=1}2\mathcal{P}^{\left(k\right)}_{k-j}R\leq2\pi_{A}\Gamma\left(1-\alpha_{*}\right)T^{\alpha_{*}}R.
\]

Consequently,
\[
E_{n}\leq2\mathbb{E}_{\alpha_{*}}\left(4\pi_{A}C_{1}t^{\alpha_{*}}_{n}\right)\left(E_{0}+2\pi_{A}\Gamma\left(1-\alpha_{*}\right)T^{\alpha_{*}}R\right).
\]

Therefore, using $t_{n}\leq T$, we obtain
\[
E_{n}\leq C\left(E_{0}+R\right),\qquad1\leq n\leq N,
\]
and considering that $E_{0}\leq C\left(h^{2l}+h^{2m}\right)$, we arrive at the assertion of the theorem.
\end{proof}

\subsection{Implementation of the Fully Discrete Scheme}

For the numerical implementation, the fully implicit scheme (\ref{eq:fully_u1})--(\ref{eq:fully_u2}) is solved at each time step by a Picard linearization initialized with the solution from the previous time step. This leads, at every Picard iteration, to a monolithic linear system \cite{Planas2011} for the coupled velocity--pressure--magnetic-field unknowns, which is solved by restarted FGMRES with a block preconditioner.

In addition to the stabilization terms included in the discrete formulation, the numerical implementation uses a divergence-cleaning step for the magnetic field after each time step.
More precisely, after solving Problem \ref{prob:fully} at time level $t_{n}$, we obtain $\left(\mathbf{u}_{h}^{n},{p}_{h}^{n},\widetilde{\mathbf{B}}_{h}^{n}\right)$, where $\widetilde{\mathbf{B}}_{h}^{n}$ denotes the magnetic field before the cleaning step. We then solve the auxiliary elliptic problem to find $\phi_{h}^{n}\in \mathcal{Y}_{h}$ such that
\begin{equation}
\left(\nabla \phi_{h}^{n},\nabla \psi_{h}\right)=-\left(\nabla\cdot \widetilde{\mathbf{B}}_{h}^{n},\psi_{h}\right),
\label{eq:elliptic}
\end{equation}
for all $\psi_{h}\in \mathcal{Y}_{h}$, where $\mathcal{Y}_{h}=\left\{\psi\in H^{1}\left(\Omega\right):\int_{\Omega} \psi\,d\mathbf{x}=0\right\}$, and then set
\begin{equation}
\mathbf{B}_{h}^{n}=\widetilde{\mathbf{B}}_{h}^{n}-\nabla \phi_{h}^{n}.
\label{eq:recalculated_B}
\end{equation}

In the periodic test cases, ({\ref{eq:elliptic}) is solved with periodic boundary conditions, and the zero-mean condition ensures uniqueness of $\phi_{h}^{n}$. In the cases with Dirichlet boundary conditions for the magnetic field, the correction (\ref{eq:recalculated_B}) does not in general preserve the prescribed boundary values. Therefore, after the cleaning step, the Dirichlet values are imposed again before proceeding to the next time level.

\section{Results}

\label{sec:Results}

This section presents a set of numerical tests to assess 
the correctness of the numerical scheme and
the effect of variable-order
fractional time derivatives on MHD flow. Although the formulation allows different variable orders
in the momentum and induction equations, we set
$\alpha\left(t\right)=\beta\left(t\right)$ in most of the numerical experiments 
to reduce the parameter space.
The case $\alpha\left(t\right)\ne\beta\left(t\right)$ is studied in Section \ref{subsec:unsymmetric}.

\subsection{Convergence of the Numerical Scheme}\label{subsec:convergence}

The first numerical experiment is designed to verify the temporal
convergence of the proposed scheme. 
We consider the system
(\ref{eq:nondim_u1})--(\ref{eq:nondim_B2}) on $\Omega=\left(0,1\right)^{2}$
over the time interval $\left(0,T\right]$ with $T=1$. We take the manufactured
solution
\begin{alignat}{1}
 & \mathbf{u}\left(\mathbf{x},t\right)=\left(\begin{array}{c}
t^{4}\left(1-x_{1}\right)^{2}x_{1}^{2}\left(4x_{2}^{3}-6x_{2}^{2}+2x_{2}\right)\\
-t^{4}\left(4x_{1}^{3}-6x_{1}^{2}+2x_{1}\right)\left(1-x_{2}\right)^{2}x_{2}^{2}
\end{array}\right),\qquad p\left(\mathbf{x},t\right)=0,
\qquad \mathbf{x}=\left(x_{1},x_{2}\right),
\label{eq:ex1_equ}\displaybreak[0]\\
 & \mathbf{B}\left(\mathbf{x},t\right)=\left(\begin{array}{c}
t^{3}\left(1-x_{1}\right)^{2}x_{1}^{2}\left(4x_{2}^{3}-6x_{2}^{2}+2x_{2}\right)\\
-t^{3}\left(4x_{1}^{3}-6x_{1}^{2}+2x_{1}\right)\left(1-x_{2}\right)^{2}x_{2}^{2}
\end{array}\right)
\label{eq:ex1_eqB}
\end{alignat} 
and three representative variable-order profiles.
In all three cases, the fractional orders in the velocity and induction equations
are taken to be the same.

\textit{Case 1.} A linearly increasing variable order
\[
{\displaystyle \alpha\left(t\right)=\beta\left(t\right)=\alpha_{0}+\left(\alpha_{1}-\alpha_{0}\right)\frac{t}{T}},\qquad\left(\alpha_{0},\alpha_{1}\right)=\left(0.6,0.95\right),\quad t\in\left[0,T\right]
\]
with homogeneous Dirichlet boundary conditions for both the velocity
and the magnetic field.

\textit{Case 2.} A periodically varying order
\[
\alpha\left(t\right)=\beta\left(t\right)=\overline{\alpha}+A_{0}\sin\left(\frac{2\pi t}{P_{0}}\right),\qquad
\overline{\alpha}=0.75,\quad
A_{0}=0.2,\quad
P_{0}=\frac{T}{2}
\]
with periodic boundary conditions for both the velocity
and the magnetic field.

\textit{Case 3.} A smoothly varying order
\[
\alpha\left(t\right)=\beta\left(t\right)=\alpha_{0}+\frac{1}{2}\left(\alpha_{1}-\alpha_{0}\right)
\left(1+\mathrm{tanh}\left(\frac{t-0.4}{0.05}\right)\right),
\qquad\left(\alpha_{0},\alpha_{1}\right)=\left(0.6,0.95\right),
\]
which models a rapid transition of the fractional order around $t\approx0.4$, with periodic boundary conditions for both the velocity
and the magnetic field.

We study temporal convergence by refining the uniform time step $\tau$
while keeping the spatial mesh fixed at $h\approx0.003535$. The parameters
in the numerical scheme are chosen as follows: $\mathrm{Re}=\mathrm{Rm}=1$,
$\zeta=0.5$, $\chi=0.125$.
The error is measured in the discrete norm
\[
\lVert \mathbf{e}\rVert _{L^{\infty}\left(0,T;\mathbf{L}^{2}\left(\Omega\right)\right)}=\max_{0\leq n\leq N}\lVert \mathbf{e}^{n}\rVert _{\mathbf{L}^{2}\left(\Omega\right)},
\]
and the observed temporal error is computed by the standard ratio
\[
r=\log_{2}\frac{E\left(\tau\right)}{E\left(\tau/2\right)}, 
\]
where $E\left(\tau\right)$
denotes the error of the solution, obtained with the time step $\tau$, 
in the norm defined above.

All runs in Section \ref{sec:Results} use the same spatial discretization: Taylor-Hood
elements $P_{2}/P_{1}$ for the 2D velocity-pressure pair $\left(\mathbf{u},p\right)$,
and continuous piecewise-quadratic $\left(P_{2}\right)$ elements
for the magnetic field $\mathbf{B}$.
In the computations, the Picard iteration is initialized with the solution from the previous time level 
and stopped when the difference between two successive iterates becomes smaller than $10^{-10}$. 
In the reported tests, convergence is achieved within 2--3 iterations. 
This is likely due to the combination of a smooth exact solution, 
relatively weak nonlinearity, and the use of the previous time level as the initial guess.

The results in Table \ref{tab:convtest1} show that the proposed variable-order
fractional MHD scheme exhibits a consistent first-order temporal convergence
for both $\mathbf{u}$ and $\mathbf{B}$. The measured orders stay
close to one (approximately $1.01-1.10$ for $\mathbf{u}$ and $1.03-1.05$
for $\mathbf{B}$) as $\tau$ decreases from $1/10$ to $1/320$.
This behavior agrees well with the theoretical prediction obtained in Theorem \ref{thm:convergence}.
Similar conclusions hold for Cases 2 and 3, presented in Table \ref{tab:convtest2} and Table \ref{tab:convtest3}, respectively: in both cases, the errors decrease consistently with time-step refinement, and the computed orders remain close to one.

\begin{table}[H]
\caption{Convergence analysis for Case 1.}
\label{tab:convtest1}
\begin{tabular*}{\textwidth}{@{\extracolsep{\fill}}ccccc@{}}
\hline
$\mathbf{\tau}$ & $\lVert \mathbf{u}-\mathbf{u}_{h}\rVert _{L^{\infty}\left(0,T;\mathbf{L}^{2}\left(\Omega\right)\right)}$ & \textbf{Order} & $\lVert \mathbf{B}-\mathbf{B}_{h}\rVert _{L^{\infty}\left(0,T;\mathbf{L}^{2}\left(\Omega\right)\right)}$ & \textbf{Order}\\
\hline
1/10 & $6.5261\times10^{-5}$ & -- & $3.5554\times10^{-5}$ & --\\
1/20 & $3.2445\times10^{-5}$ & 1.01 & $1.7411\times10^{-5}$ & 1.03\\
1/40 & $1.5825\times10^{-5}$ & 1.04 & $8.4499\times10^{-6}$ & 1.04\\
1/80 & $7.6090\times10^{-6}$ & 1.06 & $4.0834\times10^{-6}$ & 1.05\\
1/160 & $3.6152\times10^{-6}$ & 1.07 & $1.9704\times10^{-6}$ & 1.05\\
1/320 & $1.6855\times10^{-6}$ & 1.10 & $9.5280\times10^{-7}$ & 1.05\\
\hline
\end{tabular*}
\end{table}

\begin{table}[H]
\caption{Convergence analysis for Case 2.}
\label{tab:convtest2}
\begin{tabular*}{\textwidth}{@{\extracolsep{\fill}}ccccc@{}}
\hline
$\mathbf{\tau}$ & $\lVert \mathbf{u}-\mathbf{u}_{h}\rVert _{L^{\infty}\left(0,T;\mathbf{L}^{2}\left(\Omega\right)\right)}$ & \textbf{Order} & $\lVert \mathbf{B}-\mathbf{B}_{h}\rVert _{L^{\infty}\left(0,T;\mathbf{L}^{2}\left(\Omega\right)\right)}$ & \textbf{Order}\\
\hline
1/20 & $1.4530\times10^{-5}$ & -- & $1.1774\times10^{-5}$ & --\\
1/40 & $7.2694\times10^{-6}$ & 1.00 & $5.8258\times10^{-6}$ & 1.02\\
1/80 & $3.5998\times10^{-6}$ & 1.01 & $2.8495\times10^{-6}$ & 1.03\\
1/160 & $1.7515\times10^{-6}$ & 1.04 & $1.3835\times10^{-6}$ & 1.04\\
1/320 & $8.4703\times10^{-7}$ & 1.05 & $6.6896\times10^{-7}$ & 1.05\\
\hline
\end{tabular*}
\end{table}

\begin{table}[H]
\caption{Convergence analysis for Case 3.}
\label{tab:convtest3}
\begin{tabular*}{\textwidth}{@{\extracolsep{\fill}}ccccc@{}}
\hline
$\mathbf{\tau}$ & $\lVert \mathbf{u}-\mathbf{u}_{h}\rVert _{L^{\infty}\left(0,T;\mathbf{L}^{2}\left(\Omega\right)\right)}$ & \textbf{Order} & $\lVert \mathbf{B}-\mathbf{B}_{h}\rVert _{L^{\infty}\left(0,T;\mathbf{L}^{2}\left(\Omega\right)\right)}$ & \textbf{Order}\\
\hline
1/10 & $7.6527\times10^{-5}$ & -- & $4.0459\times10^{-5}$ & --\\
1/20 & $3.8289\times10^{-5}$ & 1.00 & $1.9917\times10^{-5}$ & 1.02\\
1/40 & $1.8819\times10^{-5}$ & 1.02 & $9.7121\times10^{-6}$ & 1.04\\
1/80 & $9.1632\times10^{-6}$ & 1.04 & $4.7137\times10^{-6}$ & 1.04\\
1/160 & $4.4372\times10^{-6}$ & 1.05 & $2.2827\times10^{-6}$ & 1.05\\
1/320 & $2.1399\times10^{-6}$ & 1.05 & $1.1050\times10^{-6}$ & 1.05\\
\hline
\end{tabular*}
\end{table}

We also examine the behavior of the method under spatial mesh refinement with the fixed time step $\tau=1/1000$. In all three cases, the numerical results are consistent with the expected second-order spatial convergence: the observed orders are higher on the coarser meshes but approach the asymptotic value 2 as the mesh is refined.

In addition to the convergence orders, we monitor the divergence norms in these tests. Figure~\ref{fig:conv_div_errors} shows the time evolution of $\lVert\nabla\cdot\mathbf{u}_{h}\rVert$ and $\lVert\nabla\cdot\mathbf{B}_{h}\rVert$ for Cases~1--3. In all cases, the velocity divergence remains low, and the cleaning step 
substantially reduces the magnetic divergence over the whole time interval. 
These results indicate that the divergence errors are well controlled in the convergence tests.
This behavior is consistent with the use of $H^{1}$-conforming finite element spaces, for which the constraints $\nabla\cdot\mathbf{u}=0$ and $\nabla\cdot\mathbf{B}=0$ are generally not satisfied exactly at the discrete level. Accordingly, one expects small but nonzero values of $\lVert\nabla\cdot\mathbf{u}_{h}\rVert$ and $\lVert\nabla\cdot\mathbf{B}_{h}\rVert$ in the numerical solution \cite{Wacker2016,BeirodaVeiga2025,John2017}.

\begin{figure}[H]
  \centering
  \includegraphics[width=\textwidth]{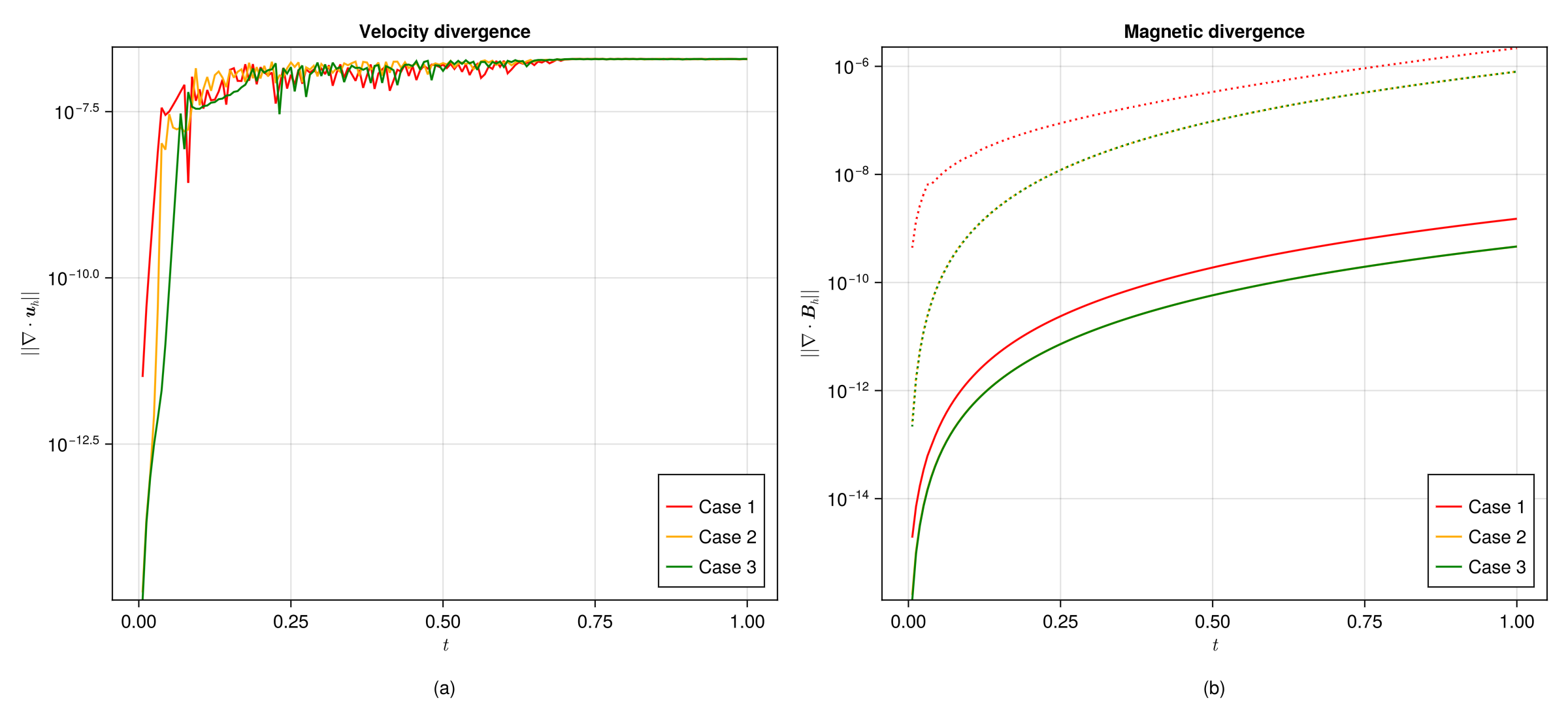}
  \caption{Time evolution of the divergence norms in the convergence tests: 
  \textbf{(a)}~$\lVert\nabla\cdot\mathbf{u}_{h}\rVert$; 
  \textbf{(b)}~$\lVert\nabla\cdot\mathbf{B}_{h}\rVert$, where the dotted and solid lines represent the values before and after the cleaning step, respectively. In panel \textbf{(b)}, the results for Cases~2 and 3 almost coincide.}
\label{fig:conv_div_errors}
\end{figure}

\subsection{Consistency of the Variable-order Fractional MHD Model with the Classical
MHD Model}

The second numerical experiment verifies that the variable-order
fractional MHD model under study is consistent with the classical (integer-order)
MHD model in the limit when the fractional orders $\alpha\left(t\right)$
and $\beta\left(t\right)$ approach one. To this end, we consider
a one-parameter family of variable orders
\[
\alpha\left(t\right)=\beta\left(t\right)=1-\varepsilon\left(1-\frac{t}{T}\right)-\delta,\qquad\delta=10^{-10},
\]
so that $\alpha\left(t\right),\beta\left(t\right)\in\left(0,1\right)$
for all $t\in\left[0,T\right]$ and $\alpha\left(t\right)\to1$ uniformly
in time as $\varepsilon\to0$. 
We note that the theoretical analysis is carried out under Assumption \ref{assu:orders_bounds},
which requires the fractional orders to be bounded away from 1. Hence, the limits $\alpha\left(t\right)$,
$\beta\left(t\right)\to 1$ are not covered by Theorem \ref{thm:convergence} and are examined here only
numerically.

We compute solutions $\left(\mathbf{u}_{\varepsilon},p_{\varepsilon},\mathbf{B}_{\varepsilon}\right)$
on $\Omega=\left(0,1\right)^{2}$ over the time interval $\left[0,T\right]$
with $T=0.5$ starting from the initial values
\[
\begin{array}{l}
{\displaystyle \mathbf{u}_{0}\left(\mathbf{x}\right)=\left(\begin{array}{c}
2\pi\sin^{2}\left(\pi x_{1}\right)\sin\left(\pi x_{2}\right)\cos\left(\pi x_{2}\right)\\
-2\pi\sin\left(\pi x_{1}\right)\cos\left(\pi x_{1}\right)\sin^{2}\left(\pi x_{2}\right)
\end{array}\right),}\displaybreak[0]\\
\mathbf{B}_{0}\left(\mathbf{x}\right)=\left(\begin{array}{c}
0.8\pi\sin^{2}\left(\pi x_{1}\right)\sin\left(2\pi x_{2}\right)\cos\left(2\pi x_{2}\right)\\
-0.4\pi\sin\left(\pi x_{1}\right)\cos\left(\pi x_{1}\right)\sin^{2}\left(2\pi x_{2}\right)
\end{array}\right)
\end{array}
\]
with homogeneous Dirichlet conditions for a sequence of decreasing
parameters $\varepsilon\in\left\{ 0.1\right.$, $0.05$, $0.03$, $0.02$,
$0.01$, $0.005$, $10^{-3}$, $10^{-4}$, $10^{-5}$, $\left.10^{-6}\right\} $, and compare
them against a reference integer-order solution $\left(\mathbf{u}_{1},p_{1},\mathbf{B}_{1}\right)$
obtained by setting $\alpha\left(t\right)=\beta\left(t\right)=1$. 
The Reynolds numbers are chosen as $\mathrm{Re}=\mathrm{Rm}=200$, 
and stabilization parameters and the Picard iteration tolerance 
are selected as in Section \ref{subsec:convergence}. 
In this test, the prescribed Picard tolerance was reached within 4--6 iterations.

In this and the following tests, the classical (integer-order) incompressible MHD system is solved with the same finite element spatial discretization and the corresponding fully implicit first-order time discretization. The nonlinear systems are solved by the same Picard iteration procedure as in the fractional case.

To quantify convergence toward the integer-order model, we introduce the time-dependent errors
\[
E_{u}^{(\varepsilon)}\left(t\right)=\lVert\mathbf{u}_{\varepsilon}\left(t\right)-\mathbf{u}_{1}\left(t\right)\rVert_{\mathbf{L}^{2}(\Omega)},
\qquad
E_{B}^{(\varepsilon)}\left(t\right)=\lVert\mathbf{B}_{\varepsilon}\left(t\right)-\mathbf{B}_{1}\left(t\right)\rVert_{\mathbf{L}^{2}(\Omega)},
\]
and also examine the energy differences
\[
|K_{\varepsilon}\left(t\right)-K_{1}\left(t\right)|,
\qquad
|M_{\varepsilon}\left(t\right)-M_{1}\left(t\right)|,
\]
where $K\left(t\right)$ and $M\left(t\right)$ are the kinetic and magnetic energies, respectively:
\begin{equation}
K\left(t\right)=\frac{1}{2}\int_{\Omega}\left|\mathbf{u}\left(\mathbf{x},t\right)\right|^{2}d\mathbf{x},\qquad M\left(t\right)=\frac{1}{2}\int_{\Omega}\left|\mathbf{B}\left(\mathbf{x},t\right)\right|^{2}d\mathbf{x},\label{eq:energy}
\end{equation}
and subscripts $\varepsilon$ and $1$ denote a corresponding diagnostic computed from the solution obtained with a given $\varepsilon$ and the classical MHD, respectively.

Figure \ref{fig:limit_test_compare_solutions} shows that both $E_{u}^{(\varepsilon)}\left(t\right)$
and $E_{B}^{(\varepsilon)}\left(t\right)$ decrease monotonically as $\varepsilon$ decreases.
For larger $\varepsilon$, the deviations from the integer-order solution grow rapidly
at early times and remain noticeable over the interval shown. As $\varepsilon$ becomes
smaller, the curves are shifted downward over the whole time interval, and for
$\varepsilon=10^{-6}$ the discrepancies are reduced to the level of about
$10^{-7}$--$10^{-5}$. This behavior is consistent with convergence of the fractional
solutions to the integer-order reference as $\varepsilon\to 0$. The logarithmic insets
further show the systematic reduction of the errors across several orders of magnitude.

\begin{figure}[H]
\centering
\includegraphics[width=0.98\textwidth]{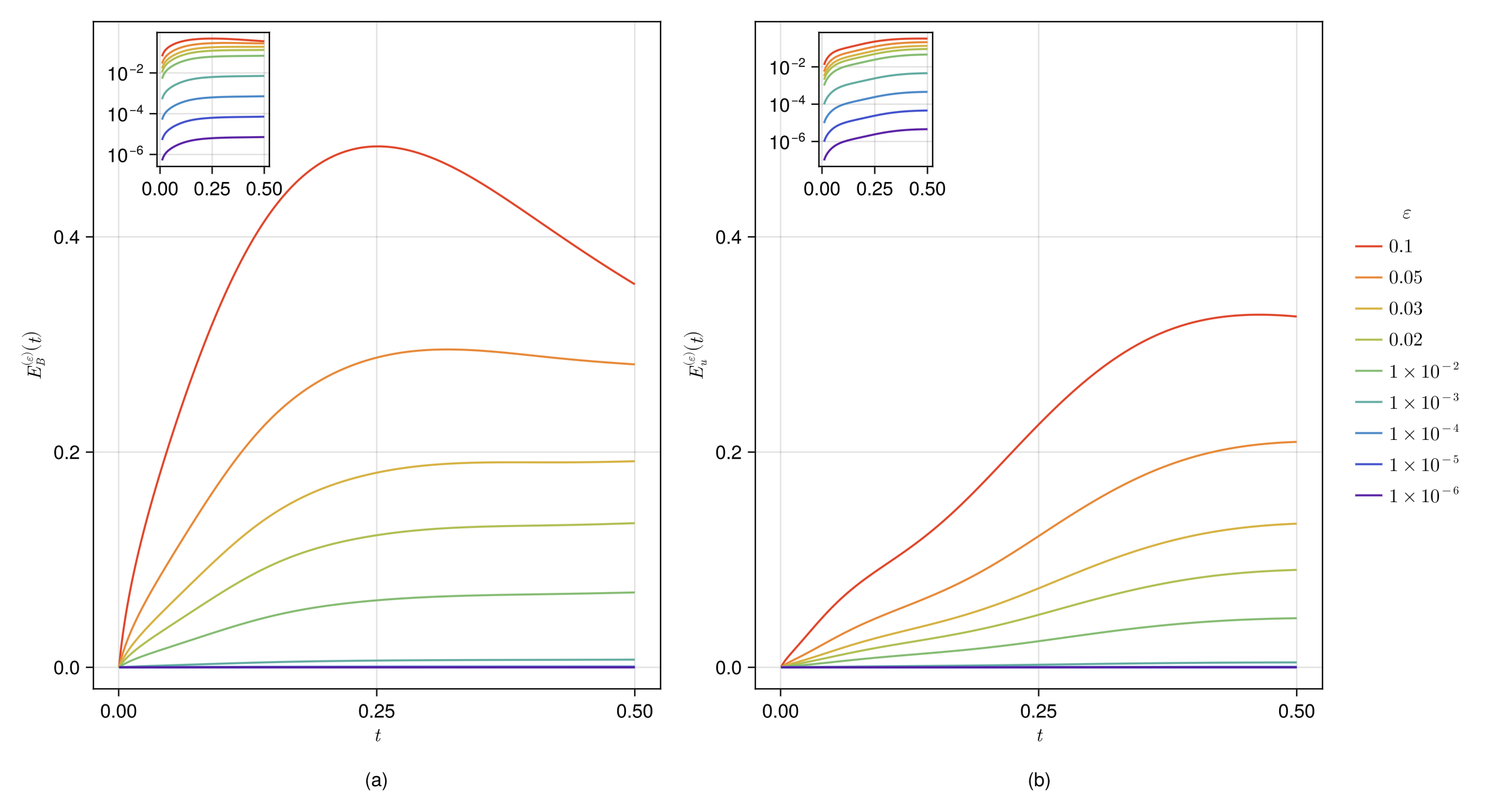}
\caption{Time evolution of the solution discrepancies between the classical model and
the fractional models defined by the variable-order law $\alpha\left(t\right)=\beta\left(t\right)=1-\varepsilon\left(1-\frac{t}{T}\right)-\delta$
with $\delta=10^{-10}$ for various $\varepsilon$:
\textbf{(a)}~$E_{B}^{\left(\varepsilon\right)}\left(t\right)=\lVert \mathbf{B}_{\varepsilon}\left(t\right)-\mathbf{B}_{1}\left(t\right)\rVert _{\mathbf{L}^{2}\left(\Omega\right)}$;
\textbf{(b)}~$E_{u}^{\left(\varepsilon\right)}\left(t\right)=\lVert \mathbf{u}_{\varepsilon}\left(t\right)-\mathbf{u}_{1}\left(t\right)\rVert _{\mathbf{L}^{2}\left(\Omega\right)}$.
The insets show the same curves on a logarithmic scale.}\label{fig:limit_test_compare_solutions}
\end{figure}

Figure \ref{fig:limit_test_compare_energy} shows the time evolution of
$\left|K_{\varepsilon}(t)-K_{1}\left(t\right)\right|$ and
$\left|M_{\varepsilon}(t)-M_{1}\left(t\right)\right|$. In both panels, the
discrepancies decrease as $\varepsilon$ becomes smaller. For the magnetic
energy, the curves remain ordered over the whole interval, with smaller
$\varepsilon$ giving uniformly smaller values of
$\left|M_{\varepsilon}\left(t\right)-M_{1}\left(t\right)\right|$. For the kinetic energy, the
same overall trend is observed, although the curves pass through values
close to zero near the middle of the interval. The logarithmic insets
make the reduction more visible and show that the energy differences
decrease by several orders of magnitude as $\varepsilon\to0$.

These results provide numerical evidence that, as $\alpha\left(t\right),\beta\left(t\right)\to1$,
the variable-order fractional MHD solutions converge to the classical
MHD solution in both state variables and energies, showing that the proposed 
formulation and discretization correctly recover the classical limit.

\begin{figure}[H]
\centering
\includegraphics[width=0.98\textwidth]{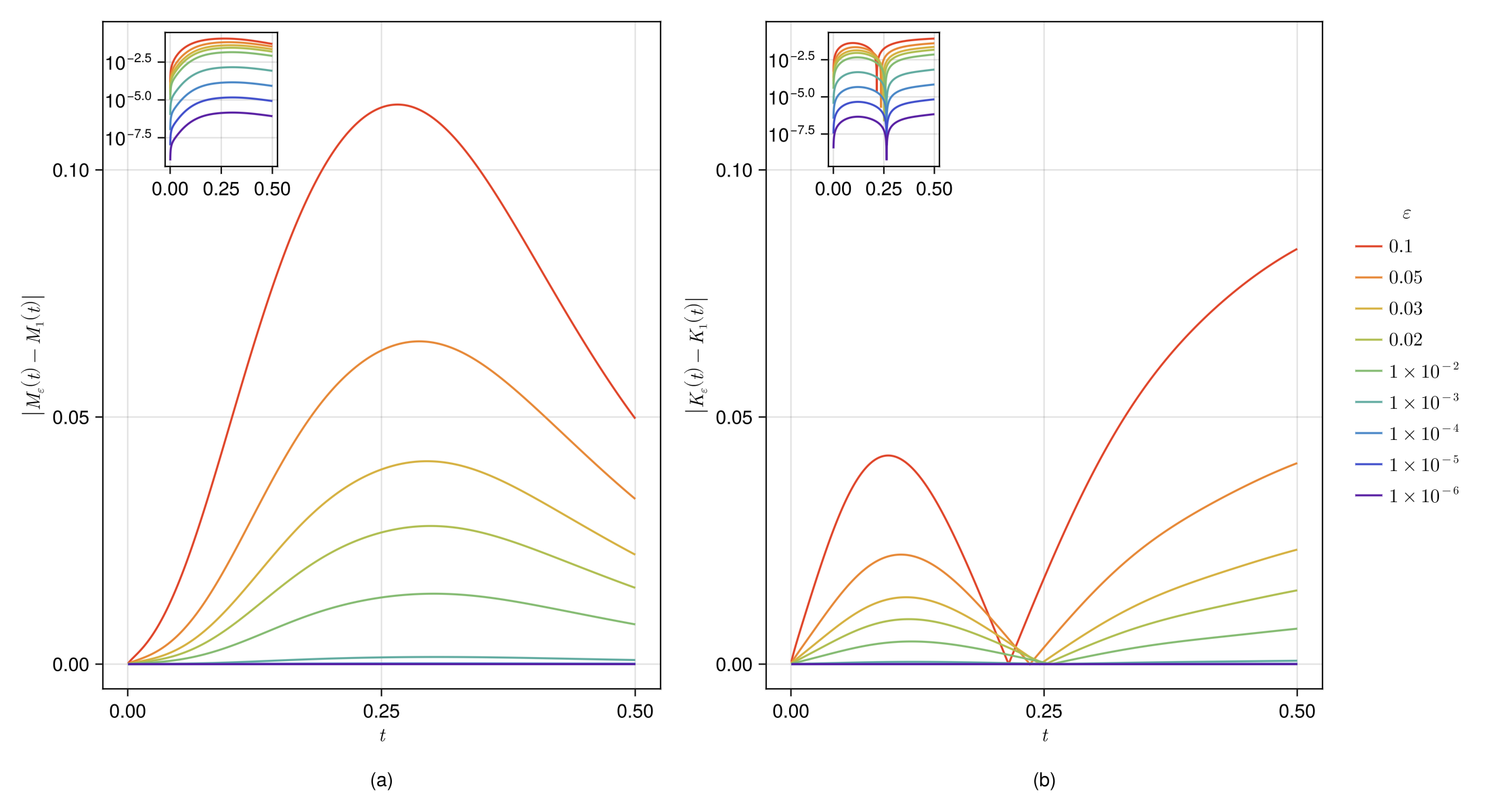}
\caption{Time evolution of energy discrepancies between the classical model and
the fractional models defined by the variable-order law $\alpha\left(t\right)=\beta\left(t\right)=1-\varepsilon\left(1-\frac{t}{T}\right)-\delta$, $\delta=10^{-10}$
for various $\varepsilon$:
\textbf{(a)}~$\left|M_{\varepsilon}\left(t\right)-M_{1}\left(t\right)\right|$;
\textbf{(b)}~$\left|K_{\varepsilon}\left(t\right)-K_{1}\left(t\right)\right|$. 
The insets show the same curves on a logarithmic
scale.}\label{fig:limit_test_compare_energy}
\end{figure}

\subsection{Impact of the Variable-order Fractional Derivatives on the MHD Flow}
\label{subsec:Orszag}

In the third experiment, we study how the order of the Caputo time-fractional
derivative affects the evolution of an MHD flow. We consider the fractional-order
periodic divergence-free vortex test in $\Omega=\left(0,1\right)^{2}$,
in which the initial values are defined as 
\[
\mathbf{u}_{0}\left(\mathbf{x}\right)=\left(\begin{array}{c}
\sin\left(2\pi x_{1}\right)\cos\left(2\pi x_{2}\right)\\
-\cos\left(2\pi x_{1}\right)\sin\left(2\pi x_{2}\right)
\end{array}\right),\qquad\mathbf{B}_{0}\left(\mathbf{x}\right)=\left(\begin{array}{c}
2\sin\left(8\pi x_{1}\right)\cos\left(8\pi x_{2}\right)\\
-2\cos\left(8\pi x_{1}\right)\sin\left(8\pi x_{2}\right)
\end{array}\right),
\]
whereas periodic boundary conditions are imposed on both velocity
and magnetic field. The remaining parameters are defined as follows:
$T=0.16$, $\mathrm{Re}=\mathrm{Rm}=300$, 
$h\approx0.009428$, $\tau=0.0002$.
In all numerical tests presented in Section \ref{subsec:Orszag}, the Picard iteration tolerance 
was set to $10^{-10}$, and this level was reached within 3--5 iterations.

To assess both constant-order and time-dependent memory effects, we
consider the following representative order functions $\alpha\left(t\right)=\beta\left(t\right)$,
all satisfying $\alpha\left(t\right)\in\left(0,1\right)$ for $t\in\left[0,T\right]$:

\textit{Case 1.} Constant order:
\[
\alpha\left(t\right)\equiv\alpha_{c},\qquad\alpha_{c}\in\left\{ 0.6,0.75,0.9\right\} .
\]

\textit{Case 2.} Linear ramp:
\[
\alpha\left(t\right)=\alpha_{0}+\left(\alpha_{1}-\alpha_{0}\right)\frac{t}{T},\qquad\alpha_{0}=0.9,\,\,\alpha_{1}=0.6.
\]

\textit{Case 3.} Step change:
\[
\alpha\left(t\right)=\left\{ \begin{array}{ll}
\alpha_{1}, & t<t_{s},\\
\alpha_{2}, & t\geq t_{s},
\end{array}\right.\qquad\alpha_{1}=0.9,\,\,\alpha_{2}=0.65,\,\,t_{s}=\frac{T}{2}.
\]

\textit{Case 4.} Periodic modulation (sinusoidal profile):
\[
\alpha\left(t\right)=\overline{\alpha}+A_{0}\sin\left(\frac{2\pi t}{P_{0}}\right),\qquad\overline{\alpha}=0.75,\,\,A_{0}=0.15,\,\,P_{0}=\frac{T}{2}.
\]

\textit{Case 5.} Smooth step:
\[
\alpha\left(t\right)=\alpha_{0}+\frac{1}{2}\left(\alpha_{1}-\alpha_{0}\right)\left(1+\mathrm{tanh}\left(\frac{t-t_{s}}{\varepsilon}\right)\right),\qquad\alpha_{0}=0.8,\,\,\alpha_{1}=0.5,\,\,t_{s}=\frac{2T}{5},\,\,\varepsilon=0.05.
\]

Here we include one case with a step-type order profile (Case 3) in order to examine the practical robustness of the numerical method in the presence of an abrupt change of memory intensity. This example lies outside the assumptions of the analysis and is therefore presented as an empirical robustness test rather than as a verification of the theoretical results.

All fractional and variable-order cases are compared against the classical
MHD model obtained for $\alpha\left(t\right)\equiv1$, $\beta\left(t\right)\equiv1$. For each case,
we monitor the kinetic energy $K\left(t\right)$ and magnetic energy
$M\left(t\right)$, defined in (\ref{eq:energy}), as well as the
enstrophy $Z\left(t\right)$ and current enstrophy $J\left(t\right)$,
given by
\[
Z\left(t\right)=\frac{1}{2}\int_{\Omega}\left|\nabla\times\mathbf{u}\left(\mathbf{x},t\right)\right|^{2}d\mathbf{x}\qquad\mathrm{and}\qquad J\left(t\right)=\frac{1}{2}\int_{\Omega}\left|\nabla\times\mathbf{B}\left(\mathbf{x},t\right)\right|^{2}d\mathbf{x}.
\]

\subsubsection{Choosing Stabilization Parameters}

Before analyzing the influence of the variable-order profiles on the computed MHD dynamics, we first specify the stabilization parameters, the grad-div parameter $\zeta$ and the magnetic divergence-penalty parameter $\chi$, used in the simulations. To choose suitable values, we perform a sensitivity study for the linear ramp profile (Case 2), which is taken here as a representative case.

Figure \ref{fig:orszag_choosing_stabilization} shows the time evolution of the divergence norms in this study. In panel (a), the parameter $\chi$ is fixed at $\chi=1$, and we examine the effect of $\zeta$ on $\left\Vert \nabla\cdot\mathbf{u}_{h}\right\Vert$. As $\zeta$ increases, the values of $\left\Vert \nabla\cdot\mathbf{u}_{h}\right\Vert$ decrease over the whole time interval. This decrease is clearly visible up to $\zeta=2000$, whereas larger values only slightly modify the curves. We therefore fix $\zeta=2000$ and then examine the role of $\chi$. Panel (b) shows that increasing $\chi$ leads to smaller values of $\left\Vert \nabla\cdot\mathbf{B}_{h}\right\Vert$ throughout the simulation. The reduction is pronounced up to $\chi=500$ and continues for larger values, although less strongly. In view of this behavior, we choose $\zeta=2000$ and $\chi=500$ for the computations in Section \ref{subsec:Orszag}, since these values already provide a substantial reduction of both divergence norms.

\begin{figure}[!h]
\centering
\includegraphics[width=0.99\textwidth]{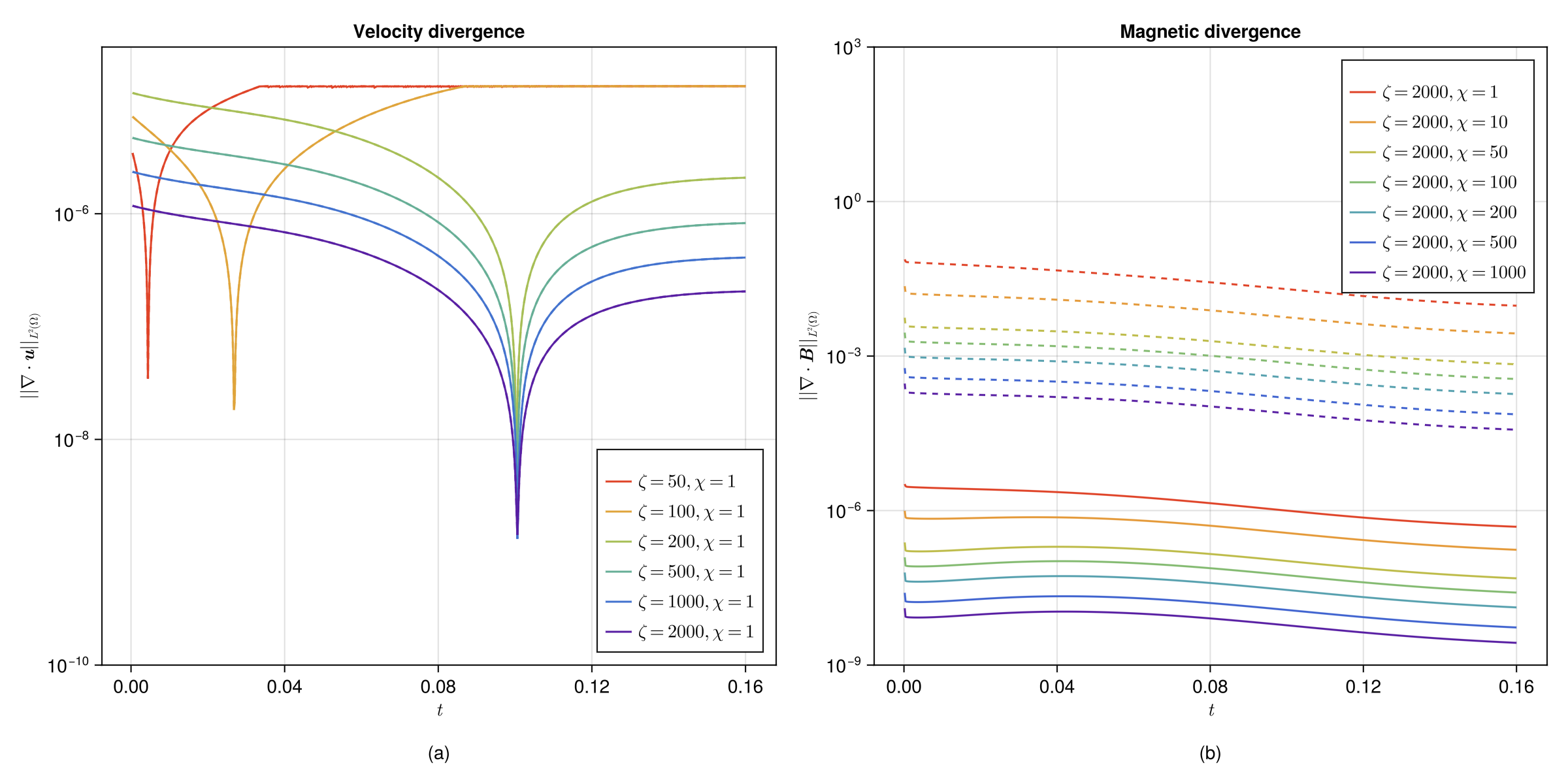}
\caption{Time evolution of the divergence norms in the linear ramp case (Case 2) used for selecting the stabilization parameters.
\textbf{(a)}~$\left\Vert \nabla\cdot\mathbf{u}_{h}\right\Vert$ for several values of $\zeta$ with $\chi=1$; 
\textbf{(b)}~$\left\Vert \nabla\cdot\mathbf{B}_{h}\right\Vert$ for several values of $\chi$ with $\zeta=2000$. Dashed lines correspond to the values before cleaning and solid lines correspond to the values after cleaning.}
\label{fig:orszag_choosing_stabilization}
\end{figure}

To check that this choice does not materially affect the computed dynamics, we compare the kinetic energy $K$, magnetic energy $M$, enstrophy $Z$, and current enstrophy $J$ for the tested parameter values. Taking the run with $(\zeta,\chi)=(2000,500)$ as the reference solution, 
let $K_{\zeta,\chi}$ denote the kinetic energy obtained with stabilization parameters $(\zeta,\chi)$. 
We then define
\[
\delta_{K}\left(\zeta,\chi\right)=\frac{\int^{T}_{0}\left|K_{\zeta,\chi}\left(t\right)-K_{2000,500}\left(t\right)\right|dt}{\int^{T}_{0}\left|K_{2000,500}\left(t\right)\right|dt}
\]
and similarly for $\delta_M$, $\delta_Z$, and $\delta_J$.

The corresponding results are collected in Table \ref{tab:orszag_optimal_stabilization}. Increasing $\zeta$ from 50 to 2000 with $\chi=1$ reduces
${\displaystyle \max_{n}\lVert \nabla\cdot\mathbf{u}^{n}_{h}\rVert}$
from $1.35\times10^{-5}$ to $1.18\times10^{-6}$, while
${\displaystyle \max_{n}\lVert \nabla\cdot\mathbf{B}^{n}_{h}\rVert}$
remains unchanged. After fixing $\zeta=2000$, increasing $\chi$ from 1 to 500 reduces
${\displaystyle \max_{n}\lVert \nabla\cdot\mathbf{B}^{n}_{h}\rVert}$
from $3.27\times10^{-6}$ to $2.53\times10^{-8}$, whereas
${\displaystyle \max_{n}\lVert \nabla\cdot\mathbf{u}^{n}_{h}\rVert}$
stays nearly the same. At the same time, the relative differences $\delta_K$, $\delta_M$, $\delta_Z$, and $\delta_J$ remain very small for all tested parameters. In particular, all of them are below $10^{-3}$, and for the larger values of $\chi$ they are much smaller. This shows that the stabilization parameters have a strong effect on the divergence norms, while their influence on the kinetic energy, magnetic energy, enstrophy, and current enstrophy is negligible on the scale of the present simulations.

\begin{table}[H]
\caption{Maximum divergence norms and relative differences in the kinetic energy, magnetic energy, enstrophy, and current enstrophy corresponding to different stabilization parameters. The quantities $\delta_{K}$, $\delta_{M}$, $\delta_{Z}$, and $\delta_{J}$ are defined with respect to the reference run with $(\zeta,\chi)=(2000,500)$.}\label{tab:orszag_optimal_stabilization}
\resizebox{\textwidth}{!}{
\begin{tabular}{@{\extracolsep{\fill}}ccccccc@{}}
\hline
$\left(\zeta,\chi\right)$ & ${\displaystyle \max_{n}\left\Vert \nabla\cdot\mathbf{u}^{n}_{h}\right\Vert }$ & ${\displaystyle \max_{n}\left\Vert \nabla\cdot\mathbf{B}^{n}_{h}\right\Vert }$ & $\delta_{K}$ & $\delta_{M}$ & $\delta_{Z}$ & $\delta_{J}$\\
\hline
(50, 1) & $1.3499\times10^{-5}$ & $3.2658\times10^{-6}$ & $3.9153\times10^{-4}$ & $7.7490\times10^{-4}$ & $1.5204\times10^{-4}$ & $8.7946\times10^{-4}$\\
(100, 1) & $1.3499\times10^{-5}$ & $3.2658\times10^{-6}$ & $3.9153\times10^{-4}$ & $7.7490\times10^{-4}$ & $1.5204\times10^{-4}$ & $8.7946\times10^{-4}$\\
(200, 1) & $1.1755\times10^{-5}$ & $3.2658\times10^{-6}$ & $3.9152\times10^{-4}$ & $7.7490\times10^{-4}$ & $1.5204\times10^{-4}$ & $8.7945\times10^{-4}$\\
(500, 1) & $4.7040\times10^{-6}$ & $3.2658\times10^{-6}$ & $3.9152\times10^{-4}$ & $7.7490\times10^{-4}$ & $1.5204\times10^{-4}$ & $8.7946\times10^{-4}$\\
(1000, 1) & $2.3524\times10^{-6}$ & $3.2658\times10^{-6}$ & $3.9152\times10^{-4}$ & $7.7490\times10^{-4}$ & $1.5204\times10^{-4}$ & $8.7946\times10^{-4}$\\
(2000, 1) & $1.1763\times10^{-6}$ & $3.2658\times10^{-6}$ & $3.9152\times10^{-4}$ & $7.7490\times10^{-4}$ & $1.5204\times10^{-4}$ & $8.7946\times10^{-4}$\\
(2000, 10) & $1.1792\times10^{-6}$ & $9.8632\times10^{-7}$ & $1.0414\times10^{-4}$ & $2.0551\times10^{-4}$ & $6.0049\times10^{-5}$ & $2.3292\times10^{-4}$\\
(2000, 50) & $1.1800\times10^{-6}$ & $2.4044\times10^{-7}$ & $2.3135\times10^{-5}$ & $4.5478\times10^{-5}$ & $1.4729\times10^{-5}$ & $5.1350\times10^{-5}$\\
(2000, 100) & $1.1801\times10^{-6}$ & $1.2360\times10^{-7}$ & $1.0572\times10^{-5}$ & $2.0777\times10^{-5}$ & $6.8492\times10^{-6}$ & $2.3434\times10^{-5}$\\
(2000, 200) & $1.1802\times10^{-6}$ & $6.2682\times10^{-8}$ & $4.0218\times10^{-6}$ & $7.9036\times10^{-6}$ & $2.6305\times10^{-6}$ & $8.9081\times10^{-6}$\\
(2000, 500) & $1.1802\times10^{-6}$ & $2.5289\times10^{-8}$ & -- & -- & -- & --\\
(2000, 1000) & $1.1802\times10^{-6}$ & $1.2681\times10^{-8}$ & $1.3563\times10^{-6}$ & $2.6655\times10^{-6}$ & $8.9438\times10^{-7}$ & $3.0025\times10^{-6}$\\
\hline
\end{tabular}
}
\end{table}

\subsubsection{Impact of Constant-order Fractional Derivatives on Energy-Enstrophy Measures}

We begin by analyzing Case 1. Figure \ref{fig:orszag_const} shows the effect of the constant fractional order $\alpha$ on the kinetic energy $K\left(t\right)$, magnetic energy $M\left(t\right)$, enstrophy $Z\left(t\right)$, and current enstrophy $J\left(t\right)$. 
The dependence on $\alpha$ is clearest in the magnetic quantities $M\left(t\right)$ and $J\left(t\right)$. In both plots, the classical case stays above all fractional cases, and smaller values of $\alpha$ lead to faster decay over the whole interval.

\begin{figure}[H]
\centering
\includegraphics[width=0.99\textwidth]{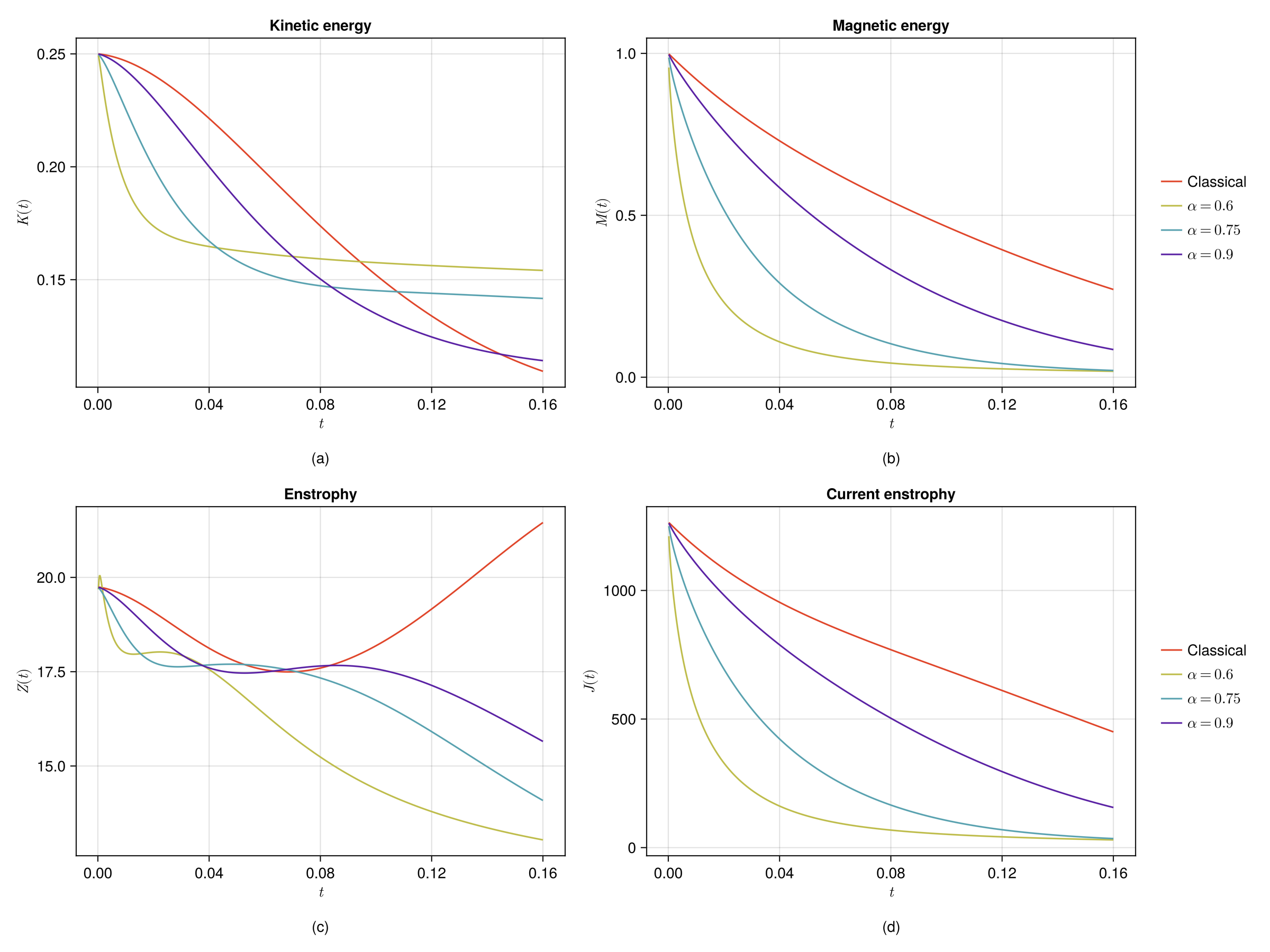}
\caption{Comparison of the diagnostics for the classical model $(\alpha\left(t\right)=\beta\left(t\right)\equiv 1)$ and the constant-order fractional cases, showing the effect of $\alpha$ on the evolution of the main energy and enstrophy measures:
\textbf{(a)}~Kinetic energy $K\left(t\right)$; 
\textbf{(b)}~Magnetic energy $M\left(t\right)$;
\textbf{(c)}~Enstrophy $Z\left(t\right)$;
\textbf{(d)}~Current enstrophy $J\left(t\right)$.}\label{fig:orszag_const}
\end{figure}

The kinetic energy $K\left(t\right)$ behaves differently. The cases $\alpha=0.6$ and $\alpha=0.75$ decrease faster at early times, but later their curves become flatter and stay above the classical one. The case $\alpha=0.9$ remains closer to the classical curve, although it also ends above it near the final time. Thus, the influence of $\alpha$ on $K\left(t\right)$ is not the same over the whole interval.

The enstrophy $Z\left(t\right)$ shows a clear difference between the classical and fractional cases. In the classical case, $Z\left(t\right)$ first decreases, reaches a minimum near the middle of the interval, and then increases strongly toward the final time. In all fractional cases, this final increase is absent, and the curves remain below the classical one in the second half of the interval.

There are also visible differences among the fractional cases themselves. For $\alpha=0.6$, the enstrophy drops most rapidly at early times and then continues to decrease after a short intermediate flattening. For $\alpha=0.75$, the behavior is similar, but the decrease is less steep and the flatter part is more pronounced. The case $\alpha=0.9$ stays closest to the classical curve for the longest time: after the initial decrease, it shows a mild increase around the middle of the interval, and only later turns downward. Thus, smaller values of $\alpha$ lead to an earlier and stronger reduction of enstrophy, while values closer to 1 retain a profile more similar to the classical one.

\subsubsection{Impact of Variable-order Fractional Derivatives on Energy-Enstrophy Measures}

Figure \ref{fig:orszag_variable} compares the variable-order profiles (Cases 2--5) with the classical model using the same diagnostics. In all four cases, the magnetic energy $M\left(t\right)$ and the current enstrophy $J\left(t\right)$ decay much faster than in the classical solution. At the same time, the variable-order curves are not identical, which shows that the results depend not only on the values of $\alpha\left(t\right)$, but also on how $\alpha\left(t\right)$ changes in time.

\begin{figure}[H]
\centering
\includegraphics[width=0.99\textwidth]{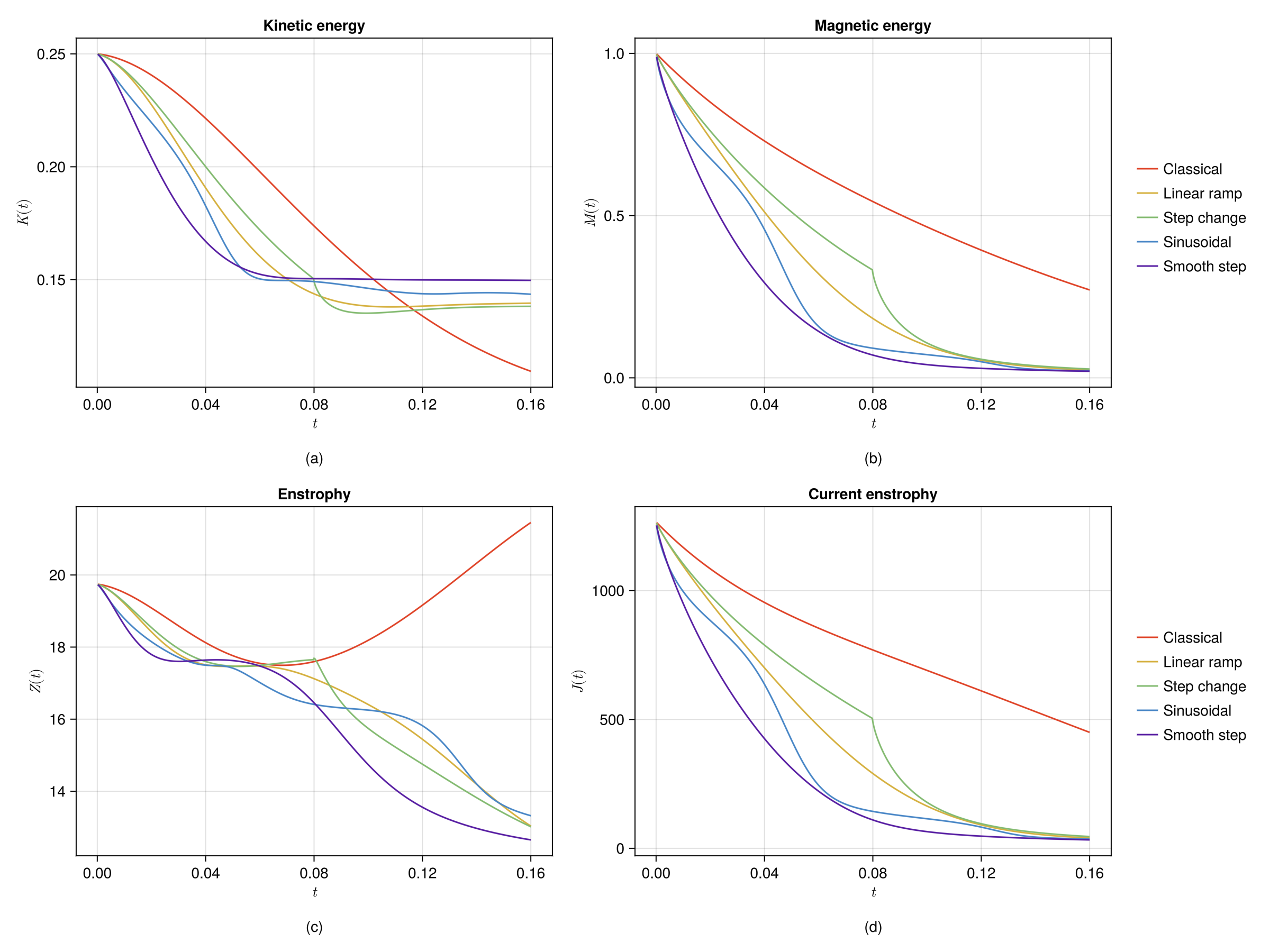}
\caption{Comparison of the diagnostics across classical model $\left(\alpha\equiv1\right)$
and variable-order fractional cases, showing the effect of $\alpha\left(t\right)$
on energy levels and small-scale activity:
\textbf{(a)} Kinetic energy $K\left(t\right)$; 
\textbf{(b)} Magnetic energy $M\left(t\right)$;
\textbf{(c)} Enstrophy $Z\left(t\right)$;
\textbf{(d)} Current enstrophy $J\left(t\right)$.}\label{fig:orszag_variable}
\end{figure}

These differences are most visible in $M\left(t\right)$ and $J\left(t\right)$. In the linear-ramp case, both quantities decrease smoothly. In the step-change case, the decay becomes steeper after the switching time. The smooth-step case behaves similarly, but the change is more gradual. The sinusoidal case is different from the others: both $M\left(t\right)$ and $J\left(t\right)$ show a flatter middle part before continuing to decrease.

The kinetic energy $K\left(t\right)$ behaves differently. All variable-order cases decrease faster than the classical one at early times, but later they flatten and remain above the classical curve. The smooth-step case stays highest at later times, while the linear-ramp, step-change, and sinusoidal cases remain closer to each other.

The enstrophy $Z\left(t\right)$ also differs clearly from the classical case. In the classical solution, it first decreases and then grows strongly near the final time. This final growth is absent in all variable-order cases. The linear-ramp case decreases rather regularly after the initial stage. The step-change case stays close to it at first and then drops faster after the change in order. The sinusoidal case remains higher for longer in the middle of the interval, while the smooth-step case decreases more steadily and reaches the smallest values near the end.

Since the kinetic and magnetic energies are not separately required to be monotone in MHD, we also monitor the total energy $K\left(t\right)+M\left(t\right)$. As shown in Figure \ref{fig:orszag_total}(a), the total energy decreases for all variable-order profiles. Thus, the mild late-time increase of $K\left(t\right)$ observed in some cases does not indicate growth of the total energy.

\begin{figure}[H]
\centering
\includegraphics[width=0.99\textwidth]{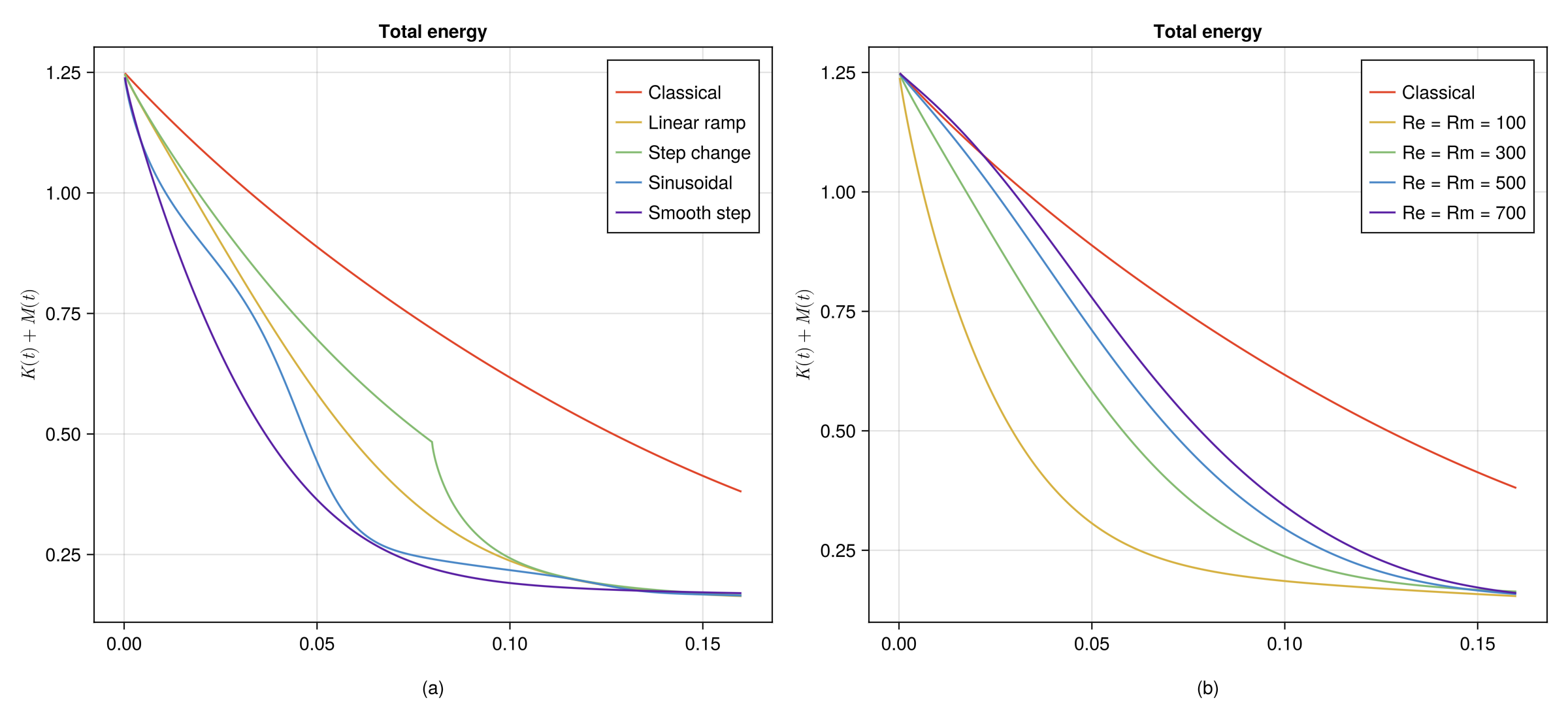}
\caption{Total energy $K\left(t\right)+M\left(t\right)$ for \textbf{(a)} variable-order profiles and \textbf{(b)} different Reynolds numbers. The decay of the total energy shows that the observed late-time increase of $K\left(t\right)$ does not correspond to growth of the total energy.}
\label{fig:orszag_total}
\end{figure}

\subsubsection{The Case of Asynchronous Variable Orders}
\label{subsec:unsymmetric}

In this experiment, we consider six configurations in which $\alpha(t)$ and $\beta(t)$ vary independently. Each configuration is denoted AR-$XY$, where $X\in\{\mathrm{U},\mathrm{D},\mathrm{C}\}$ describes the trend of $\alpha(t)$, $Y\in\{\mathrm{U},\mathrm{D},\mathrm{C}\}$ describes the trend of $\beta(t)$, and $\mathrm{U}$, $\mathrm{D}$, and $\mathrm{C}$ denote increasing, decreasing, and constant profiles, respectively:

Case AR-UD: ${\displaystyle \alpha\left(t\right)=0.6+\frac{0.3t}{T}}$, ${\displaystyle \beta\left(t\right)=0.9-\frac{0.3t}{T}}$.

Case AR-DU: ${\displaystyle \alpha\left(t\right)=0.9-\frac{0.3t}{T}}$, ${\displaystyle \beta\left(t\right)=0.6+\frac{0.3t}{T}}$.

Case AR-UC: ${\displaystyle \alpha\left(t\right)=0.6+\frac{0.3t}{T}}$, $\beta\left(t\right)=0.75$.

Case AR-DC: ${\displaystyle \alpha\left(t\right)=0.9-\frac{0.3t}{T}}$, $\beta\left(t\right)=0.75$.

Case AR-CU: $\alpha\left(t\right)=0.75$, ${\displaystyle \beta\left(t\right)=0.6+\frac{0.3t}{T}}$.

Case AR-CD: $\alpha\left(t\right)=0.75$, ${\displaystyle \beta\left(t\right)=0.9-\frac{0.3t}{T}}$.

Figure \ref{fig:orszag_unsymmetric} compares the classical model with six cases in which $\alpha\left(t\right)$ and $\beta\left(t\right)$ evolve differently. In all six cases, the fractional curves separate from the classical one soon after $t=0$, so varying $\alpha\left(t\right)$ and $\beta\left(t\right)$ independently affects all four diagnostics.

The kinetic energy $K\left(t\right)$ shows a clear spread among the fractional cases. At later times, the largest values are attained by AR-DU and AR-CU, while the smallest values are given by AR-UD. The cases AR-UC and AR-DC lie between these two groups, and AR-CD ends close to AR-UC. Thus, the late-time values of $K\left(t\right)$ differ noticeably from one profile to another.

\begin{figure}[H]
\centering
\includegraphics[width=0.99\textwidth]{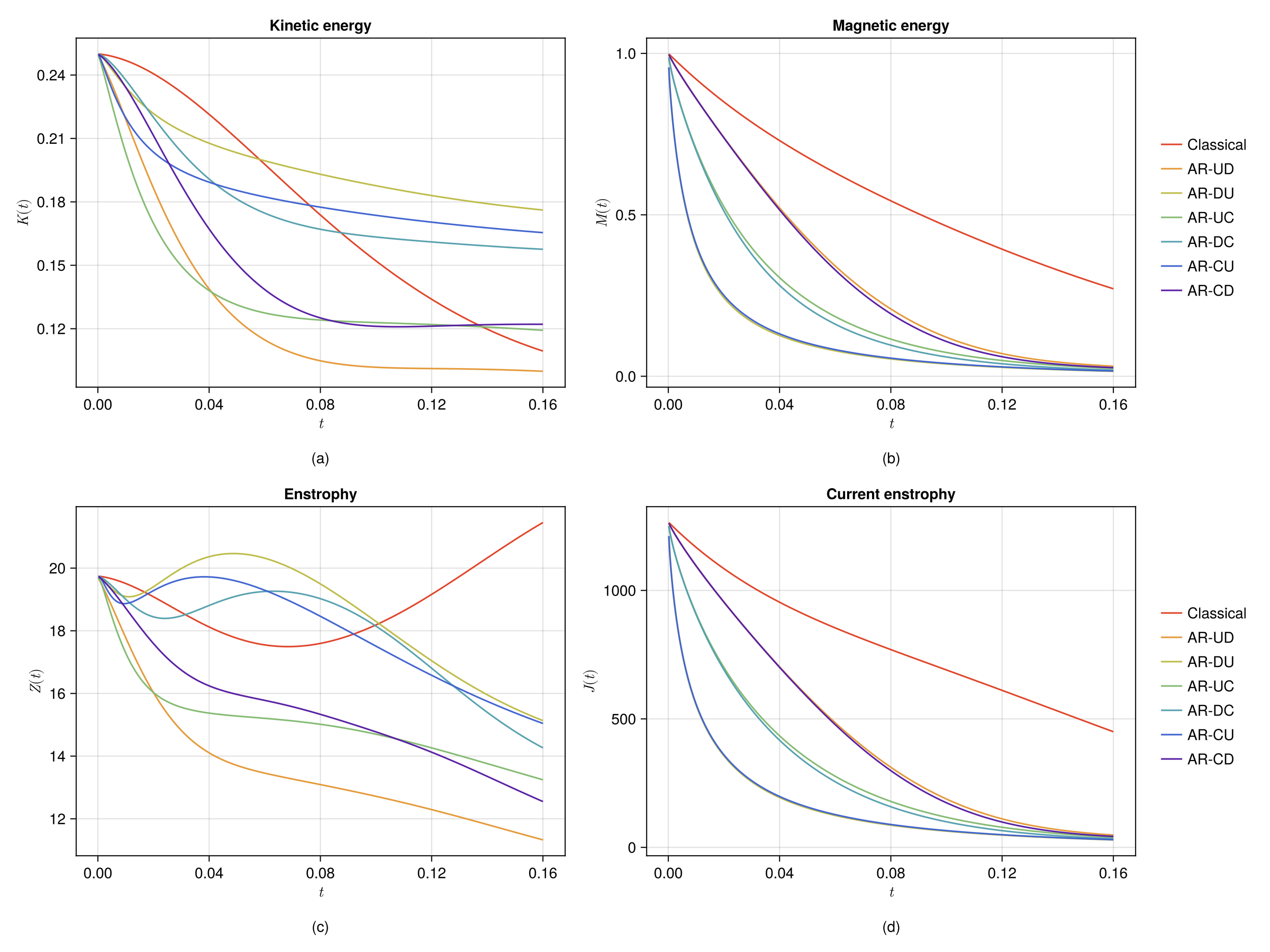}
\caption{Comparison of the diagnostics across classical model $\left(\alpha\left(t\right)=\beta\left(t\right)\equiv1\right)$
and asynchronous variable-order fractional cases, showing the effect of 
independently varying $\alpha\left(t\right)$ and  $\beta\left(t\right)$:
\textbf{(a)}~Kinetic energy $K\left(t\right)$; 
\textbf{(b)}~Magnetic energy $M\left(t\right)$;
\textbf{(c)}~Enstrophy $Z\left(t\right)$;
\textbf{(d)}~Current enstrophy $J\left(t\right)$.}\label{fig:orszag_unsymmetric}
\end{figure}

The magnetic energy $M\left(t\right)$ and the current enstrophy $J\left(t\right)$ show a clearer pattern. In all six fractional cases, both quantities decay much faster than in the classical solution. The slowest decay among the fractional runs is observed for AR-UD and AR-CD, which remain closest to the classical curve. The fastest decay is seen for AR-DU and AR-CU. The remaining two cases, AR-UC and AR-DC, stay between these groups.

The enstrophy $Z\left(t\right)$ also differs clearly from the classical case. In the classical solution, it decreases at first and then grows strongly near the final time. This final growth is absent in all six fractional cases. Among them, AR-DU reaches the highest values in the middle part of the interval, followed by AR-CU and AR-DC. The cases AR-UC and AR-CD remain lower, while AR-UD stays lowest for most of the interval and decreases most strongly toward the end.

Thus, the plots show that allowing $\alpha\left(t\right)$ and $\beta\left(t\right)$ to evolve differently changes not only the size of the magnetic quantities, but also the relative behavior of the velocity-related diagnostics. The clearest separation appears again in $M\left(t\right)$ and $J\left(t\right)$, while $K\left(t\right)$ and $Z\left(t\right)$ show a wider spread between the different profiles.

\subsubsection{Analysis of Relative Changes in Integral Quantities}

To summarize the influence of the fractional orders on the global diagnostics, we introduce the relative $L^{1}\left(0,T\right)$ deviation
\[
\Delta K=
\frac{\int_{0}^{T} \left|K_\alpha\left(t\right)-K_{1}\left(t\right)\right|\,dt}
{\int_{0}^{T} \left|K_{1}\left(t\right)\right|\,dt},
\]
where $K_{\alpha}\left(t\right)$ and $K_{1}\left(t\right)$ denote the kinetic energy in the fractional and classical cases, respectively. The quantities $\Delta M$, $\Delta Z$, and $\Delta J$ are defined in the same way for the magnetic energy, enstrophy, and current enstrophy. The values are listed in Table \ref{tab:orszag_DeltaQ}.

\begin{table}[!ht]
\caption{Relative differences in the integral diagnostics with respect to
the classical model $\left(\alpha\left(t\right)\equiv1,\,\beta\left(t\right)\equiv 1\right)$.}\label{tab:orszag_DeltaQ}
\begin{tabular*}{\textwidth}{@{\extracolsep{\fill}}ccccc@{}}
\hline
\textbf{Case} & ${\Delta K}$ & ${\Delta M}$ & ${\Delta Z}$ & ${\Delta J}$\\
\hline
Constant-order (Case 1), $\alpha\left(t\right)\equiv0.6$ & 0.1923 & 0.8121 & 0.1721 & 0.8080\\
Constant-order (Case 1), $\alpha\left(t\right)\equiv0.75$ & 0.1577 & 0.6363 & 0.1053 & 0.6314\\
Constant-order (Case 1), $\alpha\left(t\right)\equiv0.9$ & 0.0783 & 0.3028 & 0.0697 & 0.2912\\
Linear ramp (Case 2) & 0.1150 & 0.4689 & 0.1165 & 0.4717\\
Step change (Case 3) & 0.0961 & 0.3924 & 0.1254 & 0.3983\\
Sinusoidal (Case 4) & 0.1341 & 0.5664 & 0.1246 & 0.5701\\
Smooth step (Case 5) & 0.1624 & 0.6481 & 0.1584 & 0.6490\\
Asynchronous ramp (Case 6), AR-UD & 0.2821 & 0.4447 & 0.2774 & 0.4564\\
Asynchronous ramp (Case 6), AR-DU & 0.1647 & 0.7990 & 0.1153 & 0.7884\\
Asynchronous ramp (Case 6), AR-UC & 0.2321 & 0.6208 & 0.2001 & 0.6199\\
Asynchronous ramp (Case 6), AR-DC & 0.1282 & 0.6455 & 0.1026 & 0.6374\\
Asynchronous ramp (Case 6), AR-CU & 0.1605 & 0.7937 & 0.1039 & 0.7852\\
Asynchronous ramp (Case 6), AR-CD & 0.1723 & 0.4592 & 0.1810 & 0.4660\\
\hline
\end{tabular*}
\end{table}

For the constant-order cases, all four deviations decrease as $\alpha$ increases from 0.6 to 0.9. Thus, in the constant-order case, the cumulative difference from the classical case decreases as the order approaches 1.

Among the symmetric variable-order cases (Cases 2--5), the step-change profile gives the smallest values of $\Delta K$, $\Delta M$, and $\Delta J$, while the smooth-step profile gives the largest values of these quantities. The linear-ramp and sinusoidal cases lie between these two. In all four cases, $\Delta M$ and $\Delta J$ are clearly larger than $\Delta K$ and $\Delta Z$, which suggests that the magnetic diagnostics are more sensitive than the kinetic ones.

The asynchronous cases (Case 6) show a different pattern. The largest kinetic deviation is obtained for AR-UD, followed by AR-UC, whereas the largest magnetic deviations are obtained for AR-DU and AR-CU. The cases AR-UC and AR-DC lie between these extremes, while AR-CD remains closer to AR-UD in the magnetic diagnostics. The enstrophy deviations also vary noticeably, with the largest value attained by AR-UD.
It follows from the table that once $\alpha\left(t\right)$ and $\beta\left(t\right)$ vary independently, the cumulative deviation from the classical solution is no longer described by a single common trend. Some profiles produce the largest changes in the kinetic energy and enstrophy, whereas others produce the largest changes in the magnetic energy and current enstrophy.

\subsubsection{Influence of the Reynolds Numbers on Energy-Enstrophy Measures}

Figure~\ref{fig:orszag_re} shows the influence of the Reynolds numbers $\mathrm{Re}=\mathrm{Rm}$ on the diagnostics for the linear-ramp case (Case~2). The clearest effect is seen in the magnetic energy $M\left(t\right)$ and the current enstrophy $J\left(t\right)$. As $\mathrm{Re}=\mathrm{Rm}$ increases from 100 to 700, both quantities decay more slowly, and their curves remain successively higher over the whole time interval.

The kinetic energy $K\left(t\right)$ behaves differently. After a short initial stage, the case $\mathrm{Re}=\mathrm{Rm}=100$ stays above the others over most of the interval. Among the cases $\mathrm{Re}=\mathrm{Rm}=300,500,700$, the larger Reynolds numbers generally give smaller values of $K(t)$ over most of the time interval.

The enstrophy $Z\left(t\right)$ shows a different pattern. For $\mathrm{Re}=\mathrm{Rm}=100$, it decreases over the whole interval. For $\mathrm{Re}=\mathrm{Rm}=300$, the curve first decreases, then shows a mild rise, and finally decreases again. For $\mathrm{Re}=\mathrm{Rm}=500$ and 700, this rise becomes much more pronounced, and the largest peak is reached for $\mathrm{Re}=\mathrm{Rm}=700$.

Thus, increasing $\mathrm{Re}$ and $\mathrm{Rm}$ slows down the decay of $M\left(t\right)$ and $J\left(t\right)$. The effect on $K\left(t\right)$ is different and is not ordered in the same way. For $Z\left(t\right)$, larger Reynolds numbers lead to a stronger rise at intermediate times and to a higher peak before the final decay.

\begin{figure}
\centering
\includegraphics[width=\textwidth]{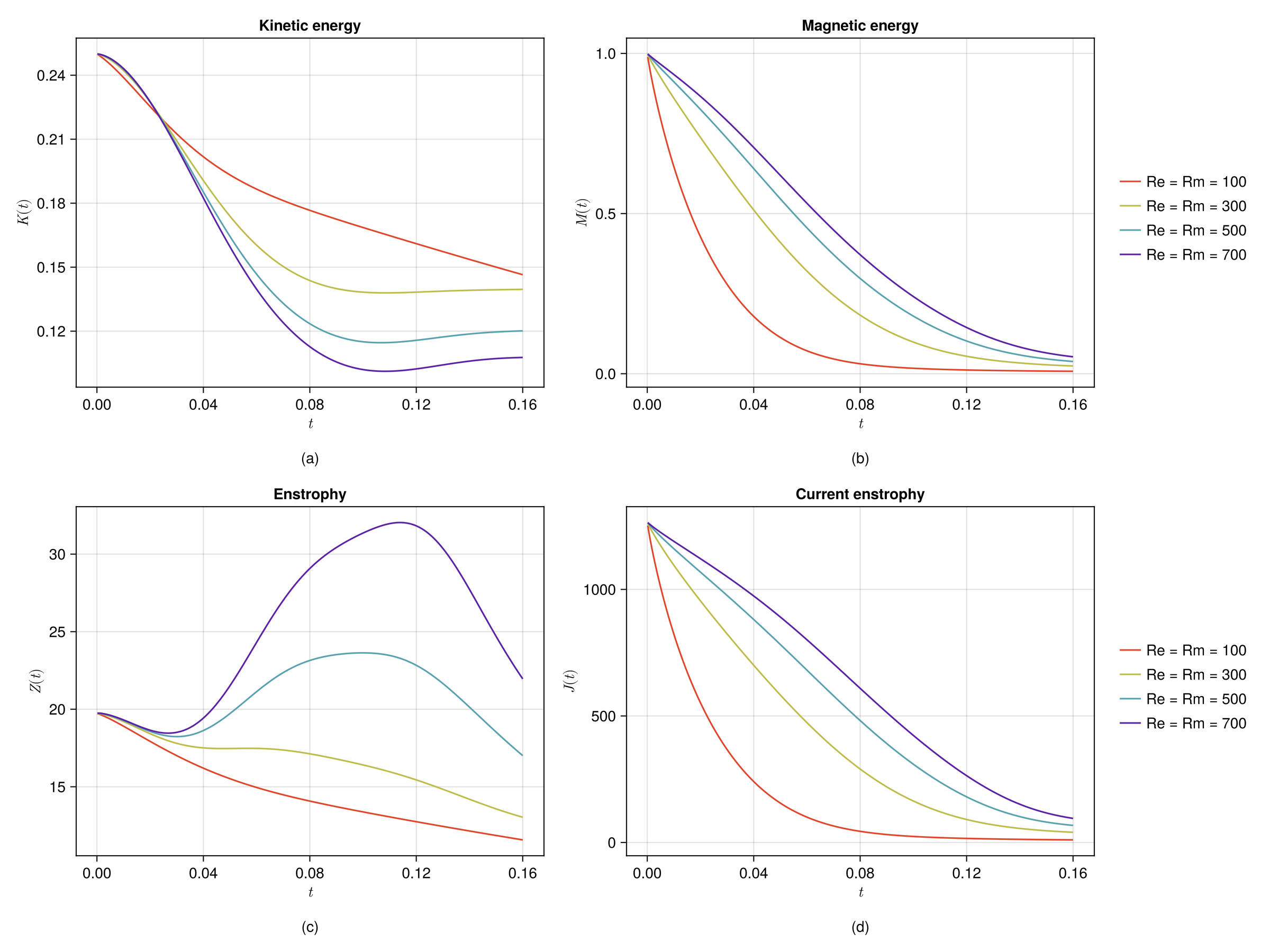}
\caption{Comparison of the diagnostics across different values of the Reynolds numbers $\mathrm{Re}=\mathrm{Rm}$:
\textbf{(a)}~Kinetic energy $K\left(t\right)$; 
\textbf{(b)}~Magnetic energy $M\left(t\right)$;
\textbf{(c)}~Enstrophy $Z\left(t\right)$;
\textbf{(d)}~Current enstrophy $J\left(t\right)$.}
\label{fig:orszag_re}
\end{figure}

The mild late-time increase of $K\left(t\right)$ for larger Reynolds numbers is again compensated by the decay of $M\left(t\right)$. The corresponding total-energy curves, shown in Figure \ref{fig:orszag_total}(b), remain decreasing throughout the interval.

\subsubsection{Analysis of the Divergence Constraints}

Figure \ref{fig:orszag_div_errors} shows the time evolution of the divergence norms for the variable-order profiles considered in the tests above. In all cases, the velocity divergence $\lVert\nabla\cdot \mathbf{u}_{h}\rVert$ remains small over the whole interval. Its values are of order $10^{-6}$ at early times, decrease further around the middle of the simulation, and then remain at the level of $10^{-7}$. Although the detailed shape depends on the chosen order profile, no growth to large values is observed.

For the magnetic field, both the values before cleaning and the values after cleaning are shown. Before cleaning, $\lVert\nabla\cdot \mathbf{B}_{h}\rVert$ stays at the level of $10^{-4}$ throughout the interval. After cleaning, it is reduced to the level of $10^{-8}$ for all considered profiles. In both cases, the curves remain bounded and show only moderate variation in time. Thus, for all variable-order functions used in the computations above, the divergence errors remain under control during the whole simulation interval.

\begin{figure}
  \centering
  \includegraphics[width=\textwidth]{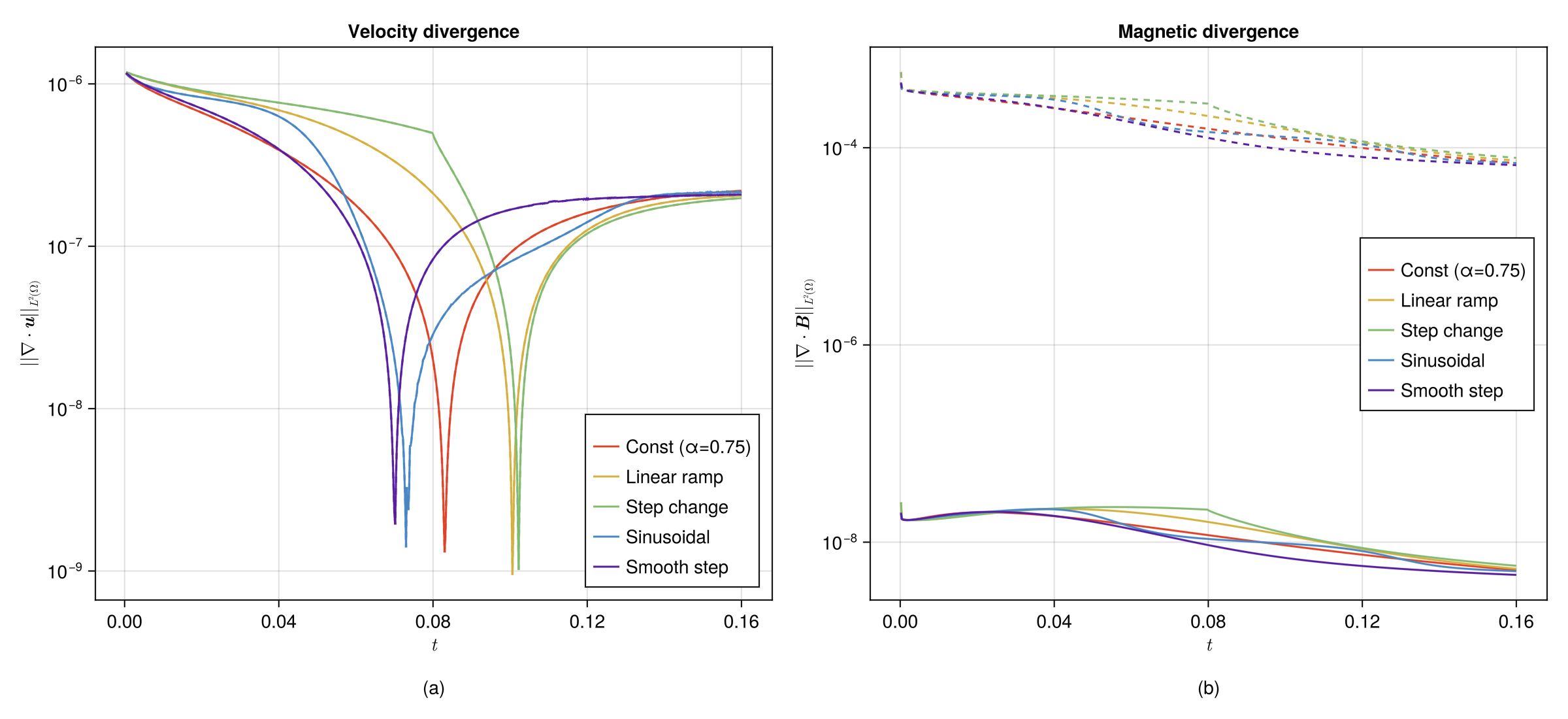}
  \caption{Time evolution of the divergence norms for several representative variable-order profiles.
\textbf{(a)} Velocity divergence $\lVert\nabla\cdot\mathbf{u}_{h}\rVert$.
\textbf{(b)} Magnetic divergence $\lVert\nabla\cdot\mathbf{B}_h\rVert$, with dashed lines corresponding to the values before cleaning and solid lines to the values after cleaning.}
\label{fig:orszag_div_errors}
\end{figure}

\subsection{Phase Diagrams in the Order-Range Plane}

We next investigate how the initial and final values of the variable-order temporal fractional derivative, denoted by $\left(\alpha_{0},\alpha_{T}\right)$, influence diagnostics of the MHD flow. In this experiment, $\left(\alpha_{0},\alpha_{T}\right)$ are the endpoints of the linear-ramp profile
\begin{equation}
\alpha\left(t\right)=\beta\left(t\right)=\alpha_{0}+(\alpha_{T}-\alpha_{0})\frac{t}{T},\qquad t\in\left[0,T\right],
\label{eq:phase_linear_ramp}
\end{equation}
with $T=0.16$, so that $\alpha\left(0\right)=\alpha_{0}$ and $\alpha\left(T\right)=\alpha_{T}$. The initial and boundary data, as well as all remaining parameters, are the same as in Section \ref{subsec:Orszag}.

For each pair $\left(\alpha_{0},\alpha_{T}\right)$ on the grid 
$0.5\le \alpha_{0}\le 0.9$, $0.5\le \alpha_{T}\le 0.9$
we compute the time histories of the diagnostic quantities $K\left(t\right)$, $M\left(t\right)$, $Z\left(t\right)$, and $J\left(t\right)$, and then evaluate their time integrals
\[
I_{K}=\int_{0}^{T} K\left(t\right)\,dt,\qquad
I_{M}=\int_{0}^{T} M\left(t\right)\,dt,\qquad
I_{Z}=\int_{0}^{T} Z\left(t\right)\,dt,\qquad
I_{J}=\int_{0}^{T} J\left(t\right)\,dt.
\]

To measure the effect of $\left(\alpha_{0},\alpha_{T}\right)$, we compare these quantities with the classical reference solution and define
\[
\delta {I}_{K}=\frac{I_{K}-I_{K}^{\mathrm{ref}}}{I_{K}^{\mathrm{ref}}},\qquad
\delta {I}_{M}=\frac{I_{M}-I_{M}^{\mathrm{ref}}}{I_{M}^{\mathrm{ref}}},\qquad
\delta {I}_{Z}=\frac{{I}_{Z}-{I}_{Z}^{\mathrm{ref}}}{{I}_{Z}^{\mathrm{ref}}},\qquad
\delta {I}_{J}=\frac{{I}_{J}-{I}_{J}^{\mathrm{ref}}}{{I}_{J}^{\mathrm{ref}}}.
\]
The heatmaps in Figure \ref{fig:phase_heatmaps} summarize how these relative deviations vary over the $(\alpha_{0},\alpha_{T})$-plane.

\begin{figure}[!htb]
  \centering
  \includegraphics[width=\textwidth]{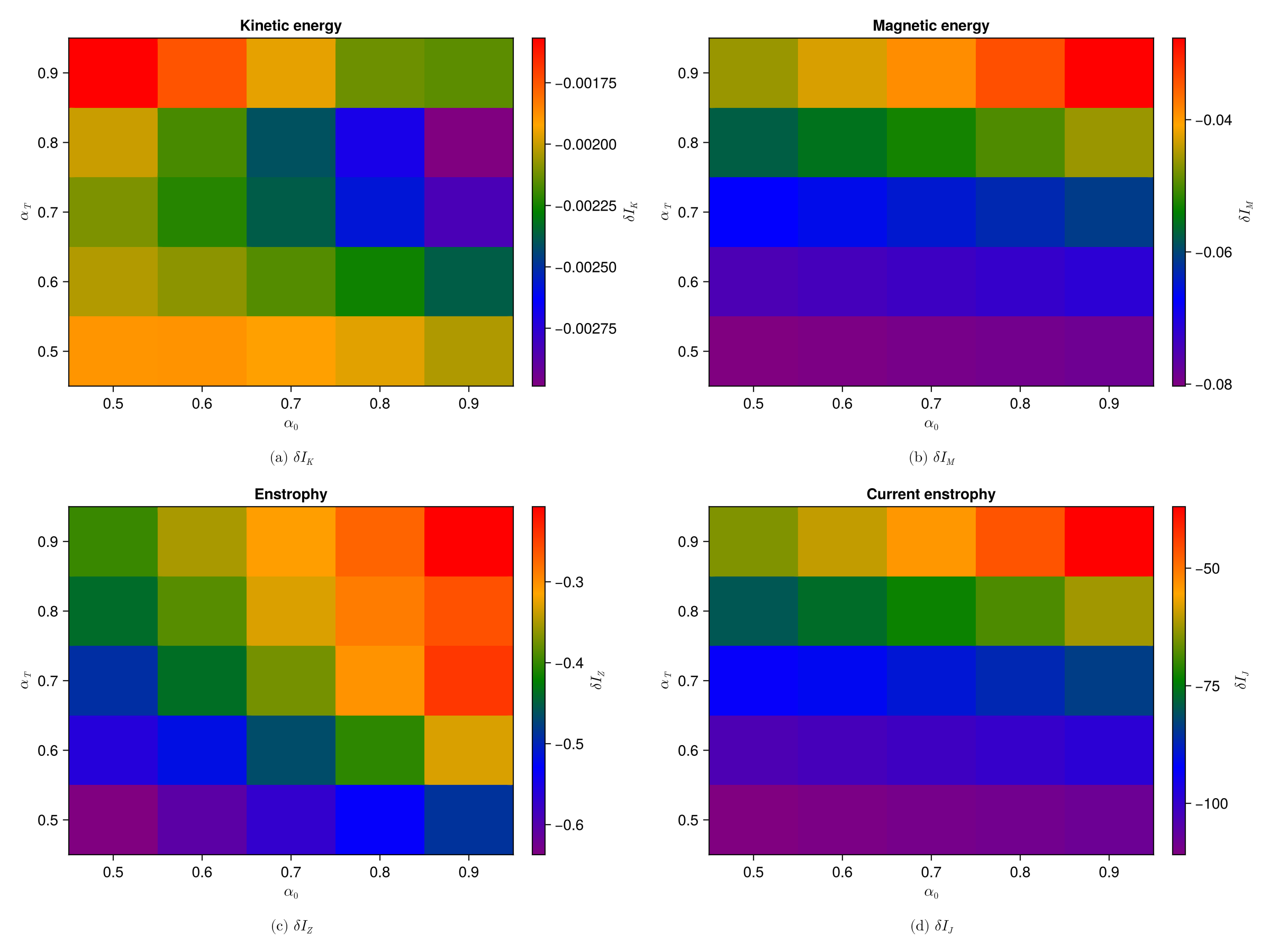}
  \caption{Heatmaps of the relative deviations of the time-integrated diagnostics for the linear-ramp profile (\ref{eq:phase_linear_ramp}).
The horizontal and vertical axes correspond to the values of $\alpha_{0}$ and $\alpha_{T}$, respectively. 
\textbf{(a)} $\delta I_{K}$, relative deviation of the time-integrated kinetic energy; 
\textbf{(b)} $\delta I_{M}$, relative deviation of the time-integrated magnetic energy; 
\textbf{(c)} $\delta I_{Z}$, relative deviation of the time-integrated enstrophy; 
\textbf{(d)} $\delta I_{J}$, relative deviation of the time-integrated current enstrophy. 
All quantities are computed relative to the classical reference solution.}
\label{fig:phase_heatmaps}
\end{figure}

The heatmaps show that all four quantities satisfy $\delta I_{K}<0$, $\delta I_{M}<0$, $\delta I_{Z}<0$, and $\delta I_{J}<0$ throughout the tested range. Therefore, for all considered ramp profiles, the time-integrated kinetic energy, magnetic energy, enstrophy, and current enstrophy are smaller than in the classical case. The largest changes appear in $I_{J}$ and $I_{Z}$. The changes in $I_{M}$ are also visible but smaller, whereas the changes in $I_{K}$ are much weaker.

The dependence on $\left(\alpha_{0},\alpha_{T}\right)$ is clearest in $\delta I_{M}$, $\delta I_{Z}$, and $\delta I_{J}$. In these three maps, the deviations become more negative when the order range moves toward smaller values, and less negative as $\left(\alpha_{0},\alpha_{T}\right)$ approaches $\left(0.9,0.9\right)$. For $\delta I_{K}$, the same pattern is only weakly visible and is less regular.

Taken together, these maps show that the order range $\left(\alpha_{0},\alpha_{T}\right)$ affects all four integrated diagnostics. The values of $I_{M}$, $I_{Z}$, and $I_{J}$ are the most sensitive among those considered, while $I_{K}$ is the least sensitive.

\section{Discussion}
\label{sec:Discussion}

The results obtained in this work show that replacing the integer-order time derivatives in the incompressible MHD system with variable-order Caputo derivatives allows a wider range of transient behaviors to be described within the same governing equations. From a physical point of view, the variable-order fractional terms introduce a memory effect whose strength changes in time. In qualitative terms, smaller fractional orders correspond to stronger temporal memory, whereas values closer to unity lead to behavior that is closer to the classical incompressible MHD model. In this way, the order functions $\alpha\left(t\right)$ and $\beta\left(t\right)$ can be used to describe time-dependent changes in memory strength.

Although the variable-order fractional MHD tests considered here do not use
exactly the same initial data as many standard benchmark studies
\cite{Borve2006,Kayanikhoo2023,Hiptmair2018}, two qualitative checks are still
useful. First, in the classical case $\alpha\left(t\right),\beta\left(t\right)\equiv 1$, the
computed solution shows the expected dissipative behavior of a periodic
divergence-free MHD vortex: the kinetic and magnetic energies decay in time,
while the enstrophy and current-enstrophy curves provide additional information
on the evolution of velocity and magnetic gradients. Such qualitative
diagnostics are often used in the MHD literature when assessing numerical
solvers on periodic vortex-type problems \cite{Orszag1979,Jadhav2021,Grauer1998}.
Therefore, the run with $\alpha\left(t\right),\beta\left(t\right)\equiv 1$ provides a consistent
reference for interpreting the fractional-order cases computed with the same
discretization, parameters, and diagnostics.

Second, varying the fractional order produces visible changes in all diagnostics over the whole time interval. These changes are seen in the levels of the curves and, in particular for $Z$ and $J$, in the timing of the transient extrema. When $\alpha\left(t\right)$ and $\beta\left(t\right)$ remain close to 1, the diagnostics stay closer to the classical solution. Lower orders, or stronger variation in time, lead to larger deviations and more noticeable shifts in the extrema. This suggests that variable-order memory affects both the overall level of the diagnostics and the transient development of the flow.

Similar sensitivity to the order profile has also been reported in other variable-order fractional flow models. For example, porous-media studies \cite{Alimbekova2024b} comparing constant, increasing, and decreasing orders show that different choices of the order function can lead to visibly different flow responses. Although the governing equations in those models differ from the present incompressible MHD system, the qualitative point is similar: changing the order profile modifies the strength of the memory effect and thereby changes the transient evolution.

Finally, we recall that the present $H^{1}$-conforming discretization does not enforce the divergence constraints exactly at the discrete level. Therefore, the divergence norms reported in this work are expected to be small but not identically zero. In the convergence tests, moderate stabilization parameters were sufficient to keep these quantities small, while in the periodic divergence-free vortex test larger stabilization parameters were used to obtain smaller divergence errors.

The main limitation of the present study is that the numerical investigation has been carried out for a finite set of order functions. A natural extension is to consider a broader class of variable-order profiles $\alpha\left(t\right)$ and $\beta\left(t\right)$, including dependencies on spatial variables. Another important direction is comparison with experimental data or with high-accuracy numerical simulations for classical MHD.

\section{Conclusions}
\label{sec:Conclusions}

This work studied a variable-order time-fractional incompressible MHD model in which the classical first-order time derivatives in the momentum and induction equations are replaced by Caputo derivatives with time-dependent order. For the resulting problem, a fully discrete finite element--L1 scheme was proposed and analyzed through stability and convergence results. The scheme was then examined numerically by means of convergence tests, a classical-limit study as the orders approach unity, computations for the periodic divergence-free vortex, and order-range sensitivity maps for selected diagnostics.

The numerical results lead to several main observations. First, the finite element--L1 discretization shows the expected temporal convergence orders for the representative variable-order profiles considered in this work. Second, the classical-limit test confirms consistency with the standard incompressible MHD model: as the fractional orders approach one, the differences between the fractional and classical solutions decrease, and the corresponding diagnostic curves become closer. Third, in the periodic divergence-free vortex test, changing the order functions produces clear changes in the kinetic and magnetic energies, enstrophy, and current enstrophy. Finally, the order-range maps show how these diagnostics vary over the $(\alpha_{0},\alpha_{T})$-plane and indicate regions where the influence of variable order is stronger or weaker.

Taken together, these results show that variable-order Caputo time derivatives provide a useful way to describe time-dependent memory effects in incompressible MHD while remaining consistent with the classical limit. Future work will include the study of broader classes of order functions in the momentum and induction equations, additional benchmark problems and boundary conditions, extension of the analysis to more general variable-order cases, and acceleration techniques for long-time simulations.

\section*{Acknowledgements}

This research was funded by the Science Committee of the Ministry of Science and Higher Education of the Republic of Kazakhstan (Grant No. AP26101983).

\section*{Appendix A. Discrete fractional Gr\"{o}nwall inequality}
\label{appendix:Gronwall}

For the reader’s convenience, we restate the kernel assumptions and the discrete fractional Grönwall theorem of \cite{Liao_2019} below in notation adapted to the present paper.

Let $0=t_{0}<t_{1}<...<t_{N}=T$, $\tau_{n}=t_{n}-t_{n-1}$, let
$0\leq\theta<1$, and let $\gamma\in\left(0,1\right)$. For any sequence $\left\{ v^{n}\right\} ^{N}_{n=0}$,
define
\[
t_{n-\theta}=\theta t_{n-1}+\left(1-\theta\right)t_{n},\qquad v^{n-\theta}=\theta v^{n-1}+\left(1-\theta\right)v^{n},\qquad\nabla_{\tau}v^{n}=v^{n}-v^{n-1}.
\]

Consider a discrete fractional derivative written in the kernel form,
\[
\left(\mathcal{D}^{\gamma}_{\tau}v\right)^{n-\theta}=\sum^{n}_{k=1}A^{\left(n\right)}_{n-k}\nabla_{\tau}v^{k},\qquad1\leq n\leq N,
\]
where the coefficients $A^{\left(n\right)}_{n-k}$ are referred to
as discrete kernels. The abstract theory of \cite{Liao_2019}
is formulated for kernels of this form under the following structural
assumptions. 

\textbf{Assumption A1.} For each fixed $n$, the kernel sequence is positive
and nonincreasing:
\[
A^{\left(n\right)}_{0}\geq A^{\left(n\right)}_{1}\geq A^{\left(n\right)}_{2}\geq...\geq A^{\left(n\right)}_{n-1}>0.
\]

\textbf{Assumption A2.} There exists a constant 
$\pi_{A}>0$, independent of the time steps, such that
\[
A^{\left(n\right)}_{n-k}\geq\frac{1}{\pi_{A}\tau_{k}}\int^{t_{k}}_{t_{k-1}}\omega_{1-\gamma}\left(t_{n}-s\right)ds,\qquad1\leq k\leq n\leq N,
\]
where ${\displaystyle \omega_{1-\gamma}\left(t\right)=\frac{t^{-\gamma}}{\Gamma\left(1-\gamma\right)}}$.

\textbf{Assumption A3.} There exists a constant $\rho>0$ such that the local
step ratios satisfy
\[
\rho_{k}=\frac{\tau_{k}}{\tau_{k+1}}\leq\rho,\qquad1\leq k\leq N-1.
\]

When A1--A2 hold, the associated complementary kernels
$\{\mathcal P_{n-j}^{(n)}\}_{j=1}^{n}$ are defined as the unique kernels
satisfying
\[
\sum_{j=m}^{n}\mathcal P_{n-j}^{(n)}A_{j-m}^{(j)}=1,
\qquad 1\le m\le n\le N.
\]

\textbf{Theorem (Discrete fractional Gr\"{o}nwall inequality).}
Assume that A1--A3 hold.
Let $\left\{ g^{n}\right\} ^{N}_{n=1}$ and $\left\{ \lambda_{\ell}\right\} ^{N-1}_{\ell=0}$
be nonnegative sequences, and suppose that there exists a constant
$\Lambda$, independent of the time steps, such that ${\displaystyle \Lambda\geq\sum^{N-1}_{\ell=0}\lambda_{\ell}}$.
Assume further that the maximum time step satisfies
\[
\max_{1\leq n\leq N}\tau_{n}\leq\left(2\pi_{A}\Gamma\left(2-\gamma\right)\Lambda\right)^{-1/\gamma}.
\]

If a nonnegative sequence $\left\{ v^{n}\right\} ^{N}_{n=0}$ satisfies
\[
\sum^{n}_{k=1}A^{\left(n\right)}_{n-k}\nabla_{\tau}v^{k}\leq\sum^{n}_{k=1}\lambda_{n-k}v^{k-\theta}+g^{n},\qquad1\leq n\leq N,
\]
then, for every $1\leq n\leq N$,
\[
v^{n}\leq2\mathbb{E}_{\gamma}\left(2\max\left(1,\rho\right)\pi_{A}\Lambda t^{\gamma}_{n}\right)\left(v^{0}+\max_{1\leq k\leq n}\sum^{k}_{j=1}\mathcal{P}^{\left(k\right)}_{k-j}g^{j}\right),
\]
where $\mathbb{E}_{\gamma}$ denotes the Mittag--Leffler function.

\textbf{Remark.} Under Assumptions A1--A2, the complementary kernels are nonnegative. By Remark 1 of \cite{Liao_2019}, they satisfy
\[
\sum_{j=1}^{k}\mathcal{P}^{(k)}_{k-j}\omega_{1-\gamma}(t_j)\leq \pi_A,\qquad 1\leq k\leq N.
\]
Hence
\[
\sum_{j=1}^{k}\mathcal{P}^{(k)}_{k-j}g^{j}
\leq
\pi_A\max_{1\leq j\leq k}\frac{g^{j}}{\omega_{1-\gamma}(t_j)}
=
\pi_A\Gamma(1-\gamma)\max_{1\leq j\leq k}\bigl(t_j^\gamma g^j\bigr),
\qquad 1\leq k\leq N.
\]
In particular, if $g^j\equiv g$ is independent of $j$, then
\[
\sum_{j=1}^{k}\mathcal{P}^{(k)}_{k-j}g
\leq
\pi_A\Gamma(1-\gamma)t_k^\gamma g
\leq
\pi_A\Gamma(1-\gamma)T^\gamma g.
\]

\bibliographystyle{cas-model2-names}
\bibliography{cas-refs}

@book{Moreau1990,
	author = {Moreau, Ren{\'e}},
	doi = {10.1007/978-94-015-7883-7},
	isbn = {9789401578837},
	issn = {0926-5112},
	journal = {Fluid Mechanics and Its Applications},
	publisher = {Springer Netherlands},
	title = {Magnetohydrodynamics},
	year = {1990}}

@book{Davidson2001,
	author = {Davidson, P. A.},
	doi = {10.1017/cbo9780511626333},
	isbn = {9780511626333},
	publisher = {Cambridge University Press},
	title = {An Introduction to Magnetohydrodynamics},
	year = {2001}}

@article{Roberts2013,
	author = {Roberts, Paul H and King, Eric M},
	doi = {10.1088/0034-4885/76/9/096801},
	issn = {1361-6633},
	journal = {Reports on Progress in Physics},
	number = {9},
	pages = {096801},
	publisher = {IOP Publishing},
	title = {On the genesis of the {Earth's} magnetism},
	volume = {76},
	year = {2013}}

@book{Mainardi2010a,
	author = {Mainardi, Francesco},
	doi = {10.1142/p614},
	isbn = {9781848163300},
	publisher = {Imperial College Press},
	title = {Fractional Calculus and Waves in Linear Viscoelasticity: An Introduction to Mathematical Models},
	year = {2010}}

@article{Metzler2000,
	author = {Metzler, Ralf and Klafter, Joseph},
	doi = {10.1016/s0370-1573(00)00070-3},
	issn = {0370-1573},
	journal = {Physics Reports},
	number = {1},
	pages = {1--77},
	publisher = {Elsevier BV},
	title = {The random walk's guide to anomalous diffusion: a fractional dynamics approach},
	volume = {339},
	year = {2000}}

@article{Baishemirov2024,
	title = {Efficient Numerical Implementation of the Time-Fractional Stochastic {Stokes–Darcy} Model},
	volume = {8},
	ISSN = {2504-3110},
	DOI = {10.3390/fractalfract8080476},
	number = {8},
	journal = {Fractal and Fractional},
	publisher = {MDPI AG},
	author = {Baishemirov,  Zharasbek and Berdyshev,  Abdumauvlen and Baigereyev,  Dossan and Boranbek,  Kulzhamila},
	year = {2024},
	pages = {476}
}

@article{Madiyarov2025,
	title = {Nonlocal Modeling and Inverse Parameter Estimation of Time-Varying Vehicular Emissions in Urban Pollution Dynamics},
	volume = {13},
	ISSN = {2227-7390},
	DOI = {10.3390/math13172772},
	number = {17},
	journal = {Mathematics},
	publisher = {MDPI AG},
	author = {Madiyarov,  Muratkan and Alimbekova,  Nurlana and Bakishev,  Aibek and Mukhamediyev,  Gabit and Yergaliyev,  Yerlan},
	year = {2025},
	pages = {2772}
}

@article{Abdiramanov2024,
	title = {An Implicit Difference Scheme for a Mixed Problem of Hyperbolic Type with Memory},
	volume = {45},
	ISSN = {1818-9962},
	DOI = {10.1134/s1995080224600249},
	number = {2},
	journal = {Lobachevskii Journal of Mathematics},
	publisher = {Pleiades Publishing Ltd},
	author = {Abdiramanov,  Zh. A. and Baishemirov,  Zh. D. and Berdyshev,  A. S. and Shiyapov,  K. M.},
	year = {2024},
	pages = {569–577}
}

@article{Magin2004,
	author = {Magin, Richard L.},
	doi = {10.1615/critrevbiomedeng.v32.10},
	issn = {0278-940X},
	journal = {Critical Reviews in Biomedical Engineering},
	number = {1},
	pages = {1--104},
	publisher = {Begell House},
	title = {Fractional Calculus in Bioengineering, Part 1},
	volume = {32},
	year = {2004}}

@book{Tarasov2010,
	author = {Tarasov, Vasily E.},
	doi = {10.1007/978-3-642-14003-7},
	isbn = {9783642140037},
	issn = {1867-8440},
	journal = {Nonlinear Physical Science},
	publisher = {Springer Berlin Heidelberg},
	title = {Fractional Dynamics: Applications of Fractional Calculus to Dynamics of Particles, Fields and Media},
	year = {2010}}

@book{Das2011,
	author = {Das, Shantanu},
	doi = {10.1007/978-3-642-20545-3},
	isbn = {9783642205453},
	publisher = {Springer Berlin Heidelberg},
	title = {Functional Fractional Calculus},
	year = {2011}}

@article{Sun2019,
	author = {Sun, H. G. and Chang, A. and Zhang, Y. and Chen, W.},
	doi = {10.1515/fca-2019-0003},
	journal = {Fractional Calculus and Applied Analysis},
	number = {1},
	pages = {27--59},
	publisher = {Springer Science and Business Media {LLC}},
	title = {A Review on Variable-Order Fractional Differential Equations: Mathematical Foundations, Physical Models, Numerical Methods and Applications},
	volume = {22},
	year = {2019}}

@article{Alimbekova2024b,
	title = {Numerical Method for the Variable-Order Fractional Filtration Equation in Heterogeneous Media},
	volume = {8},
	ISSN = {2504-3110},
	DOI = {10.3390/fractalfract8110640},
	number = {11},
	journal = {Fractal and Fractional},
	publisher = {MDPI AG},
	author = {Alimbekova,  Nurlana and Bakishev,  Aibek and Berdyshev,  Abdumauvlen},
	year = {2024},
	pages = {640}
}

@article{Ali2024,
	author = {Ali, Qasim and Amir, Muhammad and Metwally, Ahmed Sayed M. and Younas, Usman and Jan, Ahmed Zubair and Amjad, Ayesha},
	doi = {10.1007/s10973-024-13205-5},
	issn = {1588-2926},
	journal = {Journal of Thermal Analysis and Calorimetry},
	number = {15},
	pages = {8257--8270},
	publisher = {Springer Science and Business Media LLC},
	title = {Investigation of {MHD} fractionalized viscous fluid and thermal memory with slip and {Newtonian} heating effect: a fractional model based on {Mittag-Leffler} kernel},
	volume = {149},
	year = {2024}}

@article{Arif2023,
	author = {Arif, Muhammad and Kumam, Poom and Seangwattana, Thidaporn and Suttiarporn, Panawan},
	doi = {10.1016/j.heliyon.2023.e17642},
	issn = {2405-8440},
	journal = {Heliyon},
	number = {7},
	pages = {e17642},
	publisher = {Elsevier BV},
	title = {A fractional model of magnetohydrodynamics {Oldroyd-B} fluid with couple stresses, heat and mass transfer: A comparison among {non-Newtonian} fluid models},
	volume = {9},
	year = {2023}}

@article{Asjad2023a,
	author = {Asjad, Muhammad Imran and Usman, Muhammad and Assiri, Taghreed A. and Ali, Arfan and Tag-ElDin, ElSayed M.},
	doi = {10.3389/fmats.2022.1050767},
	issn = {2296-8016},
	journal = {Frontiers in Materials},
	publisher = {Frontiers Media SA},
	title = {Numerical investigation of fractional {Maxwell} nano-fluids between two coaxial cylinders via the finite difference approach},
	volume = {9},
	year = {2023}}

@article{Liu2024a,
	author = {Liu, Yi and Jiang, Xiaoyun and Jia, Junqing},
	doi = {10.3390/fractalfract8100557},
	issn = {2504-3110},
	journal = {Fractal and Fractional},
	number = {10},
	pages = {557},
	publisher = {MDPI AG},
	title = {Numerical Simulation and Parameter Estimation of the Space-Fractional Magnetohydrodynamic Flow and Heat Transfer Coupled Model},
	volume = {8},
	year = {2024}}

@article{Shakeri2011,
	author = {Shakeri, Fatemeh and Dehghan, Mehdi},
	doi = {10.1016/j.apnum.2010.07.010},
	issn = {0168-9274},
	journal = {Applied Numerical Mathematics},
	number = {1},
	pages = {1--23},
	publisher = {Elsevier BV},
	title = {A finite volume spectral element method for solving magnetohydrodynamic {(MHD)} equations},
	volume = {61},
	year = {2011}}

@article{Bader2025,
	author = {Bader, Shujaut H. and Zhu, Xiaojue},
	doi = {10.1016/j.jcp.2024.113658},
	issn = {0021-9991},
	journal = {Journal of Computational Physics},
	pages = {113658},
	publisher = {Elsevier BV},
	title = {{AFiD-MHD}: A finite difference method for magnetohydrodynamic flows},
	volume = {523},
	year = {2025}}

@article{Fambri2025,
	title = {Structure preserving hybrid Finite Volume Finite Element method for compressible MHD},
	volume = {523},
	ISSN = {0021-9991},
	DOI = {10.1016/j.jcp.2024.113691},
	journal = {Journal of Computational Physics},
	publisher = {Elsevier BV},
	author = {Fambri,  Francesco and Sonnendr\"{u}cker,  Eric},
	year = {2025},
	pages = {113691}
}

@article{Shen2018,
	author = {Shen, Ming and Chen, Shurui and Liu, Fawang},
	doi = {10.1016/j.cjph.2018.04.024},
	issn = {0577-9073},
	journal = {Chinese Journal of Physics},
	number = {3},
	pages = {1199--1211},
	publisher = {Elsevier BV},
	title = {Unsteady {MHD} flow and heat transfer of fractional {Maxwell} viscoelastic nanofluid with {Cattaneo} heat flux and different particle shapes},
	volume = {56},
	year = {2018}}

@article{Zhao2016,
	author = {Zhao, Jinhu and Zheng, Liancun and Zhang, Xinxin and Liu, Fawang},
	doi = {10.1016/j.ijheatmasstransfer.2016.07.057},
	issn = {0017-9310},
	journal = {International Journal of Heat and Mass Transfer},
	pages = {203--210},
	publisher = {Elsevier BV},
	title = {Convection heat and mass transfer of fractional {MHD Maxwell} fluid in a porous medium with {Soret} and {Dufour} effects},
	volume = {103},
	year = {2016}}

@article{Liu2024,
	author = {Liu, Yi and Jiang, Mochen},
	doi = {10.3390/magnetochemistry10100072},
	issn = {2312-7481},
	journal = {Magnetochemistry},
	number = {10},
	pages = {72},
	publisher = {MDPI AG},
	title = {Magnetohydrodynamic Analysis and Fast Calculation for Fractional {Maxwell} Fluid with Adjusted Dynamic Viscosity},
	volume = {10},
	year = {2024}}

@article{He2018,
	title = {A priori estimates and optimal finite element approximation of the {MHD} flow in smooth domains},
	volume = {52},
	ISSN = {1290-3841},
	DOI = {10.1051/m2an/2018006},
	number = {1},
	journal = {ESAIM: Mathematical Modelling and Numerical Analysis},
	publisher = {EDP Sciences},
	author = {He,  Yinnian and Zou,  Jun},
	year = {2018},
	pages = {181–206}
}

@article{Liu2022,
	author = {Huan Liu and Xiangcheng Zheng and Hongfei Fu},
	doi = {10.4208/jcm.2102-m2020-0211},
	journal = {Journal of Computational Mathematics},
	number = {5},
	pages = {814--834},
	publisher = {Global Science Press},
	title = {Analysis of a Multi-Term Variable-Order Time-Fractional Diffusion Equation and Its {Galerkin} Finite Element Approximation},
	volume = {40},
	year = {2022}}

@article{Liao_2019,
	author = {Liao, Hong-lin and McLean, William and Zhang, Jiwei},
	doi = {10.1137/16m1175742},
	issn = {1095-7170},
	journal = {SIAM Journal on Numerical Analysis},
	number = {1},
	pages = {218--237},
	publisher = {Society for Industrial & Applied Mathematics (SIAM)},
	title = {A Discrete {Gr{\"o}nwall} Inequality with Applications to Numerical Schemes for Subdiffusion Problems},
	volume = {57},
	year = {2019}}

@article{Vaz2025,
	title = {On fractional differential equations,  dimensional analysis,  and the double gamma function},
	volume = {113},
	ISSN = {1573-269X},
	DOI = {10.1007/s11071-025-11386-8},
	number = {25},
	journal = {Nonlinear Dynamics},
	publisher = {Springer Science and Business Media LLC},
	author = {Vaz,  J. and de Oliveira,  E. Capelas},
	year = {2025},
	pages = {34305–34320}
}

@article{Borve2006,
	title = {Multidimensional {MHD} Shock Tests of Regularized Smoothed Particle Hydrodynamics},
	volume = {652},
	ISSN = {1538-4357},
	DOI = {10.1086/508454},
	number = {2},
	journal = {The Astrophysical Journal},
	publisher = {American Astronomical Society},
	author = {Borve,  S. and Omang,  M. and Trulsen,  J.},
	year = {2006},
	pages = {1306–1317}
}

@misc{Kayanikhoo2023,
	doi = {10.48550/ARXIV.2312.06675},
	author = {Kayanikhoo,  Fatemeh and Cemeljic,  Miljenko and Wielgus,  Maciek and Kluzniak,  Wlodek},
	title = {Energy dissipation in astrophysical simulations: results of the {Orszag-Tang} test problem},
	publisher = {arXiv},
	year = {2023},
	copyright = {arXiv.org perpetual,  non-exclusive license}
}

@article{Hiptmair2018,
	title = {Splitting-Based Structure Preserving Discretizations for Magnetohydrodynamics},
	volume = {4},
	ISSN = {2426-8399},
	DOI = {10.5802/smai-jcm.34},
	journal = {The SMAI Journal of computational mathematics},
	publisher = {MathDoc/Centre Mersenne},
	author = {Hiptmair,  Ralf and Pagliantini,  Cecilia},
	year = {2018},
	pages = {225–257}
}

@article{Orszag1979,
	title = {Small-scale structure of two-dimensional magnetohydrodynamic turbulence},
	volume = {90},
	ISSN = {1469-7645},
	DOI = {10.1017/s002211207900210x},
	number = {1},
	journal = {Journal of Fluid Mechanics},
	publisher = {Cambridge University Press (CUP)},
	author = {Orszag,  Steven A. and Tang,  Cha-Mei},
	year = {1979},
	pages = {129–143}
}

@article{Jadhav2021,
	title = {Analysis of energy transfer through direct numerical simulations of magnetohydrodynamic {Orszag–Tang} vortex},
	volume = {33},
	ISSN = {1089-7666},
	DOI = {10.1063/5.0051476},
	number = {6},
	journal = {Physics of Fluids},
	publisher = {AIP Publishing},
	author = {Jadhav,  Kiran and Chandy,  Abhilash J.},
	year = {2021}
}

@article{Grauer1998,
	title = {Geometry of singular structures in magnetohydrodynamic flows},
	volume = {5},
	ISSN = {1089-7674},
	DOI = {10.1063/1.872939},
	number = {7},
	journal = {Physics of Plasmas},
	publisher = {AIP Publishing},
	author = {Grauer,  Rainer and Marliani,  Christiane},
	year = {1998},
	pages = {2544–2552}
}

@article{Rehman2023,
  title = {A fractional study with {Newtonian} heating effect on heat absorbing {MHD} radiative flow of rate type fluid with application of novel hybrid fractional derivative operator},
  volume = {30},
  ISSN = {2576-5299},
  DOI = {10.1080/25765299.2023.2250063},
  number = {1},
  journal = {Arab Journal of Basic and Applied Sciences},
  publisher = {Informa UK Limited},
  author = {Rehman,  Aziz Ur and Hua,  Song and Riaz,  Muhammad Bilal and Awrejcewicz,  Jan and Xiange,  Sun},
  year = {2023},
  pages = {482–495}
}

@article{Abbas2024,
  title = {Fractional Analysis of Magnetohydrodynamics {Maxwell} Flow Over an Inclined Plate with the Effect of Thermal Radiation},
  volume = {63},
  ISSN = {1572-9575},
  DOI = {10.1007/s10773-024-05654-3},
  number = {5},
  journal = {International Journal of Theoretical Physics},
  publisher = {Springer Science and Business Media LLC},
  author = {Abbas,  Shajar and Nisa,  Zaib Un and Gilani,  Syeda Farzeen Fatima and Nazar,  Mudassar and Metwally,  Ahmed Sayed M. and Jan,  Ahmed Zubair},
  year = {2024}
}

@article{Khan2023,
  title = {Numerical simulations and modeling of {MHD} boundary layer flow and heat transfer dynamics in {Darcy-Forchheimer} media with distributed fractional-order derivatives},
  volume = {49},
  ISSN = {2214-157X},
  DOI = {10.1016/j.csite.2023.103234},
  journal = {Case Studies in Thermal Engineering},
  publisher = {Elsevier BV},
  author = {Khan,  Mumtaz and Zhang,  Zhengdi and Lu,  Dianchen},
  year = {2023},
  pages = {103234}
}

@article{Qiao2024,
  title = {Rotating {MHD} flow and heat transfer of generalized {Maxwell} fluid through an infinite plate with {Hall} effect},
  volume = {40},
  ISSN = {1614-3116},
  DOI = {10.1007/s10409-023-23274-x},
  number = {5},
  journal = {Acta Mechanica Sinica},
  publisher = {Springer Science and Business Media LLC},
  author = {Qiao,  Yanli and Xu,  Huanying and Qi,  Haitao},
  year = {2024}
}

@article{Li2024,
  title = {Numerical study on radiative {MHD} flow of viscoelastic fluids with distributed-order and variable-order space fractional operators},
  volume = {215},
  ISSN = {0378-4754},
  DOI = {10.1016/j.matcom.2023.07.021},
  journal = {Mathematics and Computers in Simulation},
  publisher = {Elsevier BV},
  author = {Li,  Nan and Wang,  Xiaoping and Xu,  Huanying and Qi,  Haitao},
  year = {2024},
  pages = {291–305}
}

@article{Shiyapov2026,
	title = {High-Order Spectral Scheme with Structure Maintenance and Fast Memory Algorithm for Nonlocal Nonlinear Diffusion Equations},
	volume = {6},
	ISSN = {2673-9909},
	DOI = {10.3390/appliedmath6040054},
	number = {4},
	journal = {AppliedMath},
	publisher = {MDPI AG},
	author = {Shiyapov,  Kadrzhan and Abdiramanov,  Zhanars and Issa,  Zhuldyz and Zhumaseyitova,  Aruzhan},
	year = {2026},
	pages = {54}
}

@article{Wacker2016,
  title = {Nodal-based finite element methods with local projection stabilization for linearized incompressible magnetohydrodynamics},
  volume = {302},
  DOI = {10.1016/j.cma.2016.01.004},
  journal = {Computer Methods in Applied Mechanics and Engineering},
  publisher = {Elsevier BV},
  author = {Wacker,  Benjamin and Arndt,  Daniel and Lube,  Gert},
  year = {2016},
  pages = {170–192}
}

@article{BeirodaVeiga2025,
  title = {Pressure and convection robust finite elements for magnetohydrodynamics},
  volume = {157},
  ISSN = {0945-3245},
  DOI = {10.1007/s00211-025-01476-5},
  number = {4},
  journal = {Numerische Mathematik},
  publisher = {Springer Science and Business Media LLC},
  author = {Beirão da Veiga,  L. and Dassi,  F. and Vacca,  G.},
  year = {2025},
  pages = {1161–1209}
}

@article{John2017,
  title = {On the Divergence Constraint in Mixed Finite Element Methods for Incompressible Flows},
  volume = {59},
  ISSN = {1095-7200},
  DOI = {10.1137/15m1047696},
  number = {3},
  journal = {SIAM Review},
  publisher = {Society for Industrial & Applied Mathematics (SIAM)},
  author = {John,  Volker and Linke,  Alexander and Merdon,  Christian and Neilan,  Michael and Rebholz,  Leo G.},
  year = {2017},
  pages = {492–544}
}

@article{Planas2011,
  title = {Approximation of the inductionless MHD problem using a stabilized finite element method},
  volume = {230},
  ISSN = {0021-9991},
  DOI = {10.1016/j.jcp.2010.12.046},
  number = {8},
  journal = {Journal of Computational Physics},
  publisher = {Elsevier BV},
  author = {Planas,  Ramon and Badia,  Santiago and Codina,  Ramon},
  year = {2011},
  pages = {2977–2996}
}

@article{Altybay2026,
  title = {Numerical identification of a time-dependent coefficient in a time-fractional diffusion equation with integral constraints},
  volume = {77},
  ISSN = {1420-9039},
  DOI = {10.1007/s00033-025-02653-0},
  number = {1},
  journal = {Zeitschrift f\"{u}r angewandte Mathematik und Physik},
  publisher = {Springer Science and Business Media LLC},
  author = {Altybay,  Arshyn},
  year = {2026}
}

@article{Lin2007,
  title = {Finite difference/spectral approximations for the time-fractional diffusion equation},
  volume = {225},
  ISSN = {0021-9991},
  DOI = {10.1016/j.jcp.2007.02.001},
  number = {2},
  journal = {Journal of Computational Physics},
  publisher = {Elsevier BV},
  author = {Lin,  Yumin and Xu,  Chuanju},
  year = {2007},
  pages = {1533–1552}
}

@article{Alikhanov2015,
  title = {A new difference scheme for the time fractional diffusion equation},
  volume = {280},
  ISSN = {0021-9991},
  DOI = {10.1016/j.jcp.2014.09.031},
  journal = {Journal of Computational Physics},
  publisher = {Elsevier BV},
  author = {Alikhanov,  Anatoly A.},
  year = {2015},
  pages = {424–438}
}

\end{document}